\documentclass[12pt,onecolumn, draftclsnofoot]{IEEEtran}
\usepackage{bbm}
\usepackage{mathrsfs}
\usepackage{amsfonts}
\usepackage{epsfig}
\usepackage{psfrag}
\usepackage{graphicx}
\usepackage{epstopdf}
\usepackage{cite}
\usepackage{caption}

\usepackage{array}
\usepackage{stfloats}
\usepackage{textcomp}
\usepackage{verbatim}
\usepackage[caption=false,font=normalsize,labelfont=sf,textfont=sf]{subfig}
\usepackage[justification=centering]{caption}
\usepackage{multirow}
\usepackage{makecell}

\usepackage{amssymb}
\usepackage{amsthm}
\usepackage{algorithm,algorithmic}
\usepackage[tbtags]{amsmath}
\usepackage{url}
\usepackage{bm}

\allowdisplaybreaks

\newtheorem{theorem}{Theorem}
\newtheorem{assumption}{Assumption}

\newtheorem{lemma}{Lemma}

\newtheorem{remark}{Remark}

\renewcommand{\IEEEproof}{\textbf{Proof}.}

\usepackage{color}

\definecolor{gray}{RGB}{128,128,128}

\allowdisplaybreaks

\begin{document}
\title{\textbf{Communication-Efficient Distributed Online Nonconvex Optimization with \\Time-Varying Constraints}}
\author{Kunpeng~Zhang,
        Lei~Xu,
        Xinlei~Yi,
        Guanghui~Wen,
        Ming~Cao,\\
        Karl~H.~Johansson,
        Tianyou~Chai,
        and Tao~Yang
}


\maketitle

\begin{abstract}
This paper considers distributed online nonconvex optimization with time-varying inequality constraints, 
where the nonconvex local loss and convex local constraint functions can vary arbitrarily across iterations. 
For a time-varying directed graph,
we propose two distributed bandit online primal--dual algorithms with compressed communication to efficiently utilize communication resources in the one-point and two-point bandit feedback settings, respectively. 
To measure the performance of the proposed algorithms, we use a network regret metric grounded in the first-order optimality condition associated with the variational inequality.
We show that the compressed algorithm with one-point bandit feedback establishes an $\mathcal{O}( {{T^{\theta _1}}} )$ network regret bound 
and an $\mathcal{O}( {T^{7/4 - {\theta _1}}} )$ network cumulative constraint violation bound, 
where $T$ is the number of iterations and ${\theta _1} \in ( {3/4,5/6} ]$ is a user-defined trade-off parameter. 
When Slater’s condition holds, 
the network cumulative constraint violation bound is reduced to $\mathcal{O}( {T^{5/2 - 2{\theta _1}}} )$. 
In addition, we show that the compressed algorithm with two-point bandit feedback establishes an $\mathcal{O}( {{T^{\max \{ {1 - {\theta _1},{\theta _1}} \}}}} )$ network regret 
and an $\mathcal{O}( {T^{1 - {\theta _1}/2}} )$ network cumulative constraint violation bounds, where ${\theta _1} \in ( {0,1} )$. 
Moreover, the network cumulative constraint violation bound is reduced to $\mathcal{O}( {T^{1 - {\theta _1}}} )$ under Slater’s condition.
Finally, a simulation example is presented to validate the theoretical results.
\end{abstract}

\begin{IEEEkeywords}
Compressed communication, cumulative constraint violation, distributed online optimization, nonconvex optimization, Slater’s condition, time-varying constraints.
\end{IEEEkeywords}


\section{Introduction}
Distributed online convex optimization is a sequential decision making problem, 
which can be understood as a structured repeated game with $T$ iterations between a network of agents and an adversary. 
Specifically, at each iteration $t$, each agent $i$ selects a decision ${x_{i,t}} \in \mathbb{X}$, where $\mathbb{X} \subseteq {\mathbb{R}^p}$ is a known convex set. 
Upon selection, the local loss function ${f_{i,t}}$ is privately revealed to agent $i$ by the adversary (full-information feedback setting). 
The goal of agents is to collaboratively minimize the network-wide accumulated loss, and the corresponding performance metric is network regret
\begin{flalign}
\nonumber
\frac{1}{n}\sum\limits_{i = 1}^n {\Big( {\sum\limits_{t = 1}^T {{f_t}( {{x_{i,t}}} )}  - \mathop {\min }\limits_{x \in \mathbb{X}} \sum\limits_{t = 1}^T {{f_t}( {{x}} )} } \Big)},
\end{flalign}
where ${f_t}( x ): = \frac{1}{n}\sum\nolimits_{j = 1}^n {{f_{j,t}}( x )} $ is the global loss function of the network at iteration $t$. To solve the distributed online convex optimization problem, various projection-based distributed online algorithms have been proposed, see, e.g., \cite{Yan2012, Tsianos2012, MateosNunez2014, Koppel2015, Hosseini2016, Shahrampour2017, Yuan2021, Li2022} and recent survey paper \cite{Li2023}.

Communication bottlenecks pose a significant challenge for distributed optimization due to limited communication bandwidth and power. 
To improve communication efficiency,
the effective compressed communication strategies have been extensively studied has been extensively studied, see, e.g., \cite{Zhu2016, Tang2018, Liao2022, Yi2022}, and recent survey paper \cite{Cao2023}.
Recently, there are few studies on distributed online convex optimization with compressed communication.
For example, the authors of \cite{Li2021b} propose the ECD-AMSGrad algorithm by integrating the AMSGrad algorithm in \cite{Reddi2019} with extrapolation compression strategy proposed in \cite{Tang2018}.
However, only empirical results are presented in \cite{Li2021b}.
The authors of \cite{Cao2023a, Tu2022} propose distributed online gradient descent algorithms with compressed communication by introducing auxiliary variables to estimate the neighbors' decisions at each iteration, and establish $\mathcal{O}( {\sqrt T } )$ network regret bounds for convex local loss functions.
Moreover, the authors of \cite{Tu2022} also propose two compressed distributed online gradient descent algorithms with bandit feedback 
(in the bandit feedback setting only the values of the local loss function at some points are privately revealed to the agents at each iteration). 
The authors of \cite{Ge2023} propose a distributed online gradient descent algorithm with one-point bandit feedback by using the compression strategy similar to that used in \cite{Tu2022}, and establish 
an~$\mathcal{O}( {T^{3/4}} )$ regret bound for convex local loss functions.
Note that all the aforementioned studies on distributed online convex optimization with compressed communication consider the fixed communication topology. 
However, \emph{when the communication topology is time-varying, it remains unclear how to design provably communication-efficient distributed online algorithms with sublinear regret} (Question~1). 

In practical applications, inequality constraints often exist. In distributed online convex optimization with inequality constraints, the feasible set becomes $\mathcal{X}$ rather than $\mathbb{X}$, which is typically characterized by
\begin{flalign}
\nonumber
\mathcal{X} = \{ {x:g( x ) \le {\mathbf{0}_m},x \in \mathbb{X}} \},
\end{flalign}
where $g:{\mathbb{R}^p} \to {\mathbb{R}^m}$ is a static convex constraint function, and $m$ is a positive integer. 
However, projecting onto $\mathcal{X}$ generally incurs a heavy computational burden due to inequality constraints. 
To tackle this challenge, the authors of \cite{Yuan2017} consider the idea of long-term constraints proposed in \cite{Mahdavi2012}, 
where decisions are chosen from the simple set $\mathbb{X}$ and static inequality constraints should be satisfied in the long term on average.
To calculate the accumulated constraint violation,  constraint violation is used as the performance metric.
Different from \cite{Yuan2017}, the authors of \cite{Yuan2021b} use the cumulative constraint violation metric proposed in \cite{Yuan2018}.
The cumulative constraint violation captures all constraints that are not satisfied, and thus cumulative constraint violation is stricter than constraint violation.
The authors of \cite{Yi2023} extend distributed online convex optimization with long-term constraints into the time-varying constraints setting where the information of local constraint functions 
is privately revealed along with the information of local loss functions by the adversary in each iteration. 
However, the distributed online algorithms proposed in \cite{Yuan2021b, Yi2023} cannot achieve reduced network cumulative constraint violation under Slater’s condition, as discussed in \cite{Yi2024, Zhang2024b}. 
Recently, the authors of \cite{Yi2024, Zhang2024b} propose novel distributed online algorithms for the full-information and bandit feedback settings, respectively, and establish reduced network cumulative constraint violation bounds under Slater’s condition.
Note that all the aforementioned studies on distributed online convex optimization with inequality constraints do not consider compressed communication. 
\emph{Under compressed communication, it remains unclear whether sublinear network regret and (cumulative) constraint violation can be established for distributed online optimization with inequality constraints} (Question~2).

Different from convex optimization as discussed above,
in nonconvex optimization, finding a globally optimal decision is often an NP-hard problem \cite{Kohler2017}.  
Consequently, the studies in the literature typically focus on establishing convergence to a point satisfying the first-order optimality condition , see, e.g., \cite{Scutari2019, Cannelli2020, Kungurtsev2023}. 
Recently, to measure the performance of the algorithms for distributed online nonconvex optimization, 
the authors of \cite{Lu2021a} propose a regret metric by utilizing the first-order optimality condition associated with the variational inequality. 
As a result, an $\mathcal{O}( {\sqrt T} )$ regret bound is established for nonconvex local loss functions. 
The authors of \cite{Suo2025} consider inequality constraints for distributed online nonconvex optimization. 
By utilizing the regret metric proposed in \cite{Lu2021a} and the constraint violation metric, they establish an $\mathcal{O}( {{T^{1/2 + c}}} )$ regret bound 
and an $\mathcal{O}( {{T^{3/4}}} )$ constraint violation bound for nonconvex local loss and convex local constraint functions under Slater’s condition, where $c \in ( {0,1/2} )$. 
Note that \cite{Suo2025} considers static inequality constraints and the constraint violation metric. 
\emph{It remains unclear whether sublinear network regret and (cumulative) constraint violation can be achieved for distributed online nonconvex optimization with time-varying constraints} (Question~3). 
In addition, unlike convex optimization, in nonconvex optimization Slater’s condition alone does not generally guarantee strong duality because the standard Lagrange dual often has a nonzero duality gap in nonconvex problems \cite{Bertsekas1997}. 
\emph{It also remains unclear whether reduced network (cumulative) constraint violation can be achieved for distributed online nonconvex optimization with time-varying constraints under Slater’s condition} (Question~4). 

In this paper, we consider the distributed online nonconvex optimization problem with time-varying constraints in the bandit feedback setting.
The main contributions are outlined as follows.
\begin{itemize}
\item[$\bullet$] 
We propose two distributed bandit online primal--dual algorithms with compressed communication by using the general compressors with globally bounded absolute compression error, 
for the one-point and two-point bandit feedback settings, respectively.
Different from \cite{Li2021b, Cao2023a, Tu2022, Ge2023} that consider a fixed and undirected communication topology, 
we consider a uniformly jointly strongly connected time-varying directed communication topology.
Moreover, in \cite{Li2021b, Cao2023a, Tu2022, Ge2023} the local loss function is convex and inequality constraints are not considered while we consider nonconvex local loss functions and time-varying inequality constraints.
\item[$\bullet$]
We show in Theorem~1 that the compressed distributed online primal--dual algorithm with one-point bandit feedback establishes 
sublinear network regret and cumulative constraint violation bounds, thereby answering {Questions 1--3}.  
Moreover, the network cumulative constraint violation bound is further reduced under Slater’s condition, thereby answering {Question 4}.
To the best of our knowledge, this paper is the first to establish sublinear network regret and cumulative constraint violation bounds 
for distributed online nonconvex optimization with time-varying constraints in the one-point bandit feedback setting, both without and with Slater’s condition.
Notably, such results have not even been established for the convex local loss functions.
\item[$\bullet$]
We show in Theorem~2 that the compressed distributed online primal--dual algorithm with two-point bandit feedback establishes 
an $\mathcal{O}( {{T^{\max \{ {1 - {\theta _1},{\theta _1}} \}}}} )$ network regret bound and an $\mathcal{O}( {T^{1 - {\theta _1}/2}} )$ network cumulative constraint violation bound. 
That answers {Questions 1--3}.
These bounds are the same as the results established in \cite{Yi2023, Yi2024, Zhang2024b} where the local loss function is convex and compressed communication is not considered. 
When Slater’s condition holds, we further show in Theorem~2 that the algorithm establishes a reduced $\mathcal{O}( {T^{1 - {\theta _1}}} )$ network cumulative constraint violation bound.
That answers {Question 4}. 
This bound is the same as the results established in \cite{Yi2024, Zhang2024b}.
\end{itemize}

\renewcommand{\arraystretch}{0.95}
\begin{table*}\LARGE
\centering
\caption{Comparison of this paper to related studies on distributed online optimization.}
\resizebox{\linewidth}{!}{
\begin{tabular}{c|c|c|c|c|c|c|c|c|c}
\Xcline{1-10}{1pt}
\multicolumn{2}{c|}{\multirow{3}{*}{Reference}} & {\multirow{3}{*}{\makecell{Loss \\ functions}}} & \multirow{3}{*}{\makecell{Constraint \\functions}} & \multirow{3}{*}{\makecell{Slater's \\condition}} & \multirow{3}{*}{\makecell{Compressed \\ communication}} & \multirow{3}{*}{\makecell{Information \\feedback}} & \multirow{3}{*}{Regret} & \multirow{3}{*}{\makecell{Constraint \\violation}} & \multirow{3}{*}{\makecell{Cumulative \\constraint \\violation}}\\
\multicolumn{2}{c|}{} & & & & & & & \\
\multicolumn{2}{c|}{} & & & & & & & \\
\cline{1-10}
\multicolumn{2}{c|}{\multirow{2}{*}{\cite{Cao2023a}}} & {\multirow{2}{*}{Convex}} & \multirow{2}{*}{\makecell{$\times$}} & \multirow{2}{*}{\makecell{$\times$}} & \multirow{2}{*}{\makecell{$\checkmark$}} & \multirow{2}{*}{\makecell{Full-information}} & \multirow{2}{*}{$\mathcal{O}( {\sqrt T } )$} & \multicolumn{2}{c}{\multirow{2}{*}{Not applicable}}\\
\multicolumn{2}{c|}{} & & & & & &\\
\cline{1-10}
\multicolumn{2}{c|}{\multirow{6}{*}{\cite{Tu2022}}} & {\multirow{6}{*}{Convex}} & \multirow{6}{*}{\makecell{$\times$}} & \multirow{6}{*}{\makecell{$\times$}} & \multirow{6}{*}{\makecell{$\checkmark$}} & \multirow{2}{*}{\makecell{Full-information}} & \multirow{2}{*}{$\mathcal{O}( {\sqrt T } )$} & \multicolumn{2}{c}{\multirow{6}{*}{Not applicable}}\\
\multicolumn{2}{c|}{} & & & & & &\\
\cline{7-8}
\multicolumn{2}{c|}{} & & & & & \multirow{2}{*}{\makecell{One-point bandit}} & \multirow{2}{*}{$\mathcal{O}( {{T^{3/4}} } )$} \\
\multicolumn{2}{c|}{} & & & & & &\\
\cline{7-8}
\multicolumn{2}{c|}{} & & & & & \multirow{2}{*}{\makecell{Two-point bandit}} & \multirow{2}{*}{$\mathcal{O}( {{\sqrt T} } )$} \\
\multicolumn{2}{c|}{} & & & & & &\\
\cline{1-10}
\multicolumn{2}{c|}{\multirow{2}{*}{\cite{Ge2023}}} & {\multirow{2}{*}{Convex}} & \multirow{2}{*}{\makecell{$\times$}} & \multirow{2}{*}{\makecell{$\times$}} & \multirow{2}{*}{\makecell{$\checkmark$}} & \multirow{2}{*}{\makecell{One-point bandit}} & \multirow{2}{*}{$\mathcal{O}( {T^{3/4}} )$} & \multicolumn{2}{c}{\multirow{2}{*}{Not applicable}}\\
\multicolumn{2}{c|}{} & & & & & &\\
\cline{1-10}
\multicolumn{2}{c|}{\multirow{3}{*}{\cite{Yuan2017}}} & {\multirow{3}{*}{Convex}} & \multirow{3}{*}{\makecell{Convex and static}} & \multirow{3}{*}{\makecell{$\times$}} & \multirow{3}{*}{\makecell{$\times$}} & \multirow{3}{*}{\makecell{Full-information \\of $f_{i,t}$ and $g$ }} & \multirow{3}{*}{\makecell{$\mathcal{O}( {{T^{1/2 + c}}} )$, \\ $c \in ( {0,1/2} )$}} & \multirow{3}{*}{$\mathcal{O}( {{T^{1 - c/2}}} )$} & \multirow{3}{*}{Not applicable}\\
\multicolumn{2}{c|}{} & & & & & & &\\
\multicolumn{2}{c|}{} & & & & & & &\\
\cline{1-10}
\multicolumn{2}{c|}{\multirow{6}{*}{\cite{Yuan2021b}}} & {\multirow{6}{*}{Quadratic}} & \multirow{6}{*}{\makecell{Linear and static}} & \multirow{6}{*}{\makecell{$\times$}} & \multirow{6}{*}{\makecell{$\times$}} & \multirow{3}{*}{\makecell{Full-information \\of $f_{i,t}$ and $g$ }} & \multirow{3}{*}{\makecell{$\mathcal{O}\big( {{T^{\max \{ {c, 1 - c} \}}}} \big)$, \\$c \in ( {0,1} )$}} & \multicolumn{2}{c}{\multirow{3}{*}{$\mathcal{O}( {{T^{1 - c/2}}} )$}}\\
\multicolumn{2}{c|}{} & & & & & &\\
\multicolumn{2}{c|}{} & & & & & &\\
\cline{7-10}
\multicolumn{2}{c|}{} & & & & & \multirow{3}{*}{\makecell{Two-point bandit of $f_{i,t}$ and \\ full-information of $g$ }} & \multirow{3}{*}{\makecell{$\mathcal{O}\big( {{T^{\max \{ {c, 1 - c} \}}}} \big)$, \\$c \in ( {0,1} )$}} & \multicolumn{2}{c}{\multirow{3}{*}{$\mathcal{O}( {{T^{1 - c/2}}} )$}}\\
\multicolumn{2}{c|}{} & & & & & &\\
\multicolumn{2}{c|}{} & & & & & &\\
\cline{1-10}
\multicolumn{2}{c|}{\multirow{6}{*}{\cite{Yi2023}}} & {\multirow{6}{*}{Convex}} & \multirow{6}{*}{\makecell{Convex and time-varying}} & \multirow{6}{*}{\makecell{$\times$}} & \multirow{6}{*}{\makecell{$\times$}} & \multirow{3}{*}{\makecell{Full-information \\of $f_{i,t}$ and $g_{i,t}$ }} & \multirow{3}{*}{\makecell{$\mathcal{O}\big( {{T^{\max \{ {c, 1 - c} \}}}} \big)$, \\$c \in ( {0,1} )$}} & \multicolumn{2}{c}{\multirow{3}{*}{$\mathcal{O}( {{T^{1 - c/2}}} )$}}\\
\multicolumn{2}{c|}{} & & & & & &\\
\multicolumn{2}{c|}{} & & & & & &\\
\cline{7-10}
\multicolumn{2}{c|}{} & & & & & \multirow{3}{*}{\makecell{Two-point bandit \\ 
of $f_{i,t}$ and $g_{i,t}$ }} & \multirow{3}{*}{\makecell{$\mathcal{O}\big( {{T^{\max \{ {c, 1 - c} \}}}} \big)$, \\$c \in ( {0,1} )$}} & \multicolumn{2}{c}{\multirow{3}{*}{$\mathcal{O}( {{T^{1 - c/2}}} )$}}\\
\multicolumn{2}{c|}{} & & & & & &\\
\multicolumn{2}{c|}{} & & & & & &\\
\cline{1-10}
\multicolumn{2}{c|}{\multirow{4}{*}{\cite{Yi2024}}} & {\multirow{4}{*}{Convex}} & \multirow{4}{*}{\makecell{Convex and time-varying}} & \multirow{2}{*}{\makecell{$\times$}} & \multirow{4}{*}{\makecell{$\times$}} & \multirow{4}{*}{\makecell{Full-information \\of $f_{i,t}$ and $g_{i,t}$}} & \multirow{4}{*}{\makecell{$\mathcal{O}\big( {{T^{\max \{ {c, 1 - c} \}}}} \big)$, \\$c \in ( {0,1} )$}} & \multicolumn{2}{c}{\multirow{2}{*}{$\mathcal{O}( {{T^{1 - c/2}}} )$}}\\
\multicolumn{2}{c|}{} & & & & & &\\
\cline{5-5}
\cline{9-10}
\multicolumn{2}{c|}{} & & & \multirow{2}{*}{\makecell{$\checkmark$}} & & & & \multicolumn{2}{c}{\multirow{2}{*}{$\mathcal{O}( {{T^{1 - c}}} )$}}\\
\multicolumn{2}{c|}{} & & & & & &\\
\cline{1-10}
\multicolumn{2}{c|}{\multirow{4}{*}{\cite{Zhang2024b}}} & {\multirow{4}{*}{Convex}} & \multirow{4}{*}{\makecell{Convex and time-varying}} & \multirow{2}{*}{\makecell{$\times$}} & \multirow{4}{*}{\makecell{$\times$}} & \multirow{4}{*}{\makecell{Two-point bandit \\of $f_{i,t}$ and $g_{i,t}$}} & \multirow{4}{*}{\makecell{$\mathcal{O}\big( {{T^{\max \{ {c, 1 - c} \}}}} \big)$, \\$c \in ( {0,1} )$}} & \multicolumn{2}{c}{\multirow{2}{*}{$\mathcal{O}( {{T^{1 - c/2}}} )$}}\\
\multicolumn{2}{c|}{} & & & & & &\\
\cline{5-5}
\cline{9-10}
\multicolumn{2}{c|}{} & & & \multirow{2}{*}{\makecell{$\checkmark$}} & & & & \multicolumn{2}{c}{\multirow{2}{*}{$\mathcal{O}( {{T^{1 - c}}} )$}}\\
\multicolumn{2}{c|}{} & & & & & &\\
\cline{1-10}
\multicolumn{2}{c|}{\multirow{2}{*}{\cite{Lu2021a}}} & {\multirow{2}{*}{Nonconvex}} & \multirow{2}{*}{\makecell{$\times$}} & \multirow{2}{*}{\makecell{$\times$}} & \multirow{2}{*}{\makecell{$\times$}} & \multirow{2}{*}{\makecell{Full-information}} & \multirow{2}{*}{$\mathcal{O}( {\sqrt T } )$} & \multicolumn{2}{c}{\multirow{2}{*}{Not applicable}}\\
\multicolumn{2}{c|}{} & & & & & &\\
\cline{1-10}
\multicolumn{2}{c|}{\multirow{3}{*}{\cite{Suo2025}}} & {\multirow{3}{*}{Nonconvex}} & \multirow{3}{*}{\makecell{Convex and static}} & \multirow{3}{*}{\makecell{$\checkmark$}} & \multirow{3}{*}{\makecell{$\times$}} & \multirow{3}{*}{\makecell{Full-information \\of $f_{i,t}$ and $g_i$ }} & \multirow{3}{*}{\makecell{$\mathcal{O}( {{T^{1/2 + c}}} )$, \\ $c \in ( {0,1/2} )$}} & \multirow{3}{*}{$\mathcal{O}( {{T^{3/4}}} )$} & \multirow{3}{*}{Not applicable}\\
\multicolumn{2}{c|}{} & & & & & & &\\
\multicolumn{2}{c|}{} & & & & & & &\\
\cline{1-10}
\multicolumn{2}{c|}{\multirow{8}{*}{This paper}} & {\multirow{8}{*}{Nonconvex}} & \multirow{8}{*}{\makecell{Convex and time-varying}} & \multirow{2}{*}{\makecell{$\times$}} & \multirow{8}{*}{\makecell{$\checkmark$}} & \multirow{4}{*}{\makecell{One-point bandit \\of $f_{i,t}$ and $g_{i,t}$}} & \multirow{4}{*}{\makecell{$\mathcal{O}( {{T^{\theta _1}}} )$, \\ ${\theta _1} \in ( {3/4,5/6} ]$}} & \multicolumn{2}{c}{\multirow{2}{*}{$\mathcal{O}( {T^{7/4 - {\theta _1}}} )$}}\\
\multicolumn{2}{c|}{} & & & & & &\\
\cline{5-5}
\cline{9-10}
\multicolumn{2}{c|}{} & & & \multirow{2}{*}{\makecell{$\checkmark$}} & & & & \multicolumn{2}{c}{\multirow{2}{*}{$\mathcal{O}( {T^{5/2 - 2{\theta _1}}} )$}}\\
\multicolumn{2}{c|}{} & & & & & &\\
\cline{5-5}
\cline{7-7}
\cline{8-10}
\multicolumn{2}{c|}{} & & & \multirow{2}{*}{\makecell{$\times$}} & & \multirow{4}{*}{\makecell{Two-point bandit \\of $f_{i,t}$ and $g_{i,t}$}} & \multirow{4}{*}{\makecell{$\mathcal{O}\big( {{T^{\max \{ {{\theta _1}, 1 - {\theta _1}} \}}}} \big)$, \\$\theta _1 \in ( {0,1} )$}} & \multicolumn{2}{c}{\multirow{2}{*}{$\mathcal{O}( {{T^{1 - {\theta _1}/2}}} )$}}\\
\multicolumn{2}{c|}{} & & & & & &\\
\cline{5-5}
\cline{9-10}
\multicolumn{2}{c|}{} & & & \multirow{2}{*}{\makecell{$\checkmark$}} & & & & \multicolumn{2}{c}{\multirow{2}{*}{$\mathcal{O}( {{T^{1 - {\theta _1}}}} )$}}\\
\multicolumn{2}{c|}{} & & & & & &\\
\Xcline{1-10}{1pt}
\end{tabular}
}
\end{table*}
\renewcommand{\arraystretch}{1.0}

The detailed comparison between this paper and related studies is presented in TABLE~I. 
For clarity, we present a special case of the results in Theorem~1.

The remainder of this paper is organized as follows.
Section~II presents the problem formulation.
Sections~III and~IV propose the compressed distributed online primal--dual algorithm with one-point and two-point bandit feedback, respectively,
and analyze their network regret and cumulative constraint violation bounds without and with Slater’s condition, respectively.
Section~V provides a simulation example to verify the theoretical results.
Finally, Section~VI concludes this paper.
All proofs are given in Appendix.

\noindent\textbf{Notations:} All inequalities and equalities throughout this paper are understood componentwise. 
${\mathbb{N}_ + }$, $\mathbb{R}$, ${\mathbb{R}^p}$ and $\mathbb{R}_ + ^p$ denote the sets of all positive integers, real numbers, $p$-dimensional and nonnegative vectors, respectively.
$\lfloor  \cdot  \rfloor $ denotes the floor operation.
$\mathbf{E}[  \cdot  ]$ denotes the total expectation with respect to all underlying randomness.
${p_t} = \mathcal{O}( {{q_t}} )$ means that there exists a constant $a > 0$ such that ${p_t} \le a{q_t}$ for all $t$.
Given $m$ and $n \in {\mathbb{N}_ + }$, $[ m ]$ denotes the set $\{ {1, \cdot  \cdot  \cdot ,m} \}$, and $[m, n]$ denotes the set $\{ {m, \cdot  \cdot  \cdot ,n} \}$ for $m < n$. 
Given vectors $x$ and $y$, ${x^T}$ denotes the transpose of the vector $x$, and $\langle {x,y} \rangle $ and $x \otimes y$ denote the standard inner and Kronecker product, respectively. 
${\mathbf{0}_p}$ and ${\mathbf{1}_p}$ denote the $p$-dimensional column vector whose components are all $0$ and $1$, respectively. 
$\mathrm{col}( {q_1}, \cdot  \cdot  \cdot ,{q_n} )$ denotes the concatenated column vector of ${q_i} \in {\mathbb{R}^{{m_i}}}$ for $i \in [ n ]$. 
${\mathbb{B}^p}$ and ${\mathbb{S}^p}$ denote the unit ball and sphere centered around the origin in ${\mathbb{R}^p}$ under Euclidean norm, respectively. 
For a set $\mathbb{K} \in {\mathbb{R}^p}$ and a vector $ x \in {\mathbb{R}^p}$, ${\mathcal{P}_{\mathbb{K}}}(  x  )$ denotes the projection of the vector $x$ onto the set $\mathbb{K}$, 
i.e., ${\mathcal{P}_{\mathbb{K}}}( x ) = \arg {\min _{y \in {\mathbb{K}}}}{\| {x - y} \|^2}$, and $[  x  ]_+$ denotes ${\mathcal{P}_{\mathbb{R}_ + ^p}}( x )$. 
For a function $f$ and a vector $ x $, $\nabla f( x )$ denotes the gradient of $f$ at $x$.

\section{Problem Formulation}
Consider the distributed online nonconvex optimization problem with time-varying constraints.
At iteration $t$, a network of $n$ agents is modeled by a time-varying directed graph ${\mathcal{G}_t} = ( {\mathcal{V},{\mathcal{E}_t}} )$ with the agent set $\mathcal{V} = [ n ]$ and the edge set ${\mathcal{E}_t} \subseteq \mathcal{V} \times \mathcal{V}$. $( {j,i} ) \in {\mathcal{E}_t}$ indicates that agent $i$ can receive information from agent $j$.
The sets of in- and out-neighbors of agent~$i$ are $\mathcal{N}_i^{\text{in}}( {{\mathcal{G}_t}} ) = \{ {j \in [ n ]|( {j,i} ) \in {\mathcal{E}_t}} \}$ and $\mathcal{N}_i^{\text{out}}( {{\mathcal{G}_t}} ) = \{ {j \in [ n ]|( {i,j} ) \in {\mathcal{E}_t}} \}$, respectively.
Let $\{ {{f_{i,t}}:\mathbb{X} \to \mathbb{R}} \}$ and $\{ {{g_{i,t}}:\mathbb{X} \to {\mathbb{R}^{{m_i}}}} \}$ be sequences of nonconvex local loss and convex local constraint functions, respectively, where ${m_i}$ is a positive integer and ${g_{i,t}} \le {\mathbf{0}_{{m_i}}}$ is the local constraint.
Each agent $i$ selects a local decisions $\{ {{x_{i,t}} \in \mathbb{X}} \}$ without prior access to $\{ {{f_{i,t}}} \}$ and $\{ {{g_{i,t}}} \}$. 
After the selection, the agent receives bandit feedback for the nonconvex local loss function $\{ {{f_{i,t}}} \}$ and convex local constraint function $\{ {{g_{i,t}}} \}$, which is held privately by the agent.
The goal of the agent is to choose the decision sequence $\{ {{x_{i,t}}} \}$ for $i \in [n]$ and $t \in [T]$ such that both network regret
\begin{flalign}
{\rm{Net}\mbox{-}\rm{Reg}}( T ) &:= \frac{1}{n}\sum\limits_{i = 1}^n {\Big( {\sum\limits_{t = 1}^T {\langle {\nabla {f_t}( {{x_{i,t}}} ),{x_{i,t}}} \rangle }  - \mathop {\inf }\limits_{x \in \mathcal{X}_T} \Big\langle {\sum\limits_{t = 1}^T {\nabla {f_t}( {{x_{i,t}}} )} ,x} \Big\rangle } \Big)}, \label{regret-eq1}
\end{flalign}
and network cumulative constraint violation
\begin{flalign}
{\rm{Net}\mbox{-}\rm{CCV}}( T ) &:= \frac{1}{n}\sum\limits_{i = 1}^n {\sum\limits_{t = 1}^T {\| {{{[ {{g_t}( {{x_{i,t}}} )} ]}_ + }} \|} }, \label{CCV-eq2}
\end{flalign}
increase sublinearly, where ${f_t}( x ) = \frac{1}{n}\sum\nolimits_{j = 1}^n {{f_{j,t}}( x )} $ is the global loss function of the network at iteration $t$, 
${\mathcal{X}_T} = \{ { x :x \in \mathbb{X}, {g_t}( x ) \le {\mathbf{0}_m},\forall t \in [ T ]} \}$ is the feasible set, 
and ${g_t}( x ) = {\rm{col}}\big( {{g_{1,t}}( x ), \cdot  \cdot  \cdot ,{g_{n,t}}( x )} \big) \in {\mathbb{R}^m}$ with $m = \sum\nolimits_{i = 1}^n {{m_i}} $ 
is the global constraint function of the network at iteration $t$. Similar to existing literature on distributed online convex optimization with time-varying constraints, 
see, e.g., \cite{Yi2023, Yi2024, Zhang2024b}, we assume that the feasible set ${\mathcal{X}_T}$ is nonempty for all $T \in {\mathbb{N}_ + }$.


The following commonly used assumptions are made throughout this paper.
\begin{assumption}
The set $\mathbb{X}$ is closed.
Moreover, the convex set $\mathbb{X}$ contains the ball of radius $r( \mathbb{X} )$ and is contained in the ball of radius $R( \mathbb{X} )$, i.e.,
\begin{flalign}
r( \mathbb{X} ){\mathbb{B}^p} \subseteq \mathbb{X} \subseteq R( \mathbb{X} ){\mathbb{B}^p}. \label{ass1-eq1}
\end{flalign}
\end{assumption}

\begin{assumption}
For all $i \in [n]$, $t \in {\mathbb{N}_ + }$, the gradients $\nabla {f_{i,t}}( x )$ and $\nabla {g_{i,t}}( x )$ exist. Moreover, there exist constants ${G_1}$ and ${G_2}$ such that
\begin{subequations}
\begin{flalign}
\| {\nabla {f_{i,t}}( x )} \| &\le {G_1}, x \in \mathbb{X}, \label{ass4-eq1a}\\
\| {\nabla {g_{i,t}}( x )} \| &\le {G_2}, x \in \mathbb{X}. \label{ass4-eq1b}
\end{flalign}
\end{subequations}
\end{assumption}

For all $i \in [n]$, $t \in {\mathbb{N}_ + }$, let ${\phi _{i,t}}( u ) = {f_{i,t}}\big( {x + u( {y - x} )} \big)$, $x,y \in \mathbb{X}$.
Then, we have
\begin{flalign}
\nonumber
&\;\;\;\;\;| {{f_{i,t}}( y ) - {f_{i,t}}( x )} | \le | {{\phi _{i,t}}( 1 ) - {\phi _{i,t}}( 0 )} | \\
\nonumber
& = \Big| {\int_0^1 {\nabla {f_{i,t}}( {x + u( {y - x} )} )( {y - x} )du} } \Big| \\
& \le \int_0^1 {\| {\nabla {f_{i,t}}( {x + u( {y - x} )} )} \|\| {y - x} \|du} \le {G_1}\| {x - y} \|,
\label{LP}
\end{flalign}
where the first inequality holds due to the Cauchy--Schwarz inequality, and the first inequality holds due to \eqref{ass4-eq1a}.
In the same way, from \eqref{ass4-eq1b}, for all $i \in [n]$, $t \in {\mathbb{N}_ + }$, we have
\begin{flalign}
\| {{g_{i,t}}( x ) - {g_{i,t}}( y )} \| &\le {G_2}\| {x - y} \|, x,y \in \mathbb{X}. \label{ass4-eq2}
\end{flalign}
For all $i \in [n]$, $t \in {\mathbb{N}_ + }$, and $x \in \mathbb{X}$, from \eqref{ass1-eq1}, \eqref{LP}--\eqref{ass4-eq2}, we know that there exist a point ${x_0} \in \mathbb{X}$ and constants ${F_1}$ and  ${F_2}$ such that
\begin{flalign}
| {{f_{i,t}}( x )} | & \le | {{f_{i,t}}( x ) - {f_{i,t}}( {{x_0}} )} | + | {{f_{i,t}}( {{x_0}} )} | \le {F_1}, \label{ass2-eq1a} \\
\| {{g_{i,t}}( x )} \| & \le \| {{g_{i,t}}( x ) - {g_{i,t}}( {{x_0}} )} \| + \| {{g_{i,t}}( {{x_0}} )} \| \le {F_2}. \label{ass2-eq1b}
\end{flalign}

\begin{assumption}
For all $i \in [n]$, $t \in {\mathbb{N}_ + }$, there exists a constant ${L}$ such that
\begin{flalign}
\| {\nabla {f_{i,t}}( x ) - \nabla {f_{i,t}}( y )} \| \le {L}\| {x - y} \|, x,y \in \mathbb{X}. \label{ass5-eq1}
\end{flalign}
\end{assumption}

\begin{remark}\label{rem1}
Assumptions~1--2 are commonly used for distributed online convex optimization with inequality constraints \cite{Yuan2017, Yuan2021b, Yi2023, Yi2024, Zhang2024b}. 
Assumption~2 implies that the local loss function $ {{f_{i,t}}}$ and the local constraint function $ {{g_{i,t}}}$ are Lipschitz continuous on $\mathbb{X}$, as stated in \eqref{LP}--\eqref{ass4-eq2}.
Under Assumptions~1--2, we show that $f_{i,t}$ and $g_{i,t}$ are uniformly bounded, as stated in \eqref{ass2-eq1a}--\eqref{ass2-eq1b}. In contrast, such conditions are imposed as an additional assumption in \cite{Tu2022, Ge2023, Yi2023, Zhang2024b}.
Assumption~3 is commonly used for distributed online nonconvex optimization, e.g., \cite{Lu2021a, Suo2025}.
\end{remark}

Different from the studies on distributed online convex optimization with compressed communication \cite{Li2021b, Cao2023a, Tu2022, Ge2023} that consider a fixed and undirected graph, we consider a time-varying directed graph.
The following assumption on the graph is made, which is also used in \cite{Yuan2017, Yi2023, Yi2024, Zhang2024b}.
\begin{assumption}
For all $t \in {\mathbb{N}_ + }$, the time-varying directed graph $\mathcal{G}_t$ satisfies the following conditions:

\noindent (i) There exists a constant $w  \in ( {0,1} )$ such that ${[ {{W_t}} ]_{ij}} \ge w$ if $( {j,i} ) \in {\mathcal{E}_t}$ or $i = j$, and ${[ {{W_t}} ]_{ij}} = 0$ otherwise.

\noindent (ii) The mixing matrix ${W_t}$ is doubly stochastic, i.e., ${\sum\nolimits_{i = 1}^n {[ {{W_t}} ]} _{ij}} = {\sum\nolimits_{j = 1}^n {[ {{W_t}} ]} _{ij}} = 1$, $\forall i,j \in [ n ]$.

\noindent (iii) There exists an integer $B > 0$ such that the time-varying directed graph $( {\mathcal{V}, \cup _{l = 0}^{B - 1}{\mathcal{E} _{t + l}}} )$ is strongly connected.
\end{assumption}

We consider the scenario where the communication between agents is compressed by a class of general compressors with globally bounded absolute compression error, as given in the following.
\begin{assumption}
The compressor $\mathcal{C}:{\mathbb{Y}} \to {\mathbb{Y}}$ with $\mathbb{Y} \subseteq {\mathbb{R}^p}$ satisfies
\begin{flalign}
{\mathbf{E}_\mathcal{C}}[ {\| {\mathcal{C}( x ) - x} \|^2} ] \le C, \forall x \in {\mathbb{Y}}, \label{compre-eq1}
\end{flalign}
for constant $C \ge 0$. Here ${\mathbf{E}_\mathcal{C}}$ denotes the expectation over the internal randomness of the stochastic compression operator $\mathcal{C}$.
\end{assumption}
Assumption~5 requires that the magnitude of the compression error be bounded.
The smaller the value of $C$, the higher the compression accuracy.
Note that Assumption~5 allows that the compression error is larger than the magnitude of the original vector, which is not allowed in the compressor used in  \cite{Cao2023a, Tu2022}. 
Moreover, the compressor satisfying Assumption~5 does not need to be unbiased, unlike the compressor used in \cite{Li2021b}, which requires unbiasedness.
Assumption~5 covers many popular compressors in machine learning and signal processing such as the deterministic quantization used in \cite{Zhu2016, Li2017} 
and the unbiased stochastic quantization used in \cite{Tang2018, Reisizadeh2019}.

Slater’s condition is present in the following.
\begin{assumption}
There exists a point ${x_s} \in \mathbb{X}$ and a positive constant ${\varsigma _s}$ such that
\begin{flalign}
{g_t}( {{x_s}} ) \le  - {\varsigma _s}{\mathbf{1}_m},t \in {\mathbb{N}_ + }. \label{ass8-eq1}
\end{flalign}
\end{assumption}
Slater’s condition is a sufficient condition for strong duality in convex optimization problems \cite{Boyd2004}. 
However, for nonconvex problems, Slater’s condition alone does not generally guarantee strong duality because the standard Lagrange dual often has a nonzero duality gap in nonconvex problems \cite{Bertsekas1997}.
For distributed online convex optimization, \cite{Yi2024, Zhang2024b} establish reduced network cumulative constraint violation bounds under Slater’s condition. However, to the best of our knowledge, there are no studies to show that similar reductions can be achieved for the nonconvex setting.

\section{Compressed Distributed Online Primal--Dual Algorithm \\with One-Point Bandit Feedback}
In this section, we consider the scenario where the agent receives one-point bandit feedback for the nonconvex local loss functions and the convex local constraint functions. 
Moreover, we propose a compressed distributed online primal--dual algorithm with one-point bandit feedback by using the class of general compressors satisfying Assumption~5, 
and analyze the performance of the algorithm without and with Slater’s condition, respectively.
\subsection{Algorithm Description}
To achieve reduced network cumulative constraint violation, \cite{Yi2024} proposes a distributed online primal--dual algorithm for distributed online convex optimization with time-varying constraints. 
In that algorithm, the gradients $\nabla {f_{i,t}}({x_{i,t}})$ and $\nabla {g_{i,t}}({x_{i,t}})$ can be directly accessed by agent~$i$ at each iteration~$t$. 
In contrast, we consider the bandit feedback setting where each agent only receives one-point bandit feedback for its local loss and constraint functions. 
Inspired by the one-point gradient estimator proposed in \cite{Flaxman2005}, we estimate the gradients by
\begin{flalign}
\nonumber
\hat {\nabla}_1 {f_{i,t}}( {{x_{i,t}}} ) &= \frac{p}{{{\delta _t}}}\big( {{f_{i,t}}( {{x_{i,t}} + {\delta _t}{u_{i,t}}} )} \big){u_{i,t}} \in {\mathbb{R}^p}, \\
\nonumber
\hat {\nabla}_1 {{g_{i,t}}( {{x_{i,t}}} )} &= \frac{p}{{{\delta _t}}}{\big( {{{{g_{i,t}}( {{x_{i,t}} + {\delta _t}{u_{i,t}}} )} }} \big)^T}
\otimes {u_{i,t}} \in {\mathbb{R}^{p \times {m_i}}},
\end{flalign}
where $p$ is the dimension of ${x_{i,t}}$, ${\delta _t} \in ( {0,r( \mathbb{X} ){\xi _t}} ]$ is an exploration parameter, ${r( \mathbb{X} )}$ is a known positive constant given in Assumption~1, 
${\xi _t} \in ( {0,1} )$ is a shrinkage coefficient, and ${u_{i,t}} \in {\mathbb{S}^p}$ is a uniformly distributed random perturbation vector.
Note that $\hat {\nabla}_1 {f_{i,t}}( {{x_{i,t}}} )$ and $\hat {\nabla}_1 {{g_{i,t}}( {{x_{i,t}}} )}$ are defined over the contracted set $( {1 - {\xi _t}} )\mathbb{X}$ rather than $\mathbb{X}$. 
This contraction ensures that the perturbed point ${{x_{i,t}} + {\delta _t}{u_{i,t}}}$ remains within $\mathbb{X}$.

Then, we give a one-point bandit variant of the algorithm proposed in \cite{Yi2024} in the following: 
\begin{subequations}
\begin{flalign}
 {x_{i,t}} &= \sum\limits_{j = 1}^n {{{[{W_t}]}_{ij}}{z_{j,t}}}, \label{Algorithm0-eq1} \\
 {v_{i,t + 1}} &= {\gamma _t}{[ {{g_{i,t}}( {{x_{i,t}} + {\delta _t}{u_{i,t}}} )} ]_ + }, \label{Algorithm1-eq4} \\
 {{a} _{i,t + 1}} &= {\hat{\nabla}_1} {f_{i,t}}({x_{i,t}}) + {\hat{\nabla}_1} {{g_{i,t}}({x_{i,t}}) }{v_{i,t+1}}, \label{Algorithm1-eq5} \\
 {z_{i,t + 1}} &= {\mathcal{P}_{( {1 - {\xi _{t + 1}}} )\mathbb{X}}}( {{x_{i,t}} - {\alpha _t}{{ a }_{i,t + 1}}} ), \label{Algorithm1-eq6} 
\end{flalign} \label{Algorithm0} 
\end{subequations}
where ${{\gamma _t}}$ is the regularization parameter, and ${{\alpha _t}}$ is the stepsize.

The algorithm \eqref{Algorithm0} needs a substantial amount of data exchange over all iterations, especially when the scale of the agents and the dimension $p$ are large.
To improve communication efficiency, we use the compressor satisfying Assumption~5 and modify \eqref{Algorithm0-eq1} as follows:
\begin{flalign}
{x_{i,t}} = \sum\limits_{j = 1}^n {{{[{W_t}]}_{ij}}{\hat{z}_{j,t}^i}}, \label{Algorithm1-eq2}
\end{flalign}
where 
\begin{flalign}
\hat z_{j,t}^i = \left\{ \begin{array}{l}
\hat z_{j,t - 1}^i + {s_t}\mathcal{C}\big( {( {{z_{j,t}} - \hat z_{j,t - 1}} )/{s_t}} \big),j \in \big(\big( {{\cal N}_i^{{\rm{in}}}({{\cal G}_{t}}) \cap {\cal N}_i^{{\rm{in}}}({{\cal G}_{t-1}})} \big)\cup \{ i \}\big),\\
{\hat z_{j,t}}, j \in \big( {{\cal N}_i^{{\rm{in}}}({{\cal G}_t})\backslash {\cal N}_i^{{\rm{in}}}({{\cal G}_{t - 1}})} \big).
\end{array} \right. \label{Algorithm1-eq1}
\end{flalign}
and
\begin{flalign}
 {\hat{z}_{j,t}} &= {\hat{z}_{j,t - 1}} + {s_t}\mathcal{C}\big( {( {{z_{j,t}} - {\hat{z}_{j,t-1}}} )/{s_{t}}} \big), j \in [n]. \label{Algorithm1-eq11}
\end{flalign}
Here, ${\hat{z}_{j,t}^i}$ and ${\hat{z}_{j,t}}$ are stored in agent~$i$ and agent $j$, respectively, which are the estimates for ${z}_{j,t}$ stored in agent~$j$.
Let ${\hat{z}_{j,0}^i} = {\hat{z}_{j,0}} = {\mathbf{0}_p}$ for $j \in [ n ]$, ${\cal N}_i^{{\rm{out}}}({{\cal G}_0})={\cal N}_i^{{\rm{out}}}({{\cal G}_1})$ and ${\cal N}_i^{{\rm{in}}}({{\cal G}_0})={\cal N}_i^{{\rm{in}}}({{\cal G}_1})$. 
Then, if agent $j$ is an in-neighbor of agent $i$ at iteration~$t=1$, we know $\hat z_{j,1}^i = {\hat{z}_{j,1}}$ and only the compressed value ${\cal C}({z_{j,1}}/{s_1})$ is sent to agent $i$.
If agent $j$ remains the in-neighbor of agent $i$ at iteration~$t=2$, then $\hat z_{j,2}^i$ will be equal to ${\hat{z}_{j,2}}$.
Recall that we consider the time-varying graph.
Suppose that a new neighbor $j^*$ becomes an in-neighbor of agent $i$ at iteration~$t=2$.
In this case, agent $i$ is unable to track ${\hat{z}_{j^*,2}}$ as accurately as it does for its neighbor $j$ due to the missing of ${\hat{z}_{j^*,1}}$.
Therefore, we allow the new neighbor $j^*$ to send ${\hat{z}_{j^*,2}}$ to agent $i$, 
thereby ensuring that $\hat z_{j^*,2}^i = {\hat{z}_{j^*,2}}$. This procedure continues in the same manner for subsequent iterations.
Note that ${\hat{z}_{i,t}^i} = {\hat{z}_{i,t}}$ for $t \in {\mathbb{N}_ + }$.
As a result, we have ${\hat{z}_{j,t}^i}={\hat{z}_{j,t}}$ for $j \in \big( {\mathcal{N}_i^{\mathrm{in}}( {{\mathcal{G}_t}} ) \cup \{ i \}} \big)$ and $t \in {\mathbb{N}_ + }$.
Then, the compressed distributed online primal--dual algorithm with one-point bandit feedback is proposed. 
It is presented in pseudo-code as Algorithm~1.
\begin{algorithm}[!t]
  \caption{Compressed Distributed Online Primal--Dual Algorithm with One-Point Bandit Feedback} 
  \begin{algorithmic}
  \renewcommand{\algorithmicrequire}{\textbf{Input:}}
  \REQUIRE
   positive constant ${r( \mathbb{X} )}$, nonincreasing stepsize sequence $\{ {\alpha _t}\} \subseteq ( {0, + \infty })$, 
   nondecreasing regularization parameter sequence $\{ {\gamma _t}\} \subseteq ( {0, + \infty })$, nonincreasing shrinkage coefficient sequence $\{ {{\xi _t}} \} \subseteq ( {0,1} )$, 
   nonincreasing exploration parameter sequence $\{ {{\delta _t}} \} \subseteq ( {0,r( \mathbb{X} ){\xi _t}} ]$,
   and decreasing scaling parameter sequence $\{ {{s _t}} \} \subseteq ( {0, + \infty })$.
  \renewcommand{\algorithmicrequire}{\textbf{Initialize:}}
  \REQUIRE
         ${\hat{z}_{j,0}^i} = {\hat{z}_{j,0}} = {\mathbf{0}_p}$ for $i \in [ n ]$ and $j \in [ n ]$, and ${z_{i,1}} \in {(1-\xi_{1})\mathbb{X}}$, $\forall i \in [ n ]$, ${\cal N}_i^{{\rm{out}}}({{\cal G}_0})={\cal N}_i^{{\rm{out}}}({{\cal G}_1})$, ${\cal N}_i^{{\rm{in}}}({{\cal G}_0})={\cal N}_i^{{\rm{in}}}({{\cal G}_1})$.
    \FOR {$t = 1, \cdot  \cdot  \cdot, T-1 $}
    \FOR {$i = 1,\cdot  \cdot  \cdot,n$ in parallel}
    \STATE Send $\mathcal{C}\big( {( {{z_{i,t}} - {\hat{z}_{i,t-1}}} )/{s_{t}}} \big)$ to agent $j$ if $j \in \big( {{\cal N}_i^{{\rm{out}}}({{\cal G}_t}) \cap {\cal N}_i^{{\rm{out}}}({{\cal G}_{t - 1}})} \big)$ or send ${\hat{z}_{i,t}}$ to agent $j$ if $j \in \big( {{\cal N}_i^{{\rm{out}}}({{\cal G}_t})\backslash {\cal N}_i^{{\rm{out}}}({{\cal G}_{t - 1}})} \big)$, receive $\mathcal{C}\big( {( {{z_{j,t}} - {\hat{z}_{j,t-1}}} )/{s_{t}}} \big)$ from agent $j$ if $j \in \big( {{\cal N}_i^{{\rm{in}}}({{\cal G}_t}) \cap {\cal N}_i^{{\rm{in}}}({{\cal G}_{t - 1}})} \big)$ or receive ${\hat{z}_{j,t}}$ from agent $j$ if $j \in \big( {{\cal N}_i^{{\rm{in}}}({{\cal G}_t})\backslash {\cal N}_i^{{\rm{in}}}({{\cal G}_{t - 1}})} \big)$.
    \STATE Update ${x_{i,t}}$ according to \eqref{Algorithm1-eq2}--\eqref{Algorithm1-eq11}.
    \STATE Select vector ${u_{i,t}} \in {\mathbb{S}^p}$ independently and uniformly at random.
    \STATE Observe ${{f_{i,t}}( {{x_{i,t}} + {\delta _t}{u_{i,t}}} )}$ and ${{ {{g_{i,t}}( {{x_{i,t}} + {\delta _t}{u_{i,t}}} )}  }}$.
    \STATE Update {$z_{i,t + 1}$} according to \eqref{Algorithm1-eq4}--\eqref{Algorithm1-eq6}.
    \ENDFOR
    \ENDFOR
  \renewcommand{\algorithmicensure}{\textbf{Output:}}
  \ENSURE
      $\{ x_{i,t} \}$.
  \end{algorithmic}
\end{algorithm}

In the following, we provide a concise explanation of why the compression strategy employed in Algorithm~1 reduces the compression error when utilizing the class of general compressors that satisfy Assumption~5.
Note that the compression error is the difference between ${z_{i,t}}$ in \eqref{Algorithm0-eq1} and ${\hat{z}_{i,t}^j}$ in \eqref{Algorithm1-eq2} for $i \in \big( {\mathcal{N}_j^{\mathrm{in}}( {{\mathcal{G}_t}} ) \cup \{ j \}} \big)$.
For all $i \in [ n ]$ and $t \in {\mathbb{N}_ + }$, we have
\begin{flalign}
\nonumber
{( {{\mathbf{E}_\mathcal{C}}[ {\| {{z_{i,t}} - {\hat{z}_{i,t}^j}} \|} ]} )^2} 
& = {( {{\mathbf{E}_\mathcal{C}}[ {\| {{z_{i,t}} - {\hat{z}_{i,t}}} \|} ]} )^2} \\
\nonumber
&\le {\mathbf{E}_\mathcal{C}}[ {{\| {{z_{i,t}} - {\hat{z}_{i,t}}} \|^2}} ]  \\
\nonumber
& = {\mathbf{E}_\mathcal{C}}[{\| {{z_{i,t}} - {\hat{z}_{i,t-1}} - {s_{t}}\mathcal{C}( {( {{z_{i,t}} - {\hat{z}_{i,t-1}}} )/{s_{t}}} )} \|^2}] \\
& = s_{t}^2{\mathbf{E}_\mathcal{C}}[{\| {( {{z_{i,t}} - {\hat{z}_{i,t-1}}} )/{s_{t}} - \mathcal{C}( {( {{z_{i,t}} - {\hat{z}_{i,t-1}}} )/{s_{t}}} )} \|^2}]
\le Cs_{t}^2, \label{Lemma7-proof-eq9}
\end{flalign}
where the first equality holds due to ${\hat{z}_{i,t}^j}={\hat{z}_{i,t}}$, the first inequality holds due to the Jensen’s inequality, the second equality holds due to \eqref{Algorithm1-eq11}, 
and the last inequality holds due to \eqref{compre-eq1}.

From \eqref{Lemma7-proof-eq9}, we know that ${\mathbf{E}_\mathcal{C}}[ {\| {{z_{i,t}} - {{\hat z}_{i,t}^j}} \|} ] \le \sqrt C {s_t}$, i.e., the compression error is bounded by the decreasing sequence~$\{ {{s _t}} \}$. If we let $\{ {{s _t}} \}$ decrease to zero over iterations, then the compression error also decreases to zero.

\subsection{Performance Analysis}
In this section, we establish network regret and cumulative constraint violation bounds for Algorithm~1 without and with Slater’s condition, respectively. 
We choose some natural decreasing and increasing parameter sequences in the following theorem. 

\begin{theorem}\label{thm1}
Suppose Assumptions 1--5 hold. For all $i \in [ n ]$, let $\{ {{x_{i,t}}} \}$ be the sequences generated by Algorithm~1 with
\begin{flalign}
{\alpha _t} = \frac{{{\alpha _0}}}{{{t^{{\theta _1}}}}}, {\gamma _t} = {\gamma _0}{t^{{\theta _2}}}, {\xi _t} = \frac{1}{{{t^{{\theta _3}}}}}, 
{\delta _t} = \frac{{r( \mathbb{X} )}}{{{t^{{\theta _3}}}}}, {s_t} = \frac{{{s_0}}}{{{t^{{\theta _4}}}}}, \label{theorem1-eq1}
\end{flalign}
where ${\alpha _0} > 0$, ${\theta _1} \in ( {0,1} )$, ${\gamma _0} \in \big( {0,r{{( \mathbb{X} )}^2}/( {2{p^2}F_2^2} )} \big]$, ${\theta _2} \in ( {0,{\theta _1}/3} )$, ${\theta _3} \in \big( {{\theta _2},( {{\theta _1} - {\theta _2}} )/2} \big]$, 
${s_0} > 0$, and ${\theta _4} \ge 1$ are constants. Then, for any $T \in {\mathbb{N}_ + }$,
\begin{flalign}
\mathbf{E}[{{\rm{Net}\mbox{-}\rm{Reg}}( T )}] &= \mathcal{O}( {{T^{\max \{ {\theta _1},1 - {\theta _1} + 2{\theta _3},1 + {\theta _2} - {\theta _3} \}}}} ), \label{theorem1-eq2}\\
\mathbf{E}[{{\rm{Net} \mbox{-} \rm{CCV}}( T )}] &= \mathcal{O}( {T^{1 - {\theta _2}/2}} ). \label{theorem1-eq3}
\end{flalign}
Moreover, if Assumption~6 also holds, then
\begin{flalign}
\mathbf{E}[{{\rm{Net} \mbox{-} \rm{CCV}}( T )}] &= \mathcal{O}( {T^{1 - {\theta _2}}} ). \label{theorem1-eq4}
\end{flalign}
\end{theorem}
\begin{proof}
The proof and the explicit expressions of the right-hand sides of \eqref{theorem1-eq2}–\eqref{theorem1-eq4} are given in Appendix B.
\end{proof}
\begin{remark}\label{rem2}
We show in Theorem~1 that Algorithm~1 establishes sublinear network regret and cumulative constraint violation bounds as in \eqref{theorem1-eq2}--\eqref{theorem1-eq3}, thereby answering {Questions~1--3} raised in Section~1.  
In addition, when Slater’s condition holds, the network cumulative constraint violation bound is further reduced as in \eqref{theorem1-eq4}. 
That answers {Question~4} raised in Section~1.
In these bounds of \eqref{theorem1-eq2}--\eqref{theorem1-eq4}, ${\theta _2}$ and ${\theta _3}$ capture the impact of inequality constraints and one-point bandit feedback on the network regret and cumulative constraint violation bounds, respectively.
The exponents governing the growth of these bounds with respect to $T$ are unaffected by compressed communication since the compression error is effectively suppressed when ${\theta _4} \ge 1$.
If we choose ${\theta _1} = 1 - {\theta _1} + 2{\theta _3} = 1 + {\theta _2} - {\theta _3}$, 
then we obtain ${\theta _1} \in ( {3/4,5/6} ]$, ${\theta _2} = 2{\theta _1} - 3/2$ and ${\theta _3} = {\theta _1} - 1/2$.
Therefore, the bounds in \eqref{theorem1-eq2}--\eqref{theorem1-eq4} can be rewritten as $\mathcal{O}( {{T^{\theta _1}}} )$, $\mathcal{O}( {T^{7/4 - {\theta _1}}} )$, $\mathcal{O}( {T^{5/2 - 2{\theta _1}}} )$, respectively.
In such a case,
as the user-defined trade-off parameter ${\theta _1}$ increases, the network regret bound for Algorithm~1 becomes larger, while the network cumulative constraint violation bounds become smaller, both with and without Slater’s condition.
If inequality constraints are not considered, it is straightforward to observe that Theorem~1 establishes an $\mathcal{O}( {{T^{3/4}} } )$ network regret bound, 
which is the same as the results established by the compressed distributed online algorithms with one-point bandit feedback in \cite{Tu2022, Ge2023}.
\end{remark}

\section{Compressed Distributed Online Primal--Dual Algorithm \\with Two-Point Bandit Feedback}
This section proposes a two-point bandit feedback version of the proposed algorithm in Section~III, 
and analyzes the performance of the algorithm without and with Slater’s condition, respectively.
\subsection{Algorithm Description}
In the two-point bandit feedback setting, the values of the local loss and constraint functions at two points are revealed at each iteration. 
Inspired by the two-point gradient estimator proposed in \cite{Agarwal2010, Shamir2017}, we estimate the gradients $\nabla {f_{i,t}}({x_{i,t}})$ and $\nabla {g_{i,t}}({x_{i,t}})$ by
\begin{flalign}
\nonumber
\hat \nabla_2 {f_{i,t}}( {{x_{i,t}}} ) &= \frac{p}{{{\delta _t}}}\big( {{f_{i,t}}( {{x_{i,t}} + {\delta _t}{u_{i,t}}} ) - {f_{i,t}}( {{x_{i,t}}} )} \big){u_{i,t}} \in {\mathbb{R}^p}, \\
\nonumber
\hat \nabla_2 {{g_{i,t}}( {{x_{i,t}}} )} &= \frac{p}{{{\delta _t}}}{\big( {{{{g_{i,t}}( {{x_{i,t}} + {\delta _t}{u_{i,t}}} )} } - { {{g_{i,t}}( {{x_{i,t}}} )} }} \big)^T}
\otimes {u_{i,t}} \in {\mathbb{R}^{p \times {m_i}}}.
\end{flalign}
In this way, Algorithm~1 can be modified into a two-point bandit feedback version, which is outlined in pseudo-code as Algorithm~2.
\begin{algorithm}[!t]
  \caption{Compressed Distributed Online Primal--Dual Algorithm with Two-Point Bandit Feedback} 
  \begin{algorithmic}
  \renewcommand{\algorithmicrequire}{\textbf{Input:}}
  \REQUIRE
   positive ${r( \mathbb{X} )}$, nonincreasing sequences $\{ {\alpha _t}\} \subseteq ( {0, + \infty })$, 
   $\{ {\gamma _t}\} \subseteq ( {0, + \infty })$, 
   $\{ {{\xi _t}} \} \subseteq ( {0,1} )$, and $\{ {{\delta _t}} \} \subseteq ( {0,r( \mathbb{X} ){\xi _t}} ]$,
   decreasing sequence $\{ {{s _t}} \} \subseteq ( {0, + \infty })$.
  \renewcommand{\algorithmicrequire}{\textbf{Initialize:}}
  \REQUIRE
      ${\hat{z}_{j,0}^i} = {\hat{z}_{j,0}} = {\mathbf{0}_p}$ for $i \in [ n ]$ and $j \in [ n ]$, and ${z_{i,1}} \in {(1-\xi_{1})\mathbb{X}}$, $\forall i \in [ n ]$, ${\cal N}_i^{{\rm{out}}}({{\cal G}_0})={\cal N}_i^{{\rm{out}}}({{\cal G}_1})$, ${\cal N}_i^{{\rm{in}}}({{\cal G}_0})={\cal N}_i^{{\rm{in}}}({{\cal G}_1})$.
    \FOR {$t = 1, \cdot  \cdot  \cdot, T-1 $}
    \FOR {$i = 1,\cdot  \cdot  \cdot,n$ in parallel}
    \STATE Send $\mathcal{C}\big( {( {{z_{i,t}} - {\hat{z}_{i,t-1}}} )/{s_{t}}} \big)$ to agent $j$ if $j \in \big( {{\cal N}_i^{{\rm{out}}}({{\cal G}_t}) \cap {\cal N}_i^{{\rm{out}}}({{\cal G}_{t - 1}})} \big)$ or send ${\hat{z}_{i,t}}$ to agent $j$ if $j \in \big( {{\cal N}_i^{{\rm{out}}}({{\cal G}_t})\backslash {\cal N}_i^{{\rm{out}}}({{\cal G}_{t - 1}})} \big)$, receive $\mathcal{C}\big( {( {{z_{j,t}} - {\hat{z}_{j,t-1}}} )/{s_{t}}} \big)$ from agent $j$ if $j \in \big( {{\cal N}_i^{{\rm{in}}}({{\cal G}_t}) \cap {\cal N}_i^{{\rm{in}}}({{\cal G}_{t - 1}})} \big)$ or receive ${\hat{z}_{j,t}}$ from agent $j$ if $j \in \big( {{\cal N}_i^{{\rm{in}}}({{\cal G}_t})\backslash {\cal N}_i^{{\rm{in}}}({{\cal G}_{t - 1}})} \big)$.
    \STATE Update ${x_{i,t}}$ according to \eqref{Algorithm1-eq2}--\eqref{Algorithm1-eq11}.
    \STATE Select vector ${u_{i,t}} \in {\mathbb{S}^p}$ independently and uniformly at random.
    \STATE Observe ${{f_{i,t}}( {{x_{i,t}}} )}$, ${{f_{i,t}}( {{x_{i,t}} + {\delta _t}{u_{i,t}}} )}$, ${{ {{g_{i,t}}( {{x_{i,t}}} )} }}$, ${{ {{g_{i,t}}( {{x_{i,t}} + {\delta _t}{u_{i,t}}} )}  }}$.
    \STATE Update
    \begin{subequations}
     \begin{flalign}
       {v_{i,t + 1}} &= {\gamma _t}{[ {{g_{i,t}}( {{x_{i,t}}} )} ]_ + }, \label{Algorithm2-eq3} \\
       {b _{i,t + 1}} &= {\hat{\nabla}}_2 {f_{i,t}}({x_{i,t}}) + {\hat{\nabla}}_2 {{g_{i,t}}({x_{i,t}}) }{v_{i,t+1}}, \label{Algorithm2-eq4} \\
       {z_{i,t + 1}} &= {\mathcal{P}_{( {1 - {\xi _{t + 1}}} )\mathbb{X}}}( {{x_{i,t}} - {\alpha _t}{b_{i,t + 1}}} ). \label{Algorithm2-eq5} 
      \end{flalign}
     \end{subequations}
    \ENDFOR
    \ENDFOR
  \renewcommand{\algorithmicensure}{\textbf{Output:}}
  \ENSURE
      $\{ x_{i,t} \}$.
  \end{algorithmic}
\end{algorithm}

\subsection{Performance Analysis}
Similar to Section~III.B, in this section we establish network regret and cumulative constraint violation bounds for Algorithm~2 without and with Slater’s condition, respectively. 
We choose some natural decreasing and increasing parameter sequences in the following theorem. 

\begin{theorem}\label{thm2}
Suppose Assumptions 1--5 hold. For all $i \in [ n ]$, let $\{ {{x_{i,t}}} \}$ be the sequences generated by Algorithm~2 with
\begin{flalign}
{\alpha _t} = \frac{{\alpha _0}}{{{t^{{\theta _1}}}}}, {\gamma _t} = \frac{{{\gamma _0}}}{{{\alpha _t}}}, {\xi _t} = {\alpha _t}, 
{\delta _t} = r( \mathbb{X} ){\alpha _t}, {s_t} = \frac{{s_0}}{{{t^{{\theta _2}}}}}, \label{theorem2-eq1}
\end{flalign}
where ${\alpha _0} > 0$, ${\theta _1} \in ( {0,1} )$, ${\gamma _0} \in ( {0,1/( {4(p^2+1)G_2^2} )} ]$, ${s_0} > 0$, and ${\theta _2} \ge 1$ are constants. Then, for any $T \in {\mathbb{N}_ + }$,
\begin{flalign}
\mathbf{E}[{{\rm{Net}\mbox{-}\rm{Reg}}( T )}] &= \mathcal{O}( {{T^{\max \{ {1 - {\theta _1},{\theta _1}} \}}}} ), \label{theorem2-eq2}\\
\mathbf{E}[{{\rm{Net} \mbox{-} \rm{CCV}}( T )}] &= \mathcal{O}( {T^{1 - {\theta _1}/2}} ). \label{theorem2-eq3}
\end{flalign}
Moreover, if Assumption~6 also holds, then
\begin{flalign}
\mathbf{E}[{{\rm{Net} \mbox{-} \rm{CCV}}( T )}] &= \mathcal{O}( {T^{1 - {\theta _1}}} ). \label{theorem2-eq4}
\end{flalign}
\end{theorem}
\begin{proof}
The proof and the explicit expressions of the right-hand sides of \eqref{theorem2-eq2}–\eqref{theorem2-eq4} are given in Appendix C.
\end{proof}
\begin{remark}\label{rem4}
We show in Theorem~2 that Algorithm~2 achieves the reduced network regret bounds compared with the results established in Theorem~1 when ${\theta _1} \in [ {1/4,3/4} ]$.
Moreover, for all ${\theta _1} \in ( {0,1} )$, 
Algorithm~2 achieves strictly tighter network cumulative constraint violation bounds than the results established in Theorem~1, both without and with Slater’s condition.
For ${\theta _1} \in ( {0,1/2} ]$, as the user-defined trade-off parameter ${\theta _1}$ increases, the network regret and cumulative constraint violation bounds for Algorithm~2 become smaller, both with and without Slater’s condition. For ${\theta _1} \in ( {1/2,1} )$, as ${\theta _1}$ increases, the network regret bound becomes larger, while the network cumulative constraint violation bounds become smaller, both with and without Slater’s condition.
These bounds are the same as the state-of-the-art results established by the distributed bandit online algorithms with perfect communication in \cite{Yi2023, Zhang2024b}. 
When ${\theta _1} = 1/2$, the bounds in \eqref{theorem2-eq2}--\eqref{theorem2-eq4} becomes $\mathcal{O}( {\sqrt T } )$, $\mathcal{O}( {T^{3/4}} )$ and $\mathcal{O}( {\sqrt T } )$, respectively. 
In such a case, the network regret bound is the same as the results established in \cite{Cao2023a, Tu2022} for distributed online convex optimization without inequality constraints and the result established in \cite{Lu2021a} for distributed online nonconvex optimization without inequality constraints.
Moreover, the $\mathcal{O}( {\sqrt T } )$ network cumulative constraint violation bound is tighter than the $\mathcal{O}( {{T^{3/4}}} )$ bound established in \cite{Suo2025} for distributed online nonconvex optimization with inequality constraints under Slater’s condition.
Note that \cite{Suo2025} considers the full-information feedback setting and static inequality constraints, and use constraint violation metric, while we consider the two-point bandit feedback setting and time-varying inequality constraints, and uses the stricter cumulative constraint violation metric and compressed communication.
\end{remark}

\section{Simulation Example}
To evaluate the performance of Algorithms~1 and~2, we consider a distributed
online localization problem with time-varying constraints over a network of $n$ sensors as follows: 
\begin{subequations}
\begin{flalign}
&\mathop {\min }\limits_x \;\;\;\sum\limits_{t = 1}^T {\sum\limits_{i = 1}^n {\frac{1}{4}} } {\big| {{{\| {{S_i} - x} \|}^2} - {D_{i,t}}} \big|^2}, \label{example1-1}\\ 
&\;{\rm{s}}{\rm{.t}}{\rm{.   }}\;\;\;\;x \in \mathbb{X},{\rm{ }}{B_{i,t}}x - {b_{i,t}} \le {\mathbf{0}_{{m_i}}},\forall i \in [ n ],\forall t \in [ T ]. \label{example1-2}
\end{flalign} \label{example1}%
\end{subequations}
In this problem, the sensors aim to cooperatively track a moving target. The position of sensor~$i$ is denoted by $S_i \in {\mathbb{R}^p}$. Sensor~$i$ measures the distance between the positions of the target and itself by ${D_{i,t}} = {\| {{S_i} - {X_{0,t}}} \|^2} + {\tau _{i,t}}$ where ${X_{0,t}} \in {\mathbb{R}^p}$ and ${\tau _{i,t}} \in \mathbb{R}$ denote the position of the target and the measurement noise at iteration $t$, respectively. Moreover, each agent only has access to its distance measurement. In addition, the position of the target is required to remain within a designated safe region characterized by the set $\mathbb{X}$, and avoid prolonged deviations from the mission region characterized by time-varying linear inequality constraint ${B_{i,t}}x - {b_{i,t}} \le {\mathbf{0}_{{m_i}}}$ with ${B_{i,t}} \in {\mathbb{R}^{{m_i} \times p}}$ and ${b_{i,t}} \in {\mathbb{R}^{{m_i}}}$. However, occasional deviations are permitted to accommodate obstacle avoidance and exploration requirements. 
Inspired by \cite{Chen2012, DiLorenzo2016, Lu2023a}, the problem as in \eqref{example1} is formulated to achieve the least squares estimator for the position of the target.  
We set $n  = 100$, then the communication topology is modeled by a time-varying directed graph. 
Specifically, at each iteration~$t$, the graph is first randomly generated where the probability of any two sensors being connected is $\rho $. 
Then,
we add directed edges~$( {i,i + 1} )$ for $i \in [ 99 ]$ and directed edge~$( {100,1} )$. Moreover, let
${[ {{W_t}} ]_{ij}} = \frac{1}{100}$ if $( {j,i} ) \in {\mathcal{E}_t}$ and ${[ {{W_t}} ]_{ii}} = 1 - \sum\nolimits_{j = 1}^{100} {{{[ {{W_t}} ]}_{ij}}} $.
We can check the graph satisfies Assumption~4.

In this paper, we show in Theorems 1 and 2 that Algorithms 1 and 2 achieve sublinear bounds on both network regret and cumulative constraint violation 
for distributed online nonconvex optimization with time-varying constraints.
More importantly, under Slater’s condition, the network cumulative constraint violation bounds are further reduced.
To the best of our knowledge, such reductions under Slater’s condition have only been reported in \cite{Yi2024, Zhang2024b} 
where state-of-the-art results are established for distributed online convex optimization with time-varying constraints and perfect communication is considered.
Therefore, we compare Algorithms~1 and~2 with the algorithms in \cite{Yi2024, Zhang2024b}. 
Note that \cite{Yi2024} considers the full-information feedback setting and \cite{Zhang2024b} considers the two-point bandit feedback setting.
We set $\rho  = 0.1$, $\mathbb{X} = {[ { - 5,5} ]^p}$, $p = 2$, ${m_i} = 2$, 
and randomly choose each component of ${{S_i}}$ from the uniform distribution in the interval $[ { - 5,5} ]$. We assume that the position of the target evolves by
\begin{flalign}
\nonumber
{X_{0,t + 1}} = {X_{0,t}} + \left[ {\begin{array}{*{20}{c}}
{\frac{{{{( { - 1} )}^{{Q_t}}}\sin ( {t/50} )}}{{10t}}}\\
{\frac{{ - {Q_t}\cos ( {t/70} )}}{{40t}}}
\end{array}} \right],
\end{flalign}
where ${{Q_t}}$ is randomly generated from Bernoulli distribution with a success probability of $0.5$, and ${X_{0,0}} = {[ {0.8,0.95} ]^T}$. Moreover, ${\tau _{i,t}}$ is randomly generated from the uniform distribution from in the interval $[ { 0,0.001} ]$. Furthermore, each component of ${B_{i,t}}$ is randomly generated from the uniform distribution in the interval $[ {0,2} ]$, and each component of ${b_{i,t}}$ is randomly generated from the uniform distribution in the interval $[ {b,b+1} ]$. Here we choose $b=0.01$ such that Slater’s condition holds.
In addition, we select the following compressor for Algorithms~1 and~2.

{\bf{Standard uniform quantizer}}:
\begin{flalign}
\nonumber
\mathcal{C}( x ) = \Delta \Big\lfloor {\frac{x}{\Delta } + \frac{{{\mathbf{1}_p}}}{2}} \Big\rfloor,
\end{flalign}
where $\Delta $ is a positive integer. This compressor satisfies Assumption~5 with $C = p{\Delta ^2}/4$, 
which is also used in \cite{Yi2022, Ge2023, Li2017, Khirirat2021}. Transmitting $\mathcal{C}( x )$ requires $pq$ bits if each integer is encoded using $q$ bits. 
Here we set $\Delta =1$ and $q=8$.

\begin{figure}[!ht]
 \centering
  \includegraphics[width=6cm]{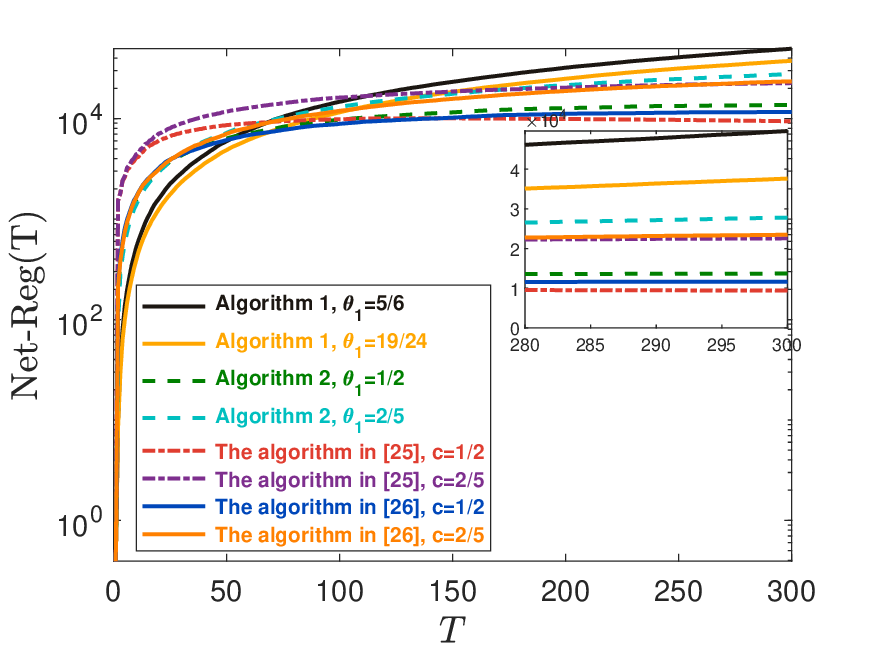}
  \caption{Evolutions of network regret under different trade-off parameters.}
\end{figure}
\begin{figure}[!ht]
  \centering
  \includegraphics[width=6cm]{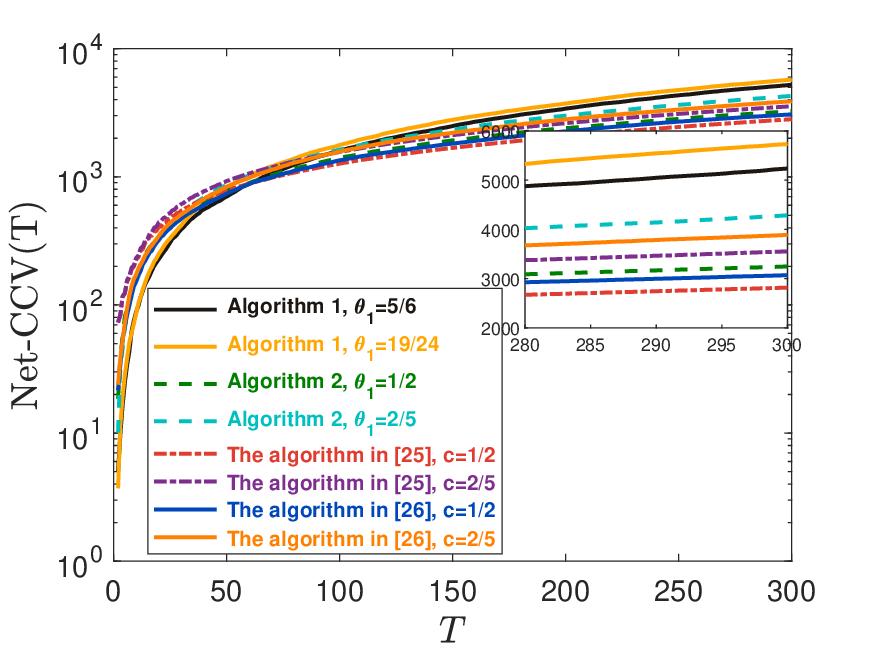}
  \caption{Evolutions of network cumulative constraint violation under different trade-off parameters.}
\end{figure}
\begin{figure}[!ht]
 \centering
  \includegraphics[width=6cm]{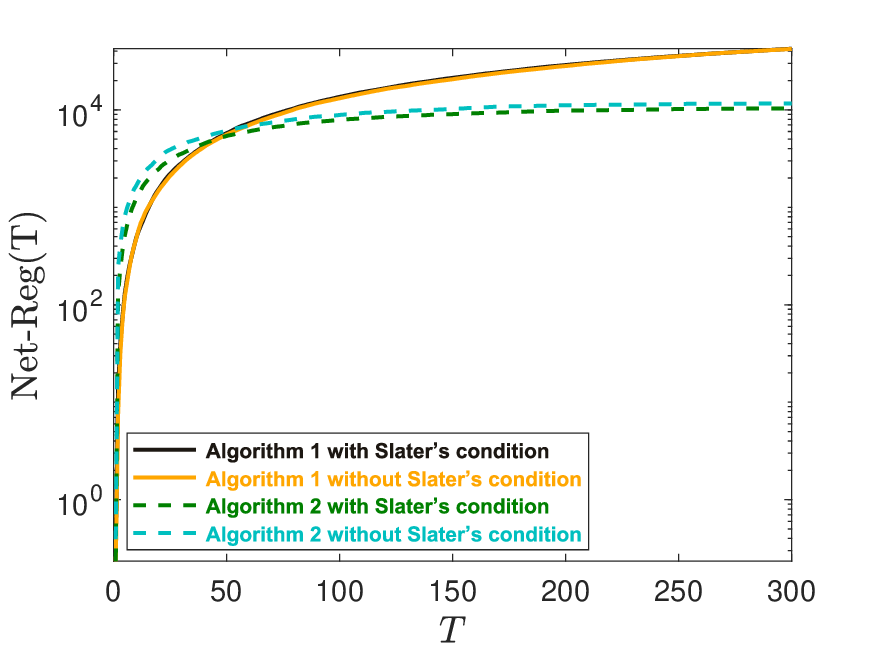}
  \caption{Evolutions of network regret with and without Slater’s condition.}
\end{figure}
\begin{figure}[!ht]
  \centering
  \includegraphics[width=6cm]{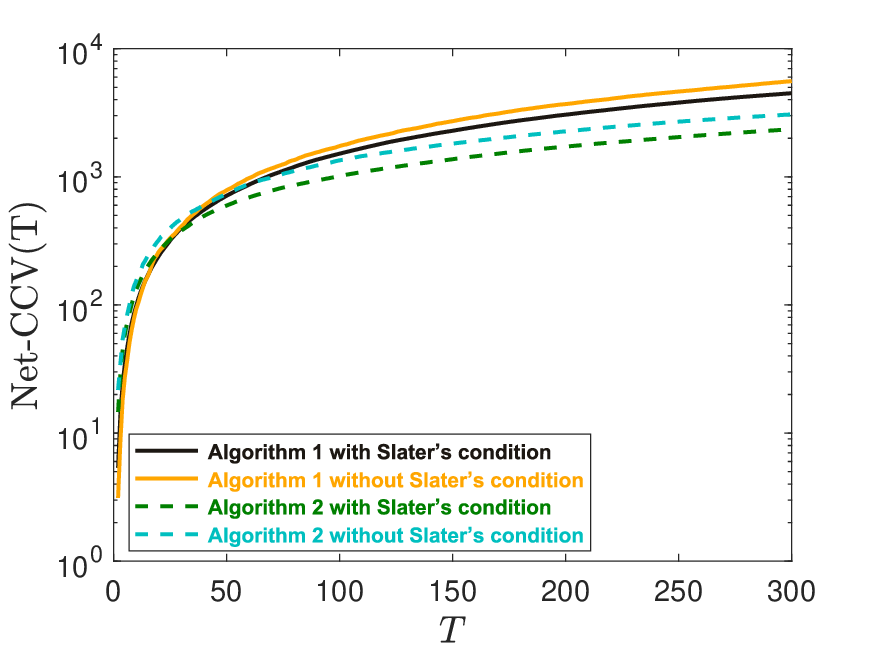}
  \caption{Evolutions of network cumulative constraint violation with and without Slater’s condition.}
\end{figure}
\begin{figure}[!ht]
 \centering
  \includegraphics[width=6cm]{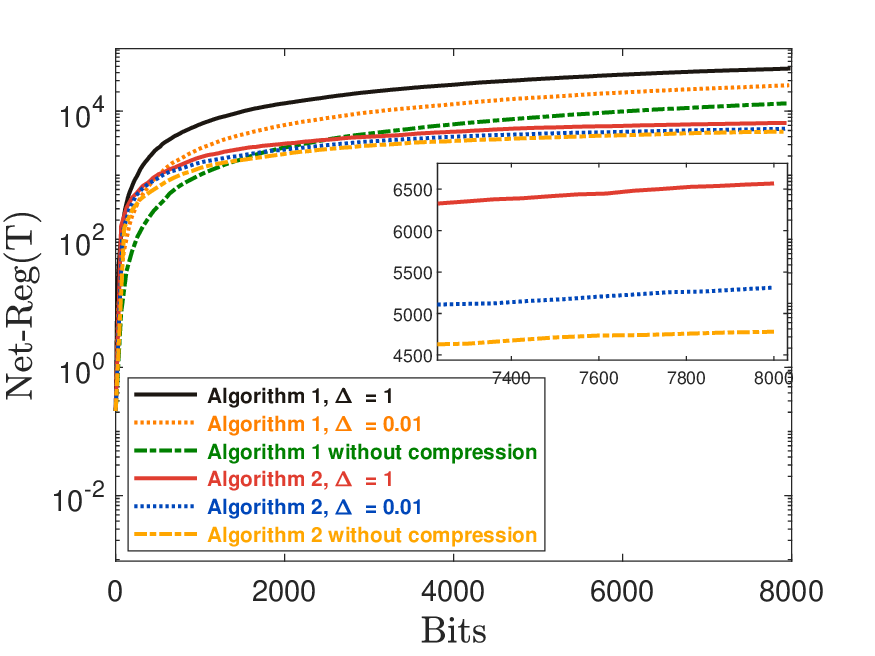}
  \caption{Evolutions of network regret under different quantization levels.}
\end{figure}
\begin{figure}[!ht]
  \centering
  \includegraphics[width=6cm]{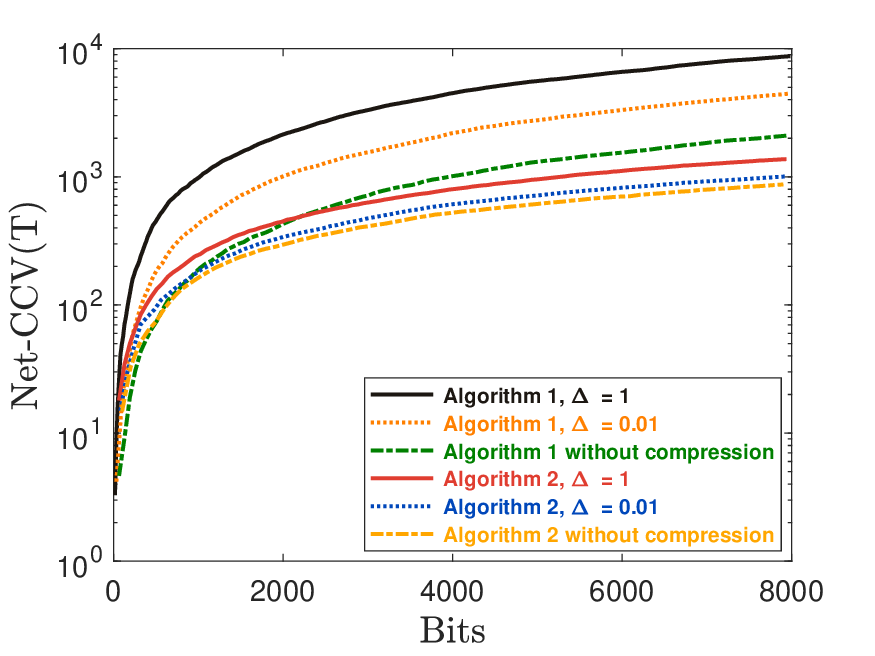}
  \caption{Evolutions of network cumulative constraint violation under different quantization levels.}
\end{figure}
Figs.~1 and 2 illustrate the evolutions of network regret and cumulative constraint violation under different trade-off parameters, respectively. 
As shown in Figs.~1 and 2, our Algorithm~2 exhibits slightly larger network regret and cumulative constraint violation than those of the algorithm in \cite{Zhang2024b}.
Furthermore, our Algorithm~1 has the largest network regret and cumulative constraint violation, and the algorithm in \cite{Yi2024} has the smallest network regret and cumulative constraint violation.
That is reasonable, as the algorithm in \cite{Yi2024} uses full-information feedback, our Algorithm~2 and the algorithm in \cite{Zhang2024b} use two-point bandit feedback,
and our Algorithm~1 only uses one-point bandit feedback. 
Moreover, our Algorithms~1 and~2 use compressed communication, whereas the algorithms in \cite{Yi2024, Zhang2024b} use perfect communication.
In addition, for Algorithm~1, as ${\theta _1}$ increases from $19/24$ to $5/6$, the network regret becomes larger, whereas the network cumulative constraint violation bound becomes smaller.
For Algorithm~2, as ${\theta _1}$ increases from $2/5$ to $1/2$, both the network regret and cumulative constraint violation become smaller.
Figs. 3 and 4 illustrate the evolutions of network regret and cumulative constraint violation with and without Slater’s condition, respectively.
For the setting where Slater’s condition does not hold, we choose ${B_{i,t}} = \left[ {\begin{array}{*{20}{c}}
\hat{b}&0\\
{ - \hat{b}}&0
\end{array}} \right]$, ${b_{i,t}} = \left[ {\begin{array}{*{20}{c}}
0\\
0
\end{array}} \right]$. Here $\hat{b}$ is randomly generated from the uniform distribution in the interval $[ {0,1} ]$.
As shown in Figs. 3 and 4, Algorithms~1 and~2 exhibit almost the same network regret regardless of whether Slater’s condition holds, respectively. In contrast, when Slater’s condition holds, they exhibit smaller network cumulative constraint violation.
These simulation results are consistent with the results in Theorems~1 and~2.
In addition, as shown in Figs.~5 and~6, for Algorithms 1 and 2, higher quantization levels result in larger network regret and cumulative constraint violation under the same number of transmitted bits.

\section{Conclusions}
This paper studied the distributed online nonconvex optimization problem with time-varying constraints. 
To better utilize communication resources, we proposed two distributed bandit online primal--dual algorithms with compressed communication, tailored for the one-point and two-point bandit feedback settings, respectively. 
More importantly, these algorithms were able to handle time-varying and directed communication topologies. 
We demonstrated that they established sublinear network regret and cumulative constraint violation bounds. 
Moreover, the network cumulative constraint violation bounds were further reduced under Slater’s condition. 
In the future, we plan to investigate compressors with bounded relative compression error, as compressors with bounded absolute and relative errors generally exhibit distinct characteristics, and neither type is strictly more restrictive than the other, nor does one imply the other. Unlike compressors with bounded absolute compression error, where the error bound is fixed regardless of the input, a key challenge in studying compressors with bounded relative compression error is that the compression error scales with the norm of the input vector, which necessitates the development of new algorithmic designs and analytical methods.

\appendix

\hspace{-3mm}\emph{A. Useful Lemmas}

We begin by presenting several preliminary results that will be utilized in the subsequent proofs.
\begin{lemma}
(\cite{Nedic2009, Nedic2014})
Let ${W_t}$ denote the mixing matrix associated with a time-varying graph that satisfies Assumption~4. Then,
\begin{flalign}
\Big| {{{[ {\Psi _s^t} ]}_{ij}} - \frac{1}{n}} \Big| \le \tau {\lambda ^{t - s}}, \forall i,j \in [ n ], \forall t \ge s \ge 1, \label{Lemma2-eq1}
\end{flalign}
where $\Psi _s^t = {W_t}{W_{t - 1}} \cdots {W_s}$, $\tau  = {( {1 - \omega /4{n^2}} )^{ - 2}} > 1$, and $\lambda  = {( {1 - \omega /4{n^2}} )^{1/B}} \in ( {0,1} )$.
\end{lemma}

\begin{lemma}
(\cite[Lemma~3]{Yi2021b} and \cite[Lemma~1]{Zhang2024b})
Let $\mathbb{K}$ denote a nonempty closed convex subset of ${\mathbb{R}^p}$ and let $b$ and $c$ denote two vectors in ${\mathbb{R}^p}$. If ${x} = {\mathcal{P}_\mathbb{K}}( {b - c} )$, then for all $y \in \mathbb{K}$,
\begin{flalign}
2\langle {x - y,c} \rangle  \le {\| {y - b} \|^2} - {\| {y - x} \|^2} - {\| {x - b} \|^2}. \label{Lemma3-eq1}
\end{flalign}
In addition, if $\Phi ( y ) = {\| {b - y} \|^2} + 2\langle {c,y} \rangle $, then we know $\Phi$ is a strongly convex function with convexity parameter $\sigma  = 2$ and $x = \mathop {\arg \min }\limits_{y \in \mathbb{K}} \Phi ( y )$. Moreover, 
\begin{flalign}
\| {x - b} \| \le {\| c \|} \label{Lemma3-eq2}
\end{flalign}
holds.
\end{lemma}

We provide some properties of the gradient estimators in the following lemma.
\begin{lemma}\label{lem1}
(\cite[Lemma~3]{Yi2023}, \cite[Lemma~2]{Yi2021b} and \cite[Lemma~5]{Tang2020})
If Assumptions~2--3 hold, and ${{{{g}_{i,t}}( x )} }$ is convex on $\mathbb{X}$. Then, ${{{{\hat g}_{i,t}}( x )} }$ is convex on $( {1 - {\xi _t}} )\mathbb{X}$. 
Moreover, for any $i \in [ n ]$, $t \in {\mathbb{N}_ + }$, $x \in ( {1 - {\xi _t}} )\mathbb{X}$, $v \in \mathbb{R}_ + ^{{m_i}}$,
\begin{subequations}
\begin{flalign}
&\nabla {{\hat f}_{i,t}}( x ) = {\mathbf{E}_{{\mathfrak{U}_t}}}[ {\hat \nabla_1 {f_{i,t}}( x )} ] = {\mathbf{E}_{{\mathfrak{U}_t}}}[ {\hat \nabla_2 {f_{i,t}}( x )} ], \label{lemma5-eq1}\\
&\| {\hat \nabla_1 {f_{i,t}}( x )} \| \le \frac{{p{F_1}}}{{{\delta _t}}}, \label{lemma5-eq2} \\
&\| {\hat \nabla_2 {f_{i,t}}( x )} \| \le p{G_1}, \label{lemma5-eq3}\\
&\nabla {{{{\hat g}_{i,t}}( x )} } = {\mathbf{E}_{{\mathfrak{U}_t}}}\big[ {\hat \nabla_1 {{{g_{i,t}}( x )} }} \big] = {\mathbf{E}_{{\mathfrak{U}_t}}}\big[ {\hat \nabla_2 {{{g_{i,t}}( x )} }} \big], \label{lemma5-eq4} \\
&{v^T}{{{g_{i,t}}( x )} } \le {v^T}{{{{\hat g}_{i,t}}( x )} }
\le {v^T}{{{g_{i,t}}( x )} } + {G_2}{\delta _t}\| v \|, \label{lemma5-eq5}\\
&\| {\hat \nabla_1 {{{g_{i,t}}( x )} }} \| \le \frac{{p{F_2}}}{{{\delta _t}}}, \label{lemma5-eq6} \\
&\| {\hat \nabla_2 {{{g_{i,t}}( x )} }} \| \le p{G_2}, \label{lemma5-eq7} \\
&\| {\nabla {{\hat f}_{i,t}}( x ) - \nabla {f_{i,t}}( x )} \| \le L{\delta _t}, \label{lemma5-eq8}
\end{flalign}
\end{subequations}
where ${{\hat f}_{i,t}}( x ) = {\mathbf{E}_{w_t \in {\mathbb{B}^p}}}[ {{f_{i,t}}( {x + {\delta _t}w_t} )} ]$ and ${{{{\hat g}_{i,t}}( x )} } = {\mathbf{E}_{w_t \in {\mathbb{B}^p}}}\big[ {{{{g_{i,t}}( {x + {\delta _t}w_t} )}}} \big]$ with $w_t$ being chosen uniformly at random, 
and ${\mathfrak{U}_t}$ is the $\sigma$-algebra induced by the independent and identically distributed variables ${u_{1,t}}, \cdot  \cdot  \cdot ,{u_{n,t}}$. Let ${\mathcal{U}_t} =  \cup _{s = 1}^t{\mathfrak{U}_s}$. 
For Algorithm~1, it is straightforward to see that $\{ {{{\hat z}_{j,t}}} \}$, $\{ {{x_{i,t}}} \}$, $\{ {{e_{i,t-1}}} \}$, $\{ {{v_{i,t}}} \}$, $\{ {{a _{i,t}}} \}$, and $\{ {{z_{i,t}}} \}$, $i \in [ n ]$  depend on ${\mathcal{U}_{t-1}}$ and are independent of ${\mathfrak{U}_s}$ for all $s \ge t$. 
For Algorithm~2, it is similarly straightforward to see that $\{ {{{\hat z}_{j,t}}} \}$, $\{ {{x_{i,t}}} \}$, $\{ {{b _{i,t}}} \}$, and $\{ {{z_{i,t}}} \}$, $i \in [ n ]$  depend on ${\mathcal{U}_{t-1}}$ and are independent of ${\mathfrak{U}_s}$ for all $s \ge t$. 
Moreover, it is also straightforward to see that $\{ {{v_{i,t}}} \}$ in Algorithm~2 is independent of ${\mathfrak{U}_s}$ for all $s$.
\end{lemma}


\begin{lemma}
If Assumption~4 holds. For all $i,j \in [ n ]$ and $t \in {\mathbb{N}_ + }$, ${\hat{z}_{i,t}}$ generated by Algorithm~1 satisfy
\begin{flalign}
&\| {{\hat{z}_{i,t}} - {{\bar z}_t}} \| \le \tau {\lambda ^{t - 2}}\sum\limits_{j = 1}^n {\| {{\hat{z}_{j,1}}} \|}  + \tau \sum\limits_{s = 1}^{t - 2} {{\lambda ^{t - s - 2}}} \sum\limits_{j = 1}^n {\| {\varepsilon _{j,s}^z} \|}  + \| {\varepsilon _{i,t - 1}^z} \| + \frac{1}{n}\sum\limits_{j = 1}^n {\| {\varepsilon _{j,t - 1}^z} \|}, \label{Lemma5.5-eq1} \\
&\;\;\;\;\;\;\;\;\;\;\;\;\;\;\;\;\;\;\;\;\;\;\;\;\;\;\;\;\;\;\;\;2\sum\limits_{t = 1}^T {\sum\limits_{i = 1}^n {\| {{{\hat z}_{i,t}} - {{\bar z}_t}} \|} }
\le n{\vartheta _1} + {{\vartheta }_2}\sum\limits_{t = 1}^T {\sum\limits_{i = 1}^n {{\| {\varepsilon _{i,t}^z} \|} } }, \label{Lemma6-eq3} 
\end{flalign}
and ${x_{i,t}}$ generated by Algorithm~1 satisfy
\begin{flalign}
\frac{1}{n}\sum\limits_{t = 1}^T {\sum\limits_{i = 1}^n {\sum\limits_{j = 1}^n {{\| {{x_{i,t}} - {x_{j,t}}} \|} } } }
&\le n{\vartheta _1} + {{\vartheta }_2}\sum\limits_{t = 1}^T {\sum\limits_{i = 1}^n {{\| {\varepsilon _{i,t}^z} \|} } }, \label{Lemma6-eq1} \\
\frac{1}{n}\sum\limits_{t = 1}^T {\sum\limits_{i = 1}^n {\sum\limits_{j = 1}^n {{{{\| {{x_{i,t}} - {x_{j,t}}} \|}^2}} } } } &\le {{ \vartheta }_3} + {{ \vartheta }_4}\sum\limits_{t = 1}^T {\sum\limits_{i = 1}^n {{{\| {\varepsilon _{i,t}^z} \|}^2}} }, \label{Lemma6-eq2}
\end{flalign}
where ${{\bar z}_t} = \frac{1}{n}\sum\nolimits_{i = 1}^n {{{\hat z}_{i,t}}} $, $\varepsilon _{i,t}^z = {\hat{z}_{i,t+1}} - {x_{i,t}}$, ${\vartheta _1} = \frac{{2\tau }}{{\lambda ( {1 - \lambda } )}}\sum\limits_{j = 1}^n {\| {{{\hat z}_{j,1}}} \|}$, ${{ \vartheta }_2} = 4 + \frac{{2n\tau }}{{1 - \lambda }}$, ${\vartheta _3} = \frac{{16n{\tau ^2}}}{{{\lambda ^2}( {1 - {\lambda ^2}} )}}{\Big( {\sum\limits_{i = 1}^n {\| {{{\hat z}_{i,1}}} \|} } \Big)^2}$, 
and ${\vartheta _4} = \frac{{16{n^2}{\tau ^2}}}{{{{( {1 - \lambda } )}^2}}} + 32$.
\end{lemma}
\begin{proof}
From \eqref{Algorithm1-eq2}, we have
\begin{flalign}
\hat z_{i,t} = \sum\limits_{j = 1}^n {{{[ {{W_{t - 1}}} ]}_{ij}}\hat z_{j,t - 1}^i}  + \varepsilon _{i,{t-1}}^z. \label{Lemma5.5-proof-eq1}
\end{flalign}
Then, noting that ${\hat{z}_{j,t}^i}={\hat{z}_{j,t}}$, and following the proof of Lemma 4 in \cite{Yi2023}, we know that \eqref{Lemma5.5-eq1} holds.

From \eqref{Lemma5.5-eq1}, we know that \eqref{Lemma6-eq3}--\eqref{Lemma6-eq2} hold by following the proof of Lemma 1 in \cite{Yi2024}.
\end{proof}

We analyze local regret at one iteration in the following lemma.
\begin{lemma}
Suppose Assumptions~1--2 and 4--5 hold. For all $i \in [ n ]$, let $\{ {{x_{i,t}}} \}$ be the sequences generated by Algorithm~1, $y$ be an arbitrary point in $\mathbb{X}$, 
and $\hat y = ( {1 - {\xi _t}} )y$, then
\begin{flalign}
\nonumber
&\;\;\;\;\;\frac{1}{n}\sum\limits_{i = 1}^n {{\mathbf{E}}[ {v_{i,t + 1}^T{g_{i,t}}( {{x_{i,t}}} )} ]} 
+ \frac{1}{n}\sum\limits_{i = 1}^n {{\mathbf{E}}[ {\langle {\nabla {f_{i,t}}( {{x_{i,t}}} ),{x_{i,t}} - y} \rangle } ]} \\
\nonumber
& \le \frac{1}{n}\sum\limits_{i = 1}^n {{\mathbf{E}}[ {v_{i,t + 1}^T{g_{i,t}}( y )} ]}  
+ \frac{1}{n}\sum\limits_{i = 1}^n {{{\mathbf{E}}[ {{\Delta _{i,t}}( {\hat y} )} ]} }  
+ \frac{1}{n}{{\mathbf{E}}[ {{{\bar \Delta }_t}} ]} \\
&\;\; + \frac{1}{n}\sum\limits_{i = 1}^n {{G_2}\big( {R( \mathbb{X} ){\xi _t} + {\delta _t}} \big){\mathbf{E}}[ {\| {{v_{i,t + 1}}} \|} ]} 
+ {G_1}R( \mathbb{X} ){\xi _t} + 2LR( \mathbb{X} ){\delta _t} + 2\sqrt C R( \mathbb{X} )\frac{{{s_t}}}{{{\alpha _t}}}, \label{Lemma7-eq1}
\end{flalign}
where 
\begin{flalign}
\nonumber
&{\Delta _{i,t}}( \hat{y} ) = \frac{1}{{2{\alpha _t}}}( {{{\| {\hat{y} - {x_{i,t}}} \|}^2} - {{\| {\hat{y} - {x_{i,t + 1}}} \|}^2}} ), \\
\nonumber
&{{ \bar{\Delta} }_t} = \sum\limits_{i = 1}^n {\Big( {\frac{{p{F_1}}}{{{\delta _t}}} + \frac{{p{F_2}}}{{{\delta _t}}}\| {{v_{i,t + 1}}} \|} \Big)\| {x_{i,t}} - z_{i,t + 1} \|}  
- \sum\limits_{i = 1}^n {\frac{{{{\| {x_{i,t}} - z_{i,t + 1} \|}^2}}}{{2{\alpha _t}}}}.
\end{flalign}
\end{lemma}
\begin{proof}
Due to $\hat y = ( {1 - {\xi _t}} )y$, for any $t \in {\mathbb{N}_ + }$, we have
\begin{flalign}
\nonumber
v_{i,t + 1}^T{{\hat g}_{i,t}}( {\hat y} ) &\le v_{i,t + 1}^T{g_{i,t}}( {\hat y} ) + {G_2}{\delta _t}\| {{v_{i,t + 1}}} \| \\
\nonumber
& = v_{i,t + 1}^T\big( {{g_{i,t}}( y ) + {g_{i,t}}( {\hat y} ) - {g_{i,t}}( y )} \big) + {G_2}{\delta _t}\| {{v_{i,t + 1}}} \| \\
\nonumber
& \le v_{i,t + 1}^T{g_{i,t}}( y ) + {G_2}\| {\hat y - y} \|\| {{v_{i,t + 1}}} \| + {G_2}{\delta _t}\| {{v_{i,t + 1}}} \| \\
& \le v_{i,t + 1}^T{g_{i,t}}( y ) + {G_2}\big( {R( \mathbb{X} ){\xi _t} + {\delta _t}} \big)\| {{v_{i,t + 1}}} \|, \label{lemma7-proof-eq1}
\end{flalign}
where the first inequality holds due to \eqref{lemma5-eq5}, the second inequality holds due to \eqref{ass4-eq2}, and the last inequality holds due to \eqref{ass1-eq1}.
\begin{flalign}
\nonumber
\langle {\nabla {{\hat f}_{i,t}}( {{x_{i,t}}} ),{x_{i,t}} - \hat y} \rangle
\nonumber
&  = \langle {{\mathbf{E}_{{\mathfrak{U}_t}}}[ {\hat \nabla_1 {f_{i,t}}( {{x_{i,t}}} )} ],{x_{i,t}} - \hat y} \rangle \\
\nonumber
&  = {\mathbf{E}_{{\mathfrak{U}_t}}}[ {\langle {\hat \nabla_1 {f_{i,t}}( {{x_{i,t}}} ),{x_{i,t}} - \hat y} \rangle } ] \\
\nonumber
&  = {\mathbf{E}_{{\mathfrak{U}_t}}}[ {\langle {\hat \nabla_1 {f_{i,t}}( {{x_{i,t}}} ),{x_{i,t}} - z_{i,t + 1}} \rangle  + \langle {\hat \nabla_1 {f_{i,t}}( {{x_{i,t}}} ),{z_{i,t + 1}} - \hat y} \rangle } ] \\
& \le {\mathbf{E}_{{\mathfrak{U}_t}}}[ {\frac{{p{F_1}}}{{{\delta _t}}}\| {x_{i,t} - z_{i,t + 1}} \| + \langle {\hat \nabla_1 {f_{i,t}}( {{x_{i,t}}} ),{z_{i,t + 1}} - \hat y} \rangle } ], \label{lemma7-proof-eq2}
\end{flalign}
where the first equality holds due to \eqref{lemma5-eq1}, the second equality holds since ${x_{i,t}}$ and $\hat y$ are independent of ${\mathfrak{U}_t}$, and the last inequality holds due to the Cauchy-Schwarz inequality and \eqref{lemma5-eq2}.

From \eqref{Algorithm1-eq5}, we have
\begin{flalign}
\nonumber
\langle {\hat \nabla_1 {f_{i,t}}( {{x_{i,t}}} ),{z_{i,t + 1}} - \hat y} \rangle
&  = \langle {\hat \nabla_1 {g_{i,t}}( {{x_{i,t}}} ){v_{i,t + 1}},\hat y - {z_{i,t + 1}}} \rangle + \langle {{{ a }_{i,t + 1}},{z_{i,t + 1}} - \hat y} \rangle \\
\nonumber
&  = \langle {\hat \nabla_1 {g_{i,t}}( {{x_{i,t}}} ){v_{i,t + 1}},\hat y - {x_{i,t}}} \rangle + \langle {\hat \nabla_1 {g_{i,t}}( {{x_{i,t}}} ){v_{i,t + 1}},{x_{i,t}} - {z_{i,t + 1}}} \rangle \\
&\;\;\;\;  + \langle {{{ a }_{i,t + 1}},{z_{i,t + 1}} - \hat y} \rangle. \label{lemma7-proof-eq3}
\end{flalign}
Since ${\hat{g}_{i,t}}$ is convex on $( {1 - {\xi _t}} )\mathbb{X}$, we have
\begin{flalign}
{g_{i,t}}( \hat{y} ) \ge {g_{i,t}}( x ) + \nabla {g_{i,t}}( x )( {\hat{y} - x} ), \forall x,\hat{y} \in ( {1 - {\xi _t}} )\mathbb{X}. \label{Lemma7-proof-eq4}
\end{flalign}
We have
\begin{flalign}
\nonumber
{\mathbf{E}_{{\mathfrak{U}_t}}}[ {\langle {\hat \nabla_1 {g_{i,t}}( {{x_{i,t}}} ){v_{i,t + 1}},\hat y - {x_{i,t}}} \rangle } ]
&  = \langle {{\mathbf{E}_{{\mathfrak{U}_t}}}[ {\hat \nabla_1 {g_{i,t}}( {{x_{i,t}}} )} ]{v_{i,t + 1}},\hat y - {x_{i,t}}} \rangle  \\
\nonumber
&  = \big\langle {\big(\nabla {{\hat g}_{i,t}}( {{x_{i,t}}} )\big)^T{v_{i,t + 1}},\hat y - {x_{i,t}}} \big\rangle \\
\nonumber
& \le v_{i,t + 1}^T{{\hat g}_{i,t}}( {\hat y} ) - v_{i,t + 1}^T{{\hat g}_{i,t}}( {{x_{i,t}}} ) \\
& \le v_{i,t + 1}^T{g_{i,t}}( y ) - v_{i,t + 1}^T{g_{i,t}}( {{x_{i,t}}} ) + {G_2}\big( {R( \mathbb{X} ){\xi _t} + {\delta _t}} \big)\| {{v_{i,t + 1}}} \|, \label{lemma7-proof-eq5}
\end{flalign}
where the first equality holds since $v_{i,t + 1}$, $\hat y$, and ${x_{i,t}}$ are independent of ${\mathfrak{U}_t}$, the last equality holds due to \eqref{lemma5-eq4}, 
the first inequality holds due to ${v_{i,t+1}} \ge {\mathbf{0}_{{m_i}}}$, ${x_{i,t}},\hat y \in ( {1 - {\xi _t}} )\mathbb{X}$ and \eqref{Lemma7-proof-eq4}, 
and the last inequality holds due to \eqref{lemma5-eq5} and \eqref{lemma7-proof-eq1}.

From the Cauchy-Schwarz inequality and \eqref{lemma5-eq6}, we have
\begin{flalign}
\big\langle {{ {\hat{\nabla}_1 {g_{i,t}}( {{x_{i,t}}} )} }{v_{i,t + 1}},{x_{i,t}} - {z_{i,t + 1}}} \big\rangle  \le \frac{{p{F_2}}}{{{\delta _t}}}\| {{v_{i,t + 1}}} \|\| {x_{i,t}} - {z_{i,t + 1}} \|. \label{lemma7-proof-eq6}
\end{flalign}
By applying \eqref{Lemma3-eq1} to the update \eqref{Algorithm1-eq6}, for all $i \in [ n ]$, we have
\begin{flalign}
\nonumber
&\;\;\;\;\;\langle {{a_{i,t + 1}},{z_{i,t + 1}} - \hat y} \rangle \\
\nonumber
& \le \frac{1}{{2{\alpha _t}}}\big( {{{\| {\hat y - {x_{i,t}}} \|}^2} - {{\| {\hat y - {z_{i,t + 1}}} \|}^2} - {{\| {{x_{i,t}} - {z_{i,t + 1}}} \|}^2}} \big) \\
\nonumber
& = \frac{1}{{2{\alpha _t}}}\big( {{{\| {\hat y - {x_{i,t}}} \|}^2} - {{\| {\hat y - {x_{i,t + 1}}} \|}^2} + {{\| {\hat y - {x_{i,t + 1}}} \|}^2} - {{\| {\hat y - {z_{i,t + 1}}} \|}^2} - {{\| {{x_{i,t}} - {z_{i,t + 1}}} \|}^2}} \big) \\
\nonumber
&  = {\Delta _{i,t}}( {\hat y} ) + \frac{1}{{2{\alpha _t}}}\big( {{{\| {\hat y - {x_{i,t + 1}}} \|}^2} - {{\| {\hat y - {z_{i,t + 1}}} \|}^2} - {{\| {{x_{i,t}} - {z_{i,t + 1}}} \|}^2}} \big) \\
\nonumber
& = {\Delta _{i,t}}( {\hat y} ) + \frac{1}{{2{\alpha _t}}}\Big( {{{\Big\| {\hat y - \sum\limits_{j = 1}^n {{{[ {{W_{t + 1}}} ]}_{ij}}{{\hat z}_{j,t + 1}^i}} } \Big\|^2}} - {{\| {\hat y - {z_{i,t + 1}}} \|}^2} - {{\| {{x_{i,t}} - {z_{i,t + 1}}} \|}^2}} \Big) \\
& \le {\Delta _{i,t}}( {\hat y} ) + \frac{1}{{2{\alpha _t}}}\Big( {\sum\limits_{j = 1}^n {{{[ {{W_{t + 1}}} ]}_{ij}}{{\| {\hat y - {{\hat z}_{j,t + 1}}} \|}^2}}  - {{\| {\hat y - {z_{i,t + 1}}} \|}^2} - {{\| {{x_{i,t}} - {z_{i,t + 1}}} \|}^2}} \Big) \label{Lemma7-proof-eq7}
\end{flalign}
where the last equality holds due to \eqref{Algorithm1-eq2}, and the last inequality holds since ${\hat{z}_{j,t}^i}={\hat{z}_{j,t}}$, ${W_{t + 1}}$ is doubly stochastic and ${\|  \cdot  \|^2}$ is convex.

We have
\begin{flalign}
\nonumber
&\;\;\;\;\; {{\mathbf{E}}[ {{{\| {\hat y - {{\hat z}_{i,t + 1}}} \|}^2}} ]} - {{\mathbf{E}}[ {{{\| {\hat y - {z_{i,t + 1}}} \|}^2}} ]}\\
\nonumber
& = {{\mathbf{E}}\big[ {{{\| {\hat y - {{\hat z}_{i,t + 1}}} \|}^2} - {{\| {\hat y - {z_{i,t + 1}}} \|}^2}} ]} \\
\nonumber
& = {{\mathbf{E}}[ {\langle {\hat y - {{\hat z}_{i,t + 1}} + \hat y - {z_{i,t + 1}},\hat y - {{\hat z}_{i,t + 1}} - \hat y + {z_{i,t + 1}}} \rangle } ]} \\
\nonumber
& \le {{\mathbf{E}}[ {\| {\hat y - {{\hat z}_{i,t + 1}} + \hat y - {z_{i,t + 1}}} \|\| {{z_{i,t + 1}} - {{\hat z}_{i,t + 1}}} \|} ]}\\
& \le 4{R}( \mathbb{X} ){{\mathbf{E}}[ {\| {{z_{i,t + 1}} - {{\hat z}_{i,t + 1}}} \|} ]}, \label{Lemma7-proof-eq8}
\end{flalign}
where the first inequality holds due to the Cauchy–Schwarz inequality, and the last inequality holds due to \eqref{ass1-eq1}.

From \eqref{Lemma7-proof-eq9}, noting that $C > 0$, we have
\begin{flalign}
{{\mathbf{E}}[ {\| {{z_{i,t + 1}} - {\hat{z}_{i,t + 1}}} \|} ]}
& \le {\sqrt C} s_{t + 1}. \label{Lemma7-proof-eq10}
\end{flalign}

We have
\begin{flalign}
\nonumber
\langle \nabla {f_{i,t}}({x_{i,t}}),{x_{i,t}} - \hat y\rangle
& = \langle \nabla {{\hat f}_{i,t}}({x_{i,t}}),{x_{i,t}} - \hat y\rangle  + \langle \nabla {f_{i,t}}({x_{i,t}}) - \nabla {{\hat f}_{i,t}}({x_{i,t}}),{x_{i,t}} - \hat y\rangle \\
\nonumber
& \le \langle \nabla {{\hat f}_{i,t}}({x_{i,t}}),{x_{i,t}} - \hat y\rangle  + \| {\nabla {f_{i,t}}({x_{i,t}}) - \nabla {{\hat f}_{i,t}}({x_{i,t}})} \|\| {{x_{i,t}} - \hat y} \| \\
& \le \langle \nabla {{\hat f}_{i,t}}({x_{i,t}}),{x_{i,t}} - \hat y\rangle  + 2LR( \mathbb{X} ){\delta _t}, \label{lemma7-proof-eq11}
\end{flalign}
where the first inequality holds due to the Cauchy-Schwarz inequality, and the last inequality holds due to \eqref{lemma5-eq8} and $\hat y = ( {1 - {\xi _t}} )y$, $y \in \mathbb{X}$, and \eqref{ass1-eq1}.

We have
\begin{flalign}
\nonumber
\langle \nabla {f_{i,t}}({x_{i,t}}),{x_{i,t}} - y\rangle
&  = \langle \nabla {f_{i,t}}({x_{i,t}}),{x_{i,t}} - \hat y\rangle  + \langle \nabla {f_{i,t}}({x_{i,t}}),\hat y - y\rangle \\
\nonumber
& \le \langle \nabla {f_{i,t}}({x_{i,t}}),{x_{i,t}} - \hat y\rangle  + \| {\nabla {f_{i,t}}({x_{i,t}})} \|\| {\hat y - y} \| \\
& \le \langle \nabla {f_{i,t}}({x_{i,t}}),{x_{i,t}} - \hat y\rangle  + {G_1}R( \mathbb{X} ){\xi _t}, \label{lemma7-proof-eq12}
\end{flalign}
where the first inequality holds due to the Cauchy--Schwarz inequality, and the last inequality holds due to \eqref{ass4-eq1a} and $\hat y = ( {1 - {\xi _t}} )y$, $y \in \mathbb{X}$, and \eqref{ass1-eq1}.

Combining \eqref{lemma7-proof-eq2}--\eqref{lemma7-proof-eq3}, \eqref{lemma7-proof-eq5}--\eqref{Lemma7-proof-eq8}, and \eqref{Lemma7-proof-eq10}--\eqref{lemma7-proof-eq12}, summing over $i \in [ n ]$, dividing by $n$,
taking expectation, using $\sum\nolimits_{i = 1}^n {{{[ {{W_t}} ]}_{ij}}}  = 1$, $\forall t \in {\mathbb{N}_ + }$, and noting that $\{ {{s_t}} \}$ is decreasing, 
we rearrange the terms to obtain \eqref{Lemma7-eq1}. 
\end{proof}

We analyze local regret and squared cumulative constraint violation over $T$ iterations in the following lemma.
\begin{lemma}
Suppose Assumptions~1--2 and 4--5 hold. For all $i \in [ n ]$, let $\{ {{x_{i,t}}} \}$ be the sequences generated by Algorithm~1 
with ${\alpha _t}{\gamma _t} \le \delta _t^2/2{p^2}F_2^2$, ${\alpha _t}{\gamma _t}$ be nonincreasing, $y$ be an arbitrary point in $\mathbb{X}$, and $\hat y = ( {1 - {\xi _t}} )y$. 
Then, for any $T \in {\mathbb{N}_ + }$,
\begin{subequations}
\begin{flalign}
\nonumber
&\;\;\;\;\;\frac{1}{n}\sum\limits_{i = 1}^n {\sum\limits_{t = 1}^T {{\mathbf{E}}[ {\langle {\nabla {f_{i,t}}( {{x_{i,t}}} ),{x_{i,t}} - y} \rangle } ]} } \\
& \le  - \frac{1}{n}\sum\limits_{i = 1}^n {\sum\limits_{t = 1}^T {\frac{{\mathbf{E}[ {{{\| {{x_{i,t}} - {z_{i,t + 1}}} \|}^2}} ]}}{{4{\alpha _t}}}} }  
+ \frac{1}{n}\sum\limits_{i = 1}^n {\sum\limits_{t = 1}^T {{{\mathbf{E}}[ {{\Delta _{i,t}}( {\hat y} )} ]}} }  
+ {\Phi _T}, \forall y \in {\mathcal{X}_T}, \label{Lemma8-eq1} \\
&\;\;\;\;\;\sum\limits_{i = 1}^n {\sum\limits_{t = 1}^T {{\mathbf{E}}\big[ {{{\| {{{[ {{g_{i,t}}( {{x_{i,t}}} )} ]}_ + }} \|}^2}} \big]} } 
 \le 2{\mathbf{E}}[ {{\Lambda _T}( y )} ] + {\Psi _T}, \label{Lemma8-eq2}
\end{flalign}
\end{subequations}
where
\begin{flalign}
\nonumber
&{\Phi _T} = 2{p^2}F_1^2\sum\limits_{t = 1}^T {\frac{{{\alpha _t}}}{{\delta _t^2}}}   
+ {F_2}{G_2}R( \mathbb{X} )\sum\limits_{t = 1}^T {{\gamma _t}{\xi _t}}  + 2{F_2}{G_2}\sum\limits_{t = 1}^T {{\gamma _t}{\delta _t}} + {G_1}R( \mathbb{X} )\sum\limits_{t = 1}^T {{\xi _t}} + 2LR( \mathbb{X} )\sum\limits_{t = 1}^T {{\delta _t}} \\
\nonumber
&\;\;\;\;\;\; + 2\sqrt C R( \mathbb{X} )\sum\limits_{t = 1}^T {\frac{{{s_t}}}{{{\alpha _t}}}},\\
\nonumber
&{\Lambda _T}( y ) = \sum\limits_{i = 1}^n {\sum\limits_{t = 1}^T {\frac{{v_{i,t + 1}^T{g_{i,t}}( y )}}{{{\gamma _t}}}} }, \\
\nonumber
&{\Psi _T} = \frac{{4nR{{( \mathbb{X} )}^2}}}{{{\alpha _T}{\gamma _T}}} + 2n{p^2}F_1^2\sum\limits_{t = 1}^T {\frac{{{\alpha _t}}}{{{\gamma _t}\delta _t^2}}}  
+ 4n{G_1}R( \mathbb{X} )\sum\limits_{t = 1}^T {\frac{1}{{{\gamma _t}}}}  
+ 2n{G_1}R( \mathbb{X} )\sum\limits_{t = 1}^T {\frac{{{\xi _t}}}{{{\gamma _t}}}}  + 4nLR( \mathbb{X} )\sum\limits_{t = 1}^T {\frac{{{\delta _t}}}{{{\gamma _t}}}} \\
\nonumber
&\;\;\;\;\;\; + 2n{F_2}{G_2}R( \mathbb{X} )\sum\limits_{t = 1}^T {{\xi _t}}  
+ 4n{F_2}{G_2}\sum\limits_{t = 1}^T {{\delta _t}} 
 + 4n\sqrt C R( \mathbb{X} )\sum\limits_{t = 1}^T {\frac{{{s_t}}}{{{\alpha _t}{\gamma _t}}}}.
\end{flalign}
\end{lemma}
\begin{proof}
(i)
Since ${g_{i,t}}( y ) \le {\mathbf{0}_{{m_i}}}$, $\forall i \in [ n ]$, $\forall t \in {\mathbb{N}_ + }$ when $ y \in {\mathcal{X}_T}$, summing \eqref{Lemma7-eq1} over $t \in [ T ]$ gives
\begin{flalign}
\nonumber
&\;\;\;\;\;\frac{1}{n}\sum\limits_{i = 1}^n {\sum\limits_{t = 1}^T {{\mathbf{E}}[ {\langle {\nabla {f_{i,t}}( {{x_{i,t}}} ),{x_{i,t}} - y} \rangle } ]} } \\
\nonumber
& \le \frac{1}{n}\sum\limits_{i = 1}^n {\sum\limits_{t = 1}^T {{\mathbf{E}}[ { - v_{i,t + 1}^T{g_{i,t}}( {{x_{i,t}}} ) 
+ \frac{{{{\bar \Delta }_t}}}{n}  } ]} }  
+ \frac{1}{n}\sum\limits_{i = 1}^n {\sum\limits_{t = 1}^T {{{\mathbf{E}}[ {{\Delta _{i,t}}( {\hat y} )} ]}} } \\
\nonumber
& \;\;+ \frac{1}{n}\sum\limits_{i = 1}^n {\sum\limits_{t = 1}^T {{G_2}\big( {R( \mathbb{X} ){\xi _t} + {\delta _t}} \big){\mathbf{E}}[ {\| {{v_{i,t + 1}}} \|} ]} }  
+ {G_1}R( \mathbb{X} )\sum\limits_{t = 1}^T {{\xi _t}}  
+ 2LR( \mathbb{X} )\sum\limits_{t = 1}^T {{\delta _t}}  \\
&\;\; + 2\sqrt C R( \mathbb{X} )\sum\limits_{t = 1}^T {\frac{{{s_t}}}{{{\alpha _t}}}}. \label{Lemma8-proof-eq1}
\end{flalign}
We have
\begin{flalign}
\Big( {\frac{{p{F_1}}}{{{\delta _t}}} + \frac{{p{F_2}}}{{{\delta _t}}}\| {{v_{i,t + 1}}} \|} \Big)\| {{x_{i,t}} - {z_{i,t + 1}}} \| \le \frac{{2{p^2}F_1^2{\alpha _t}}}{{\delta _t^2}} 
+ \frac{{2{p^2}F_2^2{\alpha _t}}}{{\delta _t^2}}{\| {{v_{i,t + 1}}} \|^2} + \frac{{{{\| {{x_{i,t}} - {z_{i,t + 1}}} \|}^2}}}{{4{\alpha _t}}}. \label{Lemma8-proof-eq2}
\end{flalign}
Let ${e_{i,t}} = {x_{i,t}} + {\delta _t}{u_{i,t}}$. From \eqref{Algorithm1-eq4}, for all $\forall t \in {\mathbb{N}_ + }$, we have
\begin{flalign}
\| {{v_{i,t + 1}}} \| = {\gamma _t}\| {{{[ {{g_{i,t}}( {{e_{i,t}}} )} ]}_ + }} \|. \label{Lemma8-proof-eq3}
\end{flalign}
From \eqref{Lemma8-proof-eq3} and \eqref{ass2-eq1b}, we have
\begin{flalign}
\| {{v_{i,t + 1}}} \| \le {F_2}{\gamma _t}. \label{Lemma8-proof-eq4}
\end{flalign}
From \eqref{Lemma8-proof-eq3} and \eqref{Algorithm1-eq4}, we have
\begin{flalign}
\nonumber
&\;\;\;\;\;\frac{{2{p^2}F_2^2{\alpha _t}}}{{\delta _t^2}}{\| {{v_{i,t + 1}}} \|^2} - v_{i,t + 1}^T{g_{i,t}}( {{x_{i,t}}} ) \\
\nonumber
& = \frac{{2{p^2}F_2^2{\alpha _t}\gamma _t^2}}{{\delta _t^2}}{\| {{{[ {{g_{i,t}}( {{e_{i,t}}} )} ]}_ + }} \|^2} - {\gamma _t}{\big( {{{[ {{g_{i,t}}( {{e_{i,t}}} )} ]}_ + }} \big)^T}{g_{i,t}}( {{x_{i,t}}} ) \\
\nonumber
& = \frac{{2{p^2}F_2^2{\alpha _t}\gamma _t^2}}{{\delta _t^2}}{\| {{{[ {{g_{i,t}}( {{e_{i,t}}} )} ]}_ + }} \|^2} - {\gamma _t}{\big( {{{[ {{g_{i,t}}( {{e_{i,t}}} )} ]}_ + }} \big)^T}\big( {{g_{i,t}}( {{e_{i,t}}} ) + {g_{i,t}}( {{x_{i,t}}} ) - {g_{i,t}}( {{e_{i,t}}} )} \big) \\
\nonumber
& \le \Big( {\frac{{2{p^2}F_2^2{\alpha _t}{\gamma _t}}}{{\delta _t^2}} - 1} \Big){\gamma _t}{\| {{{[ {{g_{i,t}}( {{e_{i,t}}} )} ]}_ + }} \|^2} + {\gamma _t}\| {{{[ {{g_{i,t}}( {{e_{i,t}}} )} ]}_ + }} \|\| {{g_{i,t}}( {{x_{i,t}}} ) - {g_{i,t}}( {{e_{i,t}}} )} \|\\
& \le \Big( {\frac{{2{p^2}F_2^2{\alpha _t}{\gamma _t}}}{{\delta _t^2}} - 1} \Big){\gamma _t}{\| {{{[ {{g_{i,t}}( {{e_{i,t}}} )} ]}_ + }} \|^2} + {F_2}{G_2}{\gamma _t}{\delta _t}, \label{Lemma8-proof-eq5}
\end{flalign}
where the first inequality holds due to the fact that ${\varphi ^T}{[ \varphi  ]_ + } = {\| {{{[ \varphi  ]}_ + }} \|^2}$ for any vector $\varphi $ and the Cauchy-Schwarz inequality, 
and the last inequality holds due to $\| {{u_{i,t}}} \| = 1$, $\forall i \in [ n ]$, $\forall t \in {\mathbb{N}_ + }$.

Due to ${\alpha _t}{\gamma _t} \le \delta _t^2/2{p^2}F_2^2$, we have
\begin{flalign}
\Big( {\frac{{2{p^2}F_2^2{\alpha _t}{\gamma _t}}}{{\delta _t^2}} - 1} \Big){\gamma _t}{\| {{{[ {{g_{i,t}}( {{e_{i,t}}} )} ]}_ + }} \|^2} \le 0. \label{Lemma8-proof-eq6}
\end{flalign}
Combining \eqref{Lemma8-proof-eq1}--\eqref{Lemma8-proof-eq2} and \eqref{Lemma8-proof-eq4}--\eqref{Lemma8-proof-eq6} yields \eqref{Lemma8-eq1}.

(ii)
From the Cauchy--Schwarz inequality and \eqref{ass1-eq1}, we have
\begin{flalign}
\langle {\nabla {f_{i,t}}( {{x_{i,t}}} ),y - {x_{i,t}}} \rangle  \le \| {\nabla {f_{i,t}}( {{x_{i,t}}} )} \|\| {y - {x_{i,t}}} \| \le 2{G_1}R( \mathbb{X} ), \forall y \in \mathbb{X}. \label{Lemma8-proof-ii-eq1}
\end{flalign}
Dividing \eqref{Lemma7-eq1} by ${\gamma _t}$, using \eqref{Lemma8-proof-ii-eq1}, and summing over $t \in [ T ]$ yields
\begin{flalign}
\nonumber
&\;\;\;\;\;\sum\limits_{i = 1}^n {\sum\limits_{t = 1}^T {\frac{{{\mathbf{E}}[ {v_{i,t + 1}^T{g_{i,t}}( {{x_{i,t}}} )} ]}}{{{\gamma _t}}}} } \\
\nonumber
& \le {\mathbf{E}}[ {{\Lambda _T}( y )} ] + {2n{G_1}R( \mathbb{X} )}\sum\limits_{t = 1}^T {\frac{1}{{{\gamma _t}}}}  
 + \sum\limits_{i = 1}^n {\sum\limits_{t = 1}^T {\frac{{{{\mathbf{E}}[ {{\Delta _{i,t}}( {\hat y} )} ]}}}{{{\gamma _t}}}} }  
 + \sum\limits_{t = 1}^T {\frac{{{{\mathbf{E}}[ {{{\bar \Delta }_t}} ]}}}{{{\gamma _t}}}}  \\
\nonumber
&\;\; + \sum\limits_{i = 1}^n {\sum\limits_{t = 1}^T {\frac{{{G_2}\big( {R( \mathbb{X} ){\xi _t} + {\delta _t}} \big){\mathbf{E}}[ {\| {{v_{i,t + 1}}} \|} ]}}{{{\gamma _t}}}} } 
      + n{G_1}R( \mathbb{X} )\sum\limits_{t = 1}^T {\frac{{{\xi _t}}}{{{\gamma _t}}}}  + 2nLR( \mathbb{X} )\sum\limits_{t = 1}^T {\frac{{{\delta _t}}}{{{\gamma _t}}}} \\
&\;\;+ 2n\sqrt C R( \mathbb{X} )\sum\limits_{t = 1}^T {\frac{{{s_t}}}{{{\alpha _t}{\gamma _t}}}}. \label{Lemma8-proof-ii-eq2}
\end{flalign}
We have
\begin{flalign}
\Big( {\frac{{p{F_1}}}{{{\gamma _t}{\delta _t}}} + \frac{{p{F_2}}}{{{\gamma _t}{\delta _t}}}\| {{v_{i,t + 1}}} \|} \Big)\| {{x_{i,t}} - {z_{i,t + 1}}} \| 
\le \frac{{{p^2}F_1^2{\alpha _t}}}{{{\gamma _t}\delta _t^2}} + \frac{{{p^2}F_2^2{\alpha _t}}}{{{\gamma _t}\delta _t^2}}{\| {{v_{i,t + 1}}} \|^2} 
+ \frac{{{{\| {{x_{i,t}} - {z_{i,t + 1}}} \|}^2}}}{{2{\alpha _t}{\gamma _t}}}. \label{Lemma8-proof-ii-eq3}
\end{flalign}
Since ${\alpha _t}{\gamma _t}$ is nonincreasing, from \eqref{ass1-eq1}, we have
\begin{flalign}
\sum\limits_{t = 1}^T {\frac{{{\Delta _{i,t}}( {\hat y} )}}{{{\gamma _t}}}}  
= \sum\limits_{t = 1}^T {\frac{1}{{2{\alpha _t}{\gamma _t}}}} \big( {{{\| {\hat y - {x_{i,t}}} \|}^2} - {{\| {\hat y - {x_{i,t + 1}}} \|}^2}} \big) 
\le \frac{{{{\| {\hat y - {x_{i,1}}} \|}^2}}}{{2{\alpha _T}{\gamma _T}}} \le \frac{{2R{{( \mathbb{X} )}^2}}}{{{\alpha _T}{\gamma _T}}}. \label{Lemma8-proof-ii-eq4}
\end{flalign}

From \eqref{Algorithm1-eq4} and \eqref{Lemma8-proof-eq5}, we have
\begin{flalign}
\nonumber
&\;\;\;\;\;\sum\limits_{i = 1}^n {\sum\limits_{t = 1}^T {\frac{{{\mathbf{E}}[ {v_{i,t + 1}^T{g_{i,t}}( {{x_{i,t}}} )} ]}}{{2{\gamma _t}}}} } \\
\nonumber
& = \frac{1}{2}\sum\limits_{i = 1}^n {\sum\limits_{t = 1}^T {{\mathbf{E}}\big[ {{{\big( {{{[ {{g_{i,t}}( {{e_{i,t}}} )} ]}_ + }} \big)}^T}{g_{i,t}}( {{x_{i,t}}} )} \big]} } \\
\nonumber
& = \frac{1}{2}\sum\limits_{i = 1}^n {\sum\limits_{t = 1}^T {{\mathbf{E}}\big[ {{{\big( {{{[ {{g_{i,t}}( {{e_{i,t}}} )} ]}_ + } 
- {{[ {{g_{i,t}}( {{x_{i,t}}} )} ]}_ + } + {{[ {{g_{i,t}}( {{x_{i,t}}} )} ]}_ + }} \big)}^T}{g_{i,t}}( {{x_{i,t}}} )} \big]} } \\
& \le \frac{1}{2}\sum\limits_{i = 1}^n {\sum\limits_{t = 1}^T {{\mathbf{E}}\big[ {{{\| {{{[ {{g_{i,t}}( {{x_{i,t}}} )} ]}_ + }} \|}^2}} \big]} } 
- \frac{1}{2}n{F_2}{G_2}\sum\limits_{t = 1}^T {{\delta _t}}. \label{Lemma8-proof-ii-eq6}
\end{flalign}

Combining \eqref{Lemma8-proof-eq5}, \eqref{Lemma8-proof-ii-eq2}--\eqref{Lemma8-proof-ii-eq6} and \eqref{Lemma8-proof-eq4}, and noting that ${\alpha _t}{\gamma _t} \le \delta _t^2/2{p^2}F_2^2$, we have \eqref{Lemma8-eq2}.
\end{proof}

We analyze network regret and cumulative constraint violation over $T$ iterations in the following lemma.
\begin{lemma}
Under the same conditions as stated in Lemma~6, and supposing that Assumption~3 holds, for any $T \in {\mathbb{N}_ + }$, it holds that
\begin{subequations}
\begin{flalign}
&\frac{1}{n}\sum\limits_{i = 1}^n {\sum\limits_{t = 1}^T {{\mathbf{E}}[ {\langle {\nabla {f_t}( {{x_{i,t}}} ),{x_{i,t}} - y} \rangle } ]} }
\le \frac{1}{n}\sum\limits_{i = 1}^n {\sum\limits_{t = 1}^T {{{\mathbf{E}}[ {{\Delta _{i,t}}( {\hat y} )} ]}} }  + {{\tilde \Phi }_T}, \forall y \in {\mathcal{X}_T},\label{Lemma9-eq1} \\
&\frac{1}{n}\sum\limits_{i = 1}^n {\sum\limits_{t = 1}^T {{\mathbf{E}}[ {\| {{{[ {{g_t}( {{x_{i,t}}} )} ]}_ + }} \|} ]} }  
\le \sqrt {T{{\tilde \Psi }_T}}, \forall y \in {\mathcal{X}_T}, \label{Lemma9-eq2} \\
&\frac{1}{n}\sum\limits_{i = 1}^n {\sum\limits_{t = 1}^T {{\mathbf{E}}[ {\| {{{[ {{g_t}( {{x_{i,t}}} )} ]}_ + }} \|} ]} }  
\le {\vartheta _5}\sum\limits_{i = 1}^n {\sum\limits_{t = 1}^T {{\mathbf{E}}\big[ {\| {{{[ {{g_{i,t}}( {{e_{i,t}}} )} ]}_ + }} \|} \big]} } + {\Xi _T}, \label{Lemma9-eq3}
\end{flalign}
\end{subequations}
where
\begin{flalign}
\nonumber
&{{\tilde \Phi }_T} = \big( {{G_1} + 2LR( \mathbb{X} )} \big){\vartheta _1} + {\Phi _T} + \big( {{G_1} + 2LR( \mathbb{X} )} \big)\sqrt C {\vartheta _2}\sum\limits_{t = 1}^T {{s_t}}  
+ {{{\big( {{G_1} + 2LR( \mathbb{X} )} \big)}^2}\vartheta _2^2}\sum\limits_{t = 1}^T {{\alpha _t}}, \\
\nonumber
&{{\tilde \Psi }_T} = 2{\Psi _T} + 2G_2^2{\vartheta _3} + 8n{p^2}F_1^2G_2^2{\vartheta _4}\sum\limits_{t = 1}^T {\frac{{\alpha _t^2}}{{\delta _t^2}}}  
+ 8n{p^2}F_2^4G_2^2{\vartheta _4}\sum\limits_{t = 1}^T {\frac{{\alpha _t^2\gamma _t^2}}{{\delta _t^2}}}  + 4n CG_2^2{\vartheta _4}\sum\limits_{t = 1}^T {s_t^2}, \\
\nonumber
&{\vartheta _5} = 1 + \frac{{{\delta _1}{G_2}{\vartheta _4}}}{{4p{F_2}}}, \\
\nonumber
&{\Xi _T} = n{G_1}{\vartheta _1} + n{G_2}\sum\limits_{t = 1}^T {{\delta _t}}  
+ np{F_1}{G_2}{\vartheta _2}\sum\limits_{t = 1}^T {\frac{{{\alpha _t}}}{{{\delta _t}}}}  + n\sqrt C {G_2}{\vartheta _2}\sum\limits_{t = 1}^T {{s_t}}.
\end{flalign}
\end{lemma}
\begin{proof}
(i)
From ${f_t}( x ) = \frac{1}{n}\sum\nolimits_{j = 1}^n {{f_{j,t}}( x )}$, we have
\begin{flalign}
\nabla {f_t}( x ) = \frac{1}{n}\sum\nolimits_{j = 1}^n {\nabla {f_{j,t}}( x )}. \label{Lemma9-proof-i-eq1}
\end{flalign}
From \eqref{Lemma9-proof-i-eq1}, we have
\begin{flalign}
\nonumber
\sum\limits_{i = 1}^n { {\nabla {f_t}( {{x_{i,t}}} )} }  &= \frac{1}{n}\sum\limits_{i = 1}^n {\sum\limits_{j = 1}^n { {\nabla {f_{j,t}}( {{x_{i,t}}} )} } } \\
\nonumber
& = \frac{1}{n}\sum\limits_{i = 1}^n {\sum\limits_{j = 1}^n { {\nabla {f_{j,t}}( {{x_{j,t}}} )} } } 
 + \frac{1}{n}\sum\limits_{i = 1}^n {\sum\limits_{j = 1}^n \big({{\nabla {f_{j,t}}( {{x_{i,t}}} ) - \nabla {f_{j,t}}( {{x_{j,t}}} )} } }\big) \\
& = \sum\limits_{i = 1}^n { {\nabla {f_{i,t}}( {{x_{i,t}}} )} }  + \frac{1}{n}\sum\limits_{i = 1}^n {\sum\limits_{j = 1}^n \big({{\nabla {f_{j,t}}( {{x_{i,t}}} ) - \nabla {f_{j,t}}( {{x_{j,t}}} )} }\big) }.
\label{Lemma9-proof-i-eq2}
\end{flalign}
From \eqref{Lemma9-proof-i-eq2}, \eqref{ass5-eq1}, \eqref{ass4-eq1a} and \eqref{ass1-eq1}, we have
\begin{flalign}
\nonumber
&\;\;\;\;\;\frac{1}{n}\sum\limits_{i = 1}^n {\sum\limits_{t = 1}^T {{\mathbf{E}}[ {\langle {\nabla {f_t}( {{x_{i,t}}} ),{x_{i,t}} - y} \rangle } ]} } \\
\nonumber
& \le \frac{1}{n}\sum\limits_{i = 1}^n {\sum\limits_{t = 1}^T {\mathbf{E}[ {\langle {\nabla {f_t}( {{x_{i,t}}} ),{{\bar z}_t} - y} \rangle } ]} }  + \frac{1}{{{n^2}}}\sum\limits_{i = 1}^n {\sum\limits_{j = 1}^n {\sum\limits_{t = 1}^T {\mathbf{E}[ {\| {\nabla {f_{j,t}}( {{x_{i,t}}} )} \|\| {{x_{i,t}} - {{\bar z}_t}} \|} ]} } } \\
\nonumber
& \le \frac{1}{n}\sum\limits_{i = 1}^n {\sum\limits_{t = 1}^T {{\mathbf{E}}[ {\langle {\nabla {f_{i,t}}( {{x_{i,t}}} ),{x_{i,t}} - y} \rangle } ]} }  
+ \frac{1}{n^2}\sum\limits_{i = 1}^n {\sum\limits_{j = 1}^n {\sum\limits_{t = 1}^T {{\mathbf{E}}[ {2LR( \mathbb{X} )\| {{x_{i,t}} - {x_{j,t}}} \|} ]} } } \\
& + \frac{{2{G_1}}}{n}\sum\limits_{i = 1}^n {\sum\limits_{t = 1}^T {\mathbf{E}[ {\| {{x_{i,t}} - {{\bar z}_t}} \|} ]} }.
\label{Lemma9-proof-i-eq3}
\end{flalign}

Next, we derive an upper bound of the second term on the right-hand side of \eqref{Lemma9-proof-i-eq3}. 

From \eqref{Algorithm1-eq2}, and noting that ${\hat{z}_{j,t}^i}={\hat{z}_{j,t}}$ and ${W_{t}}$ is doubly stochastic, we have
\begin{flalign}
\frac{{2{G_1}}}{n}\sum\limits_{i = 1}^n {\sum\limits_{t = 1}^T {\mathbf{E}[ {\| {{x_{i,t}} - {{\bar z}_t}} \|} ]} }  \le \frac{{2{G_1}}}{n}\sum\limits_{i = 1}^n {\sum\limits_{t = 1}^T {\mathbf{E}[ {\| {{{\hat z}_{i,t}} - {{\bar z}_t}} \|} ]} }. \label{Lemma9-proof-i-eq6.5}
\end{flalign}
We have
\begin{flalign}
\nonumber
&\;\;\;\;\;\sum\limits_{i = 1}^n {\sum\limits_{t = 1}^T {\mathbf{E}\Big[ {\frac{{{G_1} + 2LR( \mathbb{X} )}}{n}{\vartheta _2}\| {\varepsilon _{i,t}^z} \|} \Big]} } \\
\nonumber
& \le \sum\limits_{i = 1}^n {\sum\limits_{t = 1}^T {\mathbf{E}\Big[ {\frac{{{G_1} + 2LR( \mathbb{X} )}}{n}{\vartheta _2}\| {{z_{i,t + 1}} - {{\hat z}_{i,t + 1}}} \|} \Big]} }  + \sum\limits_{i = 1}^n {\sum\limits_{t = 1}^T {\mathbf{E}\Big[ {\frac{{{G_1} + 2LR( \mathbb{X} )}}{n}{\vartheta _2}\| {{x_{i,t}} - {z_{i,t + 1}}} \|} \Big]} } \\
& \le \big( {{G_1} + 2LR( \mathbb{X} )} \big)\sqrt C {\vartheta _2}\sum\limits_{t = 1}^T {{s_t}}  + \sum\limits_{i = 1}^n {\sum\limits_{t = 1}^T {\mathbf{E}\Big[ {\frac{{{{\big( {{G_1} + 2LR( \mathbb{X} )} \big)}^2}\vartheta _2^2{\alpha _t}}}{n} + \frac{{{{\| {{x_{i,t}} - {z_{i,t + 1}}} \|}^2}}}{{4n{\alpha _t}}}} \Big]} },
\label{Lemma9-proof-i-eq7}
\end{flalign}
where
the last inequality holds since \eqref{Lemma7-proof-eq10} holds and $\{ {{s_t}} \}$ is decreasing.

Combining \eqref{Lemma6-eq1}, \eqref{Lemma9-proof-i-eq3}--\eqref{Lemma9-proof-i-eq7}, \eqref{Lemma8-eq1}, and \eqref{Lemma6-eq3}, we know that \eqref{Lemma9-eq1} holds.

(ii)
We have
\begin{flalign}
\nonumber
{{{\| {{{[ {{g_{i,t}}( {{x_{i,t}}} )} ]}_ + }} \|}^2}} &= {{{\| {{{[ {{g_{i,t}}( {{x_{i,t}}} )} ]}_ + } - {{[ {{g_{i,t}}( {{x_{j,t}}} )} ]}_ + } + {{[ {{g_{i,t}}( {{x_{j,t}}} )} ]}_ + }} \|}^2}}  \\
\nonumber
& \ge \frac{1}{2}{{{\| {{{[ {{g_{i,t}}( {{x_{j,t}}} )} ]}_ + }} \|}^2}}  - {{{\| {{{[ {{g_{i,t}}( {{x_{i,t}}} )} ]}_ + } - {{[ {{g_{i,t}}( {{x_{j,t}}} )} ]}_ + }} \|}^2}}  \\
\nonumber
& \ge \frac{1}{2}{{{\| {{{[ {{g_{i,t}}( {{x_{j,t}}} )} ]}_ + }} \|}^2}}  - {{{\| {{g_{i,t}}( {{x_{i,t}}} ) - {g_{i,t}}( {{x_{j,t}}} )} \|}^2}}  \\
& \ge \frac{1}{2} {{{\| {{{[ {{g_{i,t}}( {{x_{j,t}}} )} ]}_ + }} \|}^2}}  - G_2^2{{{\| {{x_{i,t}} - {x_{j,t}}} \|}^2}},
\label{Lemma9-proof-ii-eq1}
\end{flalign}
where the second inequality holds due to the nonexpansive property of the projection ${[  \cdot  ]_ + }$, and the last inequality holds due to \eqref{ass4-eq2}.

From ${g_t}( x ) = {\rm{col}}( {{g_{1,t}}( x ), \cdot  \cdot  \cdot ,{g_{n,t}}( x )} )$, we have
\begin{flalign}
\sum\limits_{i = 1}^n {\sum\limits_{j = 1}^n {\sum\limits_{t = 1}^T {{{{\| {{{[ {{g_{i,t}}( {{x_{j,t}}} )} ]}_ + }} \|}^2}} } } } 
 = \sum\limits_{i = 1}^n {\sum\limits_{t = 1}^T { {{{\| {{{[ {{g_t}( {{x_{i,t}}} )} ]}_ + }} \|}^2}} } }.
\label{Lemma9-proof-ii-eq2}
\end{flalign}
From \eqref{Lemma9-proof-ii-eq1}--\eqref{Lemma9-proof-ii-eq2}, we have
\begin{flalign}
\nonumber
&\;\;\;\;\;\frac{1}{n}\sum\limits_{i = 1}^n {\sum\limits_{t = 1}^T {{\mathbf{E}}\big[ {{{\| {{{[ {{g_t}( {{x_{i,t}}} )} ]}_ + }} \|}^2}} \big]} }  \\
\nonumber
& \le \frac{2}{n}\sum\limits_{i = 1}^n {\sum\limits_{j = 1}^n {\sum\limits_{t = 1}^T {{\mathbf{E}}\big[ {{{\| {{{[ {{g_{i,t}}( {{x_{i,t}}} )} ]}_ + }} \|}^2}} \big]} } }  
+ \frac{{2G_2^2}}{n}\sum\limits_{i = 1}^n {\sum\limits_{j = 1}^n {\sum\limits_{t = 1}^T {{\mathbf{E}}[ {{{\| {{x_{i,t}} - {x_{j,t}}} \|}^2}} ]} } } \\
& = 2\sum\limits_{i = 1}^n {\sum\limits_{t = 1}^T {{\mathbf{E}}\big[ {{{\| {{{[ {{g_{i,t}}( {{x_{i,t}}} )} ]}_ + }} \|}^2}} \big]} }
+ \frac{{2G_2^2}}{n}\sum\limits_{i = 1}^n {\sum\limits_{j = 1}^n {\sum\limits_{t = 1}^T {{\mathbf{E}}[ {{{\| {{x_{i,t}} - {x_{j,t}}} \|}^2}} ]} } }.
\label{Lemma9-proof-ii-eq3}
\end{flalign}

Next, we derive an upper bound of the second term on the right-hand side of \eqref{Lemma9-proof-ii-eq3}. 

From \eqref{Lemma6-eq2}, we have
\begin{flalign}
\frac{1}{n}\sum\limits_{i = 1}^n {\sum\limits_{j = 1}^n {\sum\limits_{t = 1}^T {{\mathbf{E}}[ {{{\| {{x_{i,t}} - {x_{j,t}}} \|}^2}} ]} } } \le {{ \vartheta }_3} + {{ \vartheta }_4}\sum\limits_{t = 1}^T {\sum\limits_{i = 1}^n {{\mathbf{E}}[{{\| {\varepsilon _{i,t}^z} \|}^2}]} }. \label{Lemma9-proof-ii-eq5}
\end{flalign}

From $\varepsilon _{i,t}^z = {{\hat z}_{i,t + 1}} - {x_{i,t}}$, we have
\begin{flalign}
\nonumber
&\;\;\;\;\;2G_2^2{\vartheta _4}\sum\limits_{i = 1}^n {\sum\limits_{t = 1}^T {{\mathbf{E}}} } [ {{{\| {\varepsilon _{i,t}^z} \|}^2}} ] \\
& \le 4G_2^2{\vartheta _4}\sum\limits_{i = 1}^n {\sum\limits_{t = 1}^T {{\mathbf{E}}} } [ {{{\| {{z_{i,t + 1}} - {x_{i,t}}} \|}^2}} ]
 + 4G_2^2{\vartheta _4}\sum\limits_{i = 1}^n {\sum\limits_{t = 1}^T {{\mathbf{E}}} } [ {{{\| {{z_{i,t + 1}} - {{\hat z}_{i,t + 1}}} \|}^2}} ].
\label{Lemma9-proof-ii-eq6}
\end{flalign}
By applying \eqref{Lemma3-eq2} to the update \eqref{Algorithm1-eq6}, we have
\begin{flalign}
\nonumber
\| {{z_{i,t + 1}} - {x_{i,t}}} \| &\le {\alpha _t}\| {{a_{i,t + 1}}} \| \\
\nonumber
& = {\alpha _t}\big\| {{{\hat \nabla }_1}{f_{i,t}}( {{x_{i,t}}} ) + {{\big( {{{\hat \nabla }_1}{g_{i,t}}( {{x_{i,t}}} )} \big)}^T}{v_{i,t + 1}}} \big\|\\
\nonumber
& \le {\alpha _t}\| {{{\hat \nabla }_1}{f_{i,t}}( {{x_{i,t}}} )} \| + {\alpha _t}\| {{{\hat \nabla }_1}{g_{i,t}}( {{x_{i,t}}} )} \|\| {{v_{i,t + 1}}} \|\\
& \le \frac{{p{F_1}{\alpha _t}}}{{{\delta _t}}} + \frac{{pF_2^2{\alpha _t}{\gamma _t}}}{{{\delta _t}}},
\label{Lemma9-proof-ii-eq7}
\end{flalign}
where the equality holds due to \eqref{Algorithm1-eq5}, the last inequality holds due to \eqref{lemma5-eq2}, \eqref{lemma5-eq6} and \eqref{Lemma8-proof-eq4}.

From \eqref{Lemma9-proof-ii-eq7}, we have
\begin{flalign}
{\| {{z_{i,t + 1}} - {x_{i,t}}} \|^2} \le \frac{{2{p^2}F_1^2\alpha _t^2}}{{\delta _t^2}} + \frac{{2{p^2}F_2^4\alpha _t^2\gamma _t^2}}{{\delta _t^2}}.
\label{Lemma9-proof-ii-eq8}
\end{flalign}

From \eqref{Lemma7-proof-eq9}, and noting that $\{ {{s_t}} \}$ is decreasing, we have
\begin{flalign}
{\mathbf{E}}[ {{{\| {{z_{i,t + 1}} - {{\hat z}_{i,t + 1}}} \|}^2}} ] \le Cs_t^2.
\label{Lemma9-proof-ii-eq9}
\end{flalign}
From ${g_{i,t}}( y ) \le {\mathbf{0}_{{m_i}}}$, $\forall i \in [ n ]$, $\forall t \in {\mathbb{N}_ + }$ when $y \in {\mathcal{X}_T}$, we have
\begin{flalign}
{\Lambda _T}( y ) \le 0.
\label{Lemma9-proof-ii-eq10}
\end{flalign}

Using the H\"{o}lder’s inequality, we have
\begin{flalign}
{\Big( {\frac{1}{n}\sum\limits_{i = 1}^n {\sum\limits_{t = 1}^T {{\mathbf{E}}[ {\| {{{[ {{g_t}( {{x_{i,t}}} )} ]}_ + }} \|} ]} } } \Big)^2}
\le \frac{T}{n}\sum\limits_{i = 1}^n {\sum\limits_{t = 1}^T {{\mathbf{E}}[ {{{\| {{{[ {{g_t}( {{x_{i,t}}} )} ]}_ + }} \|}^2}} ]} }. \label{Lemma9-proof-ii-eq11}
\end{flalign}

Combining \eqref{Lemma9-proof-ii-eq3}--\eqref{Lemma9-proof-ii-eq6}, \eqref{Lemma9-proof-ii-eq8}--\eqref{Lemma9-proof-ii-eq11} and \eqref{Lemma8-eq2} yields \eqref{Lemma9-eq2}.

(iii)
We have
\begin{flalign}
\nonumber
\frac{1}{n}\sum\limits_{j = 1}^n {\sum\limits_{t = 1}^T {{\mathbf{E}}\big[ {\| {{{[ {{g_t}( {{x_{j,t}}} )} ]}_ + }} \|} \big]} } 
&\le \frac{1}{n}\sum\limits_{i = 1}^n {\sum\limits_{j = 1}^n {\sum\limits_{t = 1}^T {{\mathbf{E}}\big[ {\| {{{[ {{g_{i,t}}( {{x_{j,t}}} )} ]}_ + }} \|} \big]} } } \\
\nonumber
& = \frac{1}{n}\sum\limits_{i = 1}^n {\sum\limits_{j = 1}^n {\sum\limits_{t = 1}^T {{\mathbf{E}}\big[ {\| {{{[ {{g_{i,t}}( {{x_{i,t}}} )} ]}_ + } + {{[ {{g_{i,t}}( {{x_{j,t}}} )} ]}_ + } 
- {{[ {{g_{i,t}}( {{x_{i,t}}} )} ]}_ + }} \|} \big]} } } \\
\nonumber
& \le \frac{1}{n}\sum\limits_{i = 1}^n {\sum\limits_{j = 1}^n {\sum\limits_{t = 1}^T {{\mathbf{E}}\big[ {\| {{{[ {{g_{i,t}}( {{x_{i,t}}} )} ]}_ + }} \| + \| {{g_{i,t}}( {{x_{i,t}}} ) 
- {g_{i,t}}( {{x_{j,t}}} )} \|} \big]} } } \\
& \le \frac{1}{n}\sum\limits_{i = 1}^n {\sum\limits_{j = 1}^n {\sum\limits_{t = 1}^T {{\mathbf{E}}\big[ {\| {{{[ {{g_{i,t}}( {{x_{i,t}}} )} ]}_ + }} \| 
+ {G_2}\| {{x_{i,t}} - {x_{j,t}}} \|} \big]} } },
\label{Lemma9-proof-iii-eq1}
\end{flalign}
where the first inequality holds due to ${g_t}( x ) = {\rm{col}}( {{g_{1,t}}( x ), \cdot  \cdot  \cdot ,{g_{n,t}}( x )} )$, 
the second inequality holds due to the nonexpansive property of the projection ${[  \cdot  ]_ + }$, and the last inequality holds due to \eqref{ass4-eq2}.

From \eqref{Lemma6-eq1}, we have
\begin{flalign}
\nonumber
&\;\;\;\;\;\frac{1}{n}\sum\limits_{i = 1}^n {\sum\limits_{j = 1}^n {\sum\limits_{t = 1}^T {{\mathbf{E}}[ {\| {{x_{i,t}} - {x_{j,t}}} \|} ]} } } \\
\nonumber
& \le n{\vartheta _1} + {{ \vartheta }_2}\sum\limits_{i = 1}^n {\sum\limits_{t = 1}^T {{\mathbf{E}}[ {\| {{z_{i,t + 1}} - {x_{i,t}}} \|} ]} }  
+ {{ \vartheta }_2}\sum\limits_{i = 1}^n {\sum\limits_{t = 1}^T {{\mathbf{E}}[ {\| {{{ z}_{i,t + 1}} - {\hat z_{i,t + 1}}} \|} ]} } \\
& \le n{\vartheta _1} 
+ {{ \vartheta }_2}\sum\limits_{i = 1}^n {\sum\limits_{t = 1}^T {{\mathbf{E}}[ {\| {{z_{i,t + 1}} - {x_{i,t}}} \|} ]} }    
+ n\sqrt C {{ \vartheta }_2}\sum\limits_{t = 1}^T {{s_t}},
\label{Lemma9-proof-iii-eq2}
\end{flalign}
where the last inequality holds due to $\varepsilon _{i,t}^z = {{\hat z}_{i,t + 1}} - {x_{i,t}}$,
the last inequality holds since \eqref{Lemma7-proof-eq10} holds and $\{ {{s_t}} \}$ is decreasing.

Similar to the way to get \eqref{Lemma9-proof-ii-eq7}, we have
\begin{flalign}
\nonumber
&\;\;\;\;\;\sum\limits_{i = 1}^n {\sum\limits_{t = 1}^T {{\mathbf{E}}[ {\| {{z_{i,t + 1}} - {x_{i,t}}} \|} ]} } \\
\nonumber
& \le \sum\limits_{i = 1}^n {\sum\limits_{t = 1}^T {{\mathbf{E}}[ {{\alpha _t}\| {{{\hat \nabla }_1}{f_{i,t}}( {{x_{i,t}}} )} \| 
+ {\alpha _t}\| {{{\hat \nabla }_1}{g_{i,t}}( {{x_{i,t}}} )} \|\| {{v_{i,t + 1}}} \|} ]} } \\
\nonumber
& \le np{F_1}\sum\limits_{t = 1}^T {\frac{{{\alpha _t}}}{{{\delta _t}}}} 
 + p{F_2}\sum\limits_{i = 1}^n {\sum\limits_{t = 1}^T {\frac{{{\alpha _t}{\gamma _t}}}{{{\delta _t}}}{\mathbf{E}}\big[ {\| {{{[ {{g_{i,t}}( {{e_{i,t}}} )} ]}_ + }} \|} \big]} } \\
\nonumber
&  \le np{F_1}\sum\limits_{t = 1}^T {\frac{{{\alpha _t}}}{{{\delta _t}}}}  
+ p{F_2}\sum\limits_{i = 1}^n {\sum\limits_{t = 1}^T {\frac{{{\delta _t}}}{{4{p^2}F_2^2}}{\mathbf{E}}\big[ {\| {{{[ {{g_{i,t}}( {{e_{i,t}}} )} ]}_ + }} \|} \big]} } \\   
& \le np{F_1}\sum\limits_{t = 1}^T {\frac{{{\alpha _t}}}{{{\delta _t}}}}  
+ \frac{{{\delta _1}}}{{4p{F_2}}}\sum\limits_{i = 1}^n {\sum\limits_{t = 1}^T {{\mathbf{E}}\big[ {\| {{{[ {{g_{i,t}}( {{e_{i,t}}} )} ]}_ + }} \|} \big]} },
\label{Lemma9-proof-iii-eq3}
\end{flalign}
where the second inequality holds due to \eqref{lemma5-eq2}, \eqref{lemma5-eq6} and \eqref{Algorithm1-eq4},
the third inequality holds due to ${\alpha _t}{\gamma _t} \le \delta _t^2/4{p^2}F_2^2$,
the last inequality holds since $\{ {{\delta_t}} \}$ is nonincreasing.

We have
\begin{flalign}
\nonumber
&\;\;\;\;\;\sum\limits_{i = 1}^n {\sum\limits_{t = 1}^T {{\mathbf{E}}\big[ {\| {{{[ {{g_{i,t}}( {{x_{i,t}}} )} ]}_ + }} \|} \big]} } \\
\nonumber
& = \sum\limits_{i = 1}^n {\sum\limits_{t = 1}^T {{\mathbf{E}}\big[ {\| {{{[ {{g_{i,t}}( {{x_{i,t}}} )} ]}_ + } - {{[ {{g_{i,t}}( {{e_{i,t}}} )} ]}_ + } 
+ {{[ {{g_{i,t}}( {{e_{i,t}}} )} ]}_ + }} \|} \big]} } \\
\nonumber
& \le \sum\limits_{i = 1}^n {\sum\limits_{t = 1}^T {{\mathbf{E}}\big[ {\| {{g_{i,t}}( {{x_{i,t}}} ) - {g_{i,t}}( {{e_{i,t}}} )} \| 
+ \| {{{[ {{g_{i,t}}( {{e_{i,t}}} )} ]}_ + }} \|} \big]} } \\
\nonumber
& \le \sum\limits_{i = 1}^n {\sum\limits_{t = 1}^T {{\mathbf{E}}\big[ {{G_2}\| {{\delta _t}{u_{i,t}}} \| + \| {{{[ {{g_{i,t}}( {{e_{i,t}}} )} ]}_ + }} \|} \big]} } \\   
&  = n{G_2}\sum\limits_{t = 1}^T {{\delta _t}}  + \sum\limits_{i = 1}^n {\sum\limits_{t = 1}^T {{\mathbf{E}}\big[ {\| {{{[ {{g_{i,t}}( {{e_{i,t}}} )} ]}_ + }} \|} \big]} },
\label{Lemma9-proof-iii-eq4}
\end{flalign}

Combining \eqref{Lemma9-proof-iii-eq1}--\eqref{Lemma9-proof-iii-eq4} yields \eqref{Lemma9-eq3}.
\end{proof}

\hspace{-3mm}\emph{B. Proof of Theorem~1}

In the following, we prove that \eqref{theorem1-eq2}--\eqref{theorem1-eq4} in Theorem~1 hold in three separate steps.

(i)
For any $T \in {\mathbb{N}_ + }$ and $\theta  \in [ {0,1} )$, we have
\begin{flalign}
\sum\limits_{t = 1}^T {\frac{1}{{{t^\theta }}}}  = \sum\limits_{t = 2}^T {\frac{1}{{{t^\theta }}}}  + 1 \le \int_1^T {\frac{1}{{{t^\theta }}}} dt + 1 \le \frac{{{T^{1 - \theta }}}}{{1 - \theta }}.
\label{Theorem1-proof-i-eq1}
\end{flalign} 
We have
\begin{flalign}
\nonumber
\frac{1}{n}\sum\limits_{i = 1}^n {\sum\limits_{t = 1}^T {{{\mathbf{E}}[ {{\Delta _{i,t}}( {\hat y} )} ]} } }  
& = \frac{1}{{2n}}\sum\limits_{i = 1}^n {\sum\limits_{t = 1}^T {{{\mathbf{E}}\Big[ {\frac{1}{{{\alpha _t}}}\big( {{{\| {\hat y - {x_{i,t}}} \|}^2} 
- {{\| {\hat y - {x_{i,t + 1}}} \|}^2}} \big)} \Big]} } } \\
\nonumber
& \le \frac{1}{{2n}}\sum\limits_{i = 1}^n {\sum\limits_{t = 1}^T {{{\mathbf{E}}\Big[ {\frac{1}{{{\alpha _T}}}\big( {{{\| {\hat y - {x_{i,t}}} \|}^2} 
- {{\| {\hat y - {x_{i,t + 1}}} \|}^2}} \big)} \Big]} } } \\
\nonumber
& = \frac{1}{{2n}}\sum\limits_{i = 1}^n {{{\mathbf{E}}\Big[ {\frac{1}{{{\alpha _T}}}\big( {{{\| {\hat y - {x_{i,1}}} \|}^2} 
- {{\| {\hat y - {x_{i,T + 1}}} \|}^2}} \big)} \Big]}} \\
& \le \frac{1}{{2n{\alpha _T}}}\sum\limits_{i = 1}^n {{{\| {\hat y - {x_{i,1}}} \|}^2}}  \le \frac{{2R{{( \mathbb{X} )}^2}}}{{{\alpha _T}}},
\label{Theorem1-proof-i-eq4}
\end{flalign}
where the first inequality holds since $\{ {{\alpha _t}} \}$ is nonincreasing, and the last inequality holds due to $\hat y = ( {1 - {\xi _t}} )y$, $x_{i,1}$, $y \in \mathbb{X}$, and \eqref{ass1-eq1}.

From \eqref{theorem1-eq1}, we know that ${\alpha _t}{\gamma _t} < \delta _t^2/4{p^2}F_2^2$ due to $2{\theta _3} < {\theta _1} - {\theta _2}$ 
and ${\alpha _t}{\gamma _t}$ is nonincreasing due to ${\theta _1} > {\theta _2}$.

From \eqref{theorem1-eq1} and \eqref{Theorem1-proof-i-eq1}, for any $T \in \mathbb{N}_ +$, we have
\begin{flalign}
&\sum\limits_{t = 1}^T {{s_t}} \le {s_0}\sum\limits_{t = 1}^T {\frac{1}{t}}  \le {s_0}\Big( {\int_1^T {\frac{1}{t}dt}  + 1} \Big) \le  {s_0}{\log ( T ) + {s_0}}, \label{Theorem1-proof-i-eq6}\\
&\sum\limits_{t = 1}^T {\frac{{{s_t}}}{{{\alpha _t}}}}  \le \frac{{{s_0}}}{{{\alpha _0}}}\sum\limits_{t = 1}^T {\frac{1}{{{t^{1 - {\theta _1}}}}}}  \le \frac{{{{s_0}T^{ {\theta _1}}}}}{{{\theta _1} }{\alpha _0}}.
\label{Theorem1-proof-i-eq7}
\end{flalign}
Combining \eqref{Lemma9-eq1}, \eqref{Theorem1-proof-i-eq4}, \eqref{Theorem1-proof-i-eq6}--\eqref{Theorem1-proof-i-eq7} and \eqref{theorem1-eq1}, from \eqref{Theorem1-proof-i-eq1} and the arbitrariness of ${y} \in {\mathcal{X}_T}$, we have
\begin{flalign}
\nonumber
&\;\;\;\;\;\mathbf{E}[{{\rm{Net}\mbox{-}\rm{Reg}}( T )}] \\
\nonumber
&\le \frac{{2R{{( \mathbb{X} )}^2}{T^{{\theta _1}}}}}{{{\alpha _0}}} + \big( {{G_1} + 2LR( \mathbb{X} )} \big){\vartheta _1} 
+  \big( {{G_1} + 2LR( \mathbb{X} )} \big)\sqrt C{\vartheta _2}{s_0}\log ( T ) + \big( {{G_1} + 2LR( \mathbb{X} )} \big){\vartheta _2}{s_0}
 \\
\nonumber
&\;\; + \frac{{{\big( {{G_1} + 2LR( \mathbb{X} )} \big)^2}\vartheta _2^2{\alpha _0}{T^{1 - {\theta _1}}}}}{{1 - {\theta _1}}}
+ \frac{{2{p^2}F_1^2{\alpha _0}{T^{1 - {\theta _1} + 2{\theta _3}}}}}{{( {1 - {\theta _1} + 2{\theta _3}} )r{{( \mathbb{X} )}^2}}}  
+ \frac{{{F_2}{G_2}R( \mathbb{X} )\gamma _0{T^{1 + {\theta _2} - {\theta _3}}}}}{{1 + {\theta _2} - {\theta _3}}} 
 \\
&\;\; + \frac{{2{F_2}{G_2}r( \mathbb{X} )\gamma _0{T^{1 + {\theta _2} - {\theta _3}}}}}{{1 + {\theta _2} - {\theta _3}}}
 + \frac{{{G_1}R( \mathbb{X} ){T^{1 - {\theta _3}}}}}{{1 - {\theta _3}}} + \frac{{2Lr( \mathbb{X} )R( \mathbb{X} ){T^{1 - {\theta _3}}}}}{{1 - {\theta _3}}} 
+ \frac{{2\sqrt C R( \mathbb{X} ){s_0}{T^{{\theta _1}}}}}{{{\theta _1}{\alpha _0}}}.
\label{Theorem1-proof-i-eq8}
\end{flalign}

From \eqref{Theorem1-proof-i-eq8}, we know that \eqref{theorem1-eq2} holds.

(ii)
From \eqref{theorem1-eq1} with ${\theta _4} \in [ {1,1 + {\theta _1} - {\theta _2}} )$ and \eqref{Theorem1-proof-i-eq1}, 
and noting that ${\theta _2} < {\theta _4} - {\theta _1} + {\theta _2} < 1$ and $\{ {{s_t}} \}$ is decreasing, for any $T \in \mathbb{N}_ +$, we have
\begin{flalign}
&\sum\limits_{t = 1}^T {\frac{{{s_t}}}{{{\alpha _t}{\gamma _t}}}}  
\le \frac{{{s_0}{T^{1 + {\theta _1} - {\theta _2} - {\theta _4}}}}}{{( {1 + {\theta _1} - {\theta _2} - {\theta _4}} ){\alpha _0}\gamma _0}}, \label{Theorem1-proof-ii-eq1}\\
&\sum\limits_{t = 1}^T {s_t^2}  \le {s_0}\sum\limits_{t = 1}^T {{s_t}} \le s_0^2\sum\limits_{t = 1}^T {\frac{1}{t}}  \le s_0^2\log ( T ) + s_0^2.
\label{Theorem1-proof-ii-eq2}
\end{flalign}
From \eqref{theorem1-eq1} with $2{\theta _1} - 2{\theta _3} < 1$ and \eqref{Theorem1-proof-i-eq1}, 
and noting that $0 < 2{\theta _1} - 2{\theta _2} - 2{\theta _3} < 1$, we have
\begin{flalign}
&\sum\limits_{t = 1}^T {\frac{{\alpha _t^2}}{{\delta _t^2}}}  
\le \frac{{\alpha _0^2{T^{1 - 2{\theta _1} + 2{\theta _3}}}}}{{( {1 - 2{\theta _1} + 2{\theta _3}} )r{{( \mathbb{X} )}^2}}}, \label{Theorem1-proof-ii-eq3}\\
&\sum\limits_{t = 1}^T {\frac{{\alpha _t^2\gamma _t^2}}{{\delta _t^2}}}  
\le \frac{{\alpha _0^2\gamma _0^2{T^{1 - 2{\theta _1} + 2{\theta _2} + 2{\theta _3}}}}}{{( {1 - 2{\theta _1} + 2{\theta _2} + 2{\theta _3}} )r{{( \mathbb{X} )}^2}}}.
\label{Theorem1-proof-ii-eq4}
\end{flalign}
Combining \eqref{Lemma9-eq2} and \eqref{Theorem1-proof-ii-eq1}--\eqref{Theorem1-proof-ii-eq4}, from \eqref{Theorem1-proof-i-eq1}, we have
\begin{flalign}
\nonumber
&\;\;\;\;\;{( {\frac{1}{n}\sum\limits_{i = 1}^n {\sum\limits_{t = 1}^T {{\mathbf{E}}[ {\| {{{[ {{g_t}( {{x_{i,t}}} )} ]}_ + }} \|} ]} } } )^2} \\
\nonumber
& \le \frac{{8nR{{( \mathbb{X} )}^2}{T^{1 + {\theta _1} - {\theta _2}}}}}{{{\alpha _0}\gamma _0}} + \frac{{4n{p^2}F_1^2{\alpha _0}{T^{2 - {\theta _1} - {\theta _2} 
+ 2{\theta _3}}}}}{{( {1 - {\theta _1} - {\theta _2} + 2{\theta _3}} )r{{( \mathbb{X} )}^2}\gamma _0}} 
+ \frac{{8n{G_1}R( \mathbb{X} ){T^{2 - {\theta _2}}}}}{({1 - {\theta _2}})\gamma _0} + \frac{{4n{G_1}R( \mathbb{X} ){T^{2 - {\theta _2} - {\theta _3}}}}}{({1 - {\theta _2} - {\theta _3}})\gamma _0}\\
\nonumber
&\;\; + \frac{{8nLr( \mathbb{X} )R( \mathbb{X} ){T^{2 - {\theta _2} - {\theta _3}}}}}{({1 - {\theta _2} - {\theta _3}})\gamma _0}  + \frac{{4n{F_2}{G_2}R( \mathbb{X} ){T^{2 - {\theta _3}}}}}{{1 - {\theta _3}}} 
 + \frac{{8n{F_2}{G_2}r( \mathbb{X} ){T^{2 - {\theta _3}}}}}{{1 - {\theta _3}}}\\
\nonumber
&\;\; + \frac{{8n\sqrt C R( \mathbb{X} ){s_0}{T^{2 + {\theta _1} - {\theta _2} - {\theta _4}}}}}{{( {1 + {\theta _1} - {\theta _2} - {\theta _4}} ){\alpha _0}\gamma _0}} 
+ \frac{{8n{p^2}F_1^2G_2^2{\vartheta _4}\alpha _0^2{T^{2 - 2{\theta _1} 
+ 2{\theta _3}}}}}{{( {1 - 2{\theta _1} + 2{\theta _3}} )r{{( \mathbb{X} )}^2}}} + \frac{{8n{p^2}F_2^4G_2^2{\vartheta _4}\alpha _0^2\gamma _0^2{T^{2 - 2{\theta _1} + 2{\theta _2} 
+ 2{\theta _3}}}}}{{( {1 - 2{\theta _1} + 2{\theta _2} + 2{\theta _3}} )r{{( \mathbb{X} )}^2}}}\\
&\;\; + 4nCG_2^2{\vartheta _4}s_0^2 T\log ( T ) + 4nCG_2^2{\vartheta _4}s_0^2 T + 2G_2^2{\vartheta _3}T,
\label{Theorem1-proof-ii-eq14}
\end{flalign}
From \eqref{theorem1-eq1} with $1 \le 2{\theta _1} - 2{\theta _3} < 1 + 2{\theta _2}$ and \eqref{Theorem1-proof-i-eq1}, 
and noting that $0 < 2{\theta _1} - 2{\theta _2} - 2{\theta _3} < 1$, we have
\begin{flalign}
&\sum\limits_{t = 1}^T {\frac{{\alpha _t^2}}{{\delta _t^2}}}  
\le \frac{{\alpha _0^2}}{{r{{( \mathbb{X} )}^2}}}\sum\limits_{t = 1}^T {\frac{1}{t}}  
\le \frac{{\alpha _0^2\log ( T )}}{{r{{( \mathbb{X} )}^2}}} + \frac{{\alpha _0^2}}{{r{{( \mathbb{X} )}^2}}}, \label{Theorem1-proof-ii-eq6}\\
&\sum\limits_{t = 1}^T {\frac{{\alpha _t^2\gamma _t^2}}{{\delta _t^2}}}  
\le \frac{{\alpha _0^2\gamma _0^2{T^{1 - 2{\theta _1} + 2{\theta _2} + 2{\theta _3}}}}}{{( {1 - 2{\theta _1} + 2{\theta _2} + 2{\theta _3}} )r{{( \mathbb{X} )}^2}}}.
\label{Theorem1-proof-ii-eq7}
\end{flalign}
Combining \eqref{Lemma9-eq2}, \eqref{Theorem1-proof-ii-eq1}--\eqref{Theorem1-proof-ii-eq2} and \eqref{Theorem1-proof-ii-eq6}--\eqref{Theorem1-proof-ii-eq7}, from \eqref{Theorem1-proof-i-eq1}, we have
\begin{flalign}
\nonumber
&\;\;\;\;\;{( {\frac{1}{n}\sum\limits_{i = 1}^n {\sum\limits_{t = 1}^T {{\mathbf{E}}[ {\| {{{[ {{g_t}( {{x_{i,t}}} )} ]}_ + }} \|} ]} } } )^2} \\
\nonumber
& \le \frac{{8nR{{( \mathbb{X} )}^2}{T^{1 + {\theta _1} - {\theta _2}}}}}{{{\alpha _0}\gamma _0}} + \frac{{4n{p^2}F_1^2{\alpha _0}{T^{2 - {\theta _1} - {\theta _2} 
+ 2{\theta _3}}}}}{{( {1 - {\theta _1} - {\theta _2} + 2{\theta _3}} )r{{( \mathbb{X} )}^2}\gamma _0}} 
+ \frac{{8n{G_1}R( \mathbb{X} ){T^{2 - {\theta _2}}}}}{({1 - {\theta _2}})\gamma _0} + \frac{{4n{G_1}R( \mathbb{X} ){T^{2 - {\theta _2} - {\theta _3}}}}}{({1 - {\theta _2} - {\theta _3}})\gamma _0}\\
\nonumber
&\;\; + \frac{{8nLr( \mathbb{X} )R( \mathbb{X} ){T^{2 - {\theta _2} - {\theta _3}}}}}{({1 - {\theta _2} - {\theta _3}})\gamma _0}  + \frac{{4n{F_2}{G_2}R( \mathbb{X} ){T^{2 - {\theta _3}}}}}{{1 - {\theta _3}}} 
 + \frac{{8n{F_2}{G_2}r( \mathbb{X} ){T^{2 - {\theta _3}}}}}{{1 - {\theta _3}}}\\
\nonumber
&\;\; + \frac{{8n\sqrt C R( \mathbb{X} ){s_0}{T^{2 + {\theta _1} - {\theta _2} - {\theta _4}}}}}{{( {1 + {\theta _1} - {\theta _2} - {\theta _4}} ){\alpha _0}\gamma _0}} 
+ \frac{{8n{p^2}F_1^2G_2^2{\vartheta _4}\alpha _0^2 T\log ( T )}}{{r{{( \mathbb{X} )}^2}}} + \frac{{8n{p^2}F_1^2G_2^2{\vartheta _4}\alpha _0^2}T}{{r{{( \mathbb{X} )}^2}}} 
+ 2G_2^2{\vartheta _3}T\\
&\;\; + \frac{{8n{p^2}F_2^4G_2^2{\vartheta _4}\alpha _0^2\gamma _0^2{T^{2 - 2{\theta _1} + 2{\theta _2} 
+ 2{\theta _3}}}}}{{( {1 - 2{\theta _1} + 2{\theta _2} + 2{\theta _3}} )r{{( \mathbb{X} )}^2}}} 
+ 4nCG_2^2{\vartheta _4}s_0^2 T\log ( T ) + 4nCG_2^2{\vartheta _4}s_0^2 T,
\label{Theorem1-proof-ii-eq15}
\end{flalign}
From \eqref{theorem1-eq1} with $2{\theta _1} - 2{\theta _2} - 2{\theta _3} \ge 1$ and \eqref{Theorem1-proof-i-eq1}, 
we have
\begin{flalign}
&\sum\limits_{t = 1}^T {\frac{{\alpha _t^2}}{{\delta _t^2}}}  
\le \frac{{\alpha _0^2}}{{r{{( \mathbb{X} )}^2}}}\sum\limits_{t = 1}^T {\frac{1}{t}}  
\le \frac{{\alpha _0^2\log ( T )}}{{r{{( \mathbb{X} )}^2}}} + \frac{{\alpha _0^2}}{{r{{( \mathbb{X} )}^2}}}, \label{Theorem1-proof-ii-eq9}\\
&\sum\limits_{t = 1}^T {\frac{{\alpha _t^2\gamma _t^2}}{{\delta _t^2}}}  
\le \frac{{\alpha _0^2}\gamma _0^2}{{r{{( \mathbb{X} )}^2}}}\sum\limits_{t = 1}^T {\frac{1}{t}}  
\le \frac{{\alpha _0^2\gamma _0^2\log ( T )}}{{r{{( \mathbb{X} )}^2}}} + \frac{{\alpha _0^2}\gamma _0^2}{{r{{( \mathbb{X} )}^2}}}.
\label{Theorem1-proof-ii-eq10}
\end{flalign}
Combining \eqref{Lemma9-eq2}, \eqref{Theorem1-proof-ii-eq1}--\eqref{Theorem1-proof-ii-eq2} and \eqref{Theorem1-proof-ii-eq9}--\eqref{Theorem1-proof-ii-eq10}, from \eqref{Theorem1-proof-i-eq1}, we have
\begin{flalign}
\nonumber
&\;\;\;\;\;{( {\frac{1}{n}\sum\limits_{i = 1}^n {\sum\limits_{t = 1}^T {{\mathbf{E}}[ {\| {{{[ {{g_t}( {{x_{i,t}}} )} ]}_ + }} \|} ]} } } )^2} \\
\nonumber
& \le \frac{{8nR{{( \mathbb{X} )}^2}{T^{1 + {\theta _1} - {\theta _2}}}}}{{{\alpha _0}\gamma _0}} + \frac{{4n{p^2}F_1^2{\alpha _0}{T^{2 - {\theta _1} - {\theta _2} 
+ 2{\theta _3}}}}}{{( {1 - {\theta _1} - {\theta _2} + 2{\theta _3}} )r{{( \mathbb{X} )}^2}\gamma _0}} 
+ \frac{{8n{G_1}R( \mathbb{X} ){T^{2 - {\theta _2}}}}}{({1 - {\theta _2}})\gamma _0} + \frac{{4n{G_1}R( \mathbb{X} ){T^{2 - {\theta _2} - {\theta _3}}}}}{({1 - {\theta _2} - {\theta _3}})\gamma _0} \\
\nonumber
&\;\; + \frac{{8nLr( \mathbb{X} )R( \mathbb{X} ){T^{2 - {\theta _2} - {\theta _3}}}}}{({1 - {\theta _2} - {\theta _3}})\gamma _0}  + \frac{{4n{F_2}{G_2}R( \mathbb{X} ){T^{2 - {\theta _3}}}}}{{1 - {\theta _3}}} 
 + \frac{{8n{F_2}{G_2}r( \mathbb{X} ){T^{2 - {\theta _3}}}}}{{1 - {\theta _3}}} + 2G_2^2{\vartheta _3}T\\
\nonumber
&\;\; + \frac{{8n\sqrt C R( \mathbb{X} ){s_0}{T^{2 + {\theta _1} - {\theta _2} - {\theta _4}}}}}{{( {1 + {\theta _1} - {\theta _2} - {\theta _4}} ){\alpha _0}\gamma _0}} 
+ \frac{{8n{p^2}F_1^2G_2^2{\vartheta _4}\alpha _0^2 T\log ( T )}}{{r{{( \mathbb{X} )}^2}}} + \frac{{8n{p^2}F_1^2G_2^2{\vartheta _4}\alpha _0^2}T}{{r{{( \mathbb{X} )}^2}}} 
+ 4nCG_2^2{\vartheta _4}s_0^2 T \\
&\;\; + \frac{{8n{p^2}F_2^4G_2^2{\vartheta _4}\alpha _0^2 \gamma _0^2T\log ( T )}}{{r{{( \mathbb{X} )}^2}}} 
+ \frac{{8n{p^2}F_2^4G_2^2{\vartheta _4}\alpha _0^2}\gamma _0^2T}{{r{{( \mathbb{X} )}^2}}} + 4nCG_2^2{\vartheta _4}s_0^2 T\log ( T ).
\label{Theorem1-proof-ii-eq16}
\end{flalign}
From \eqref{theorem1-eq1} with ${\theta _4} \ge 1 + {\theta _1} - {\theta _2}$ and \eqref{Theorem1-proof-i-eq1}, 
and noting that $\{ {{s_t}} \}$ is decreasing, for any $T \in \mathbb{N}_ +$, we have
\begin{flalign}
&\sum\limits_{t = 1}^T {\frac{{{s_t}}}{{{\alpha _t}{\gamma _t}}}}  
\le \frac{{{s_0}}}{{{\alpha _0}\gamma _0}}\sum\limits_{t = 1}^T {\frac{1}{t}}  \le \frac{{{s_0}\log ( T )}}{{{\alpha _0}\gamma _0}} + \frac{{{s_0}}}{{{\alpha _0}\gamma _0}}, \label{Theorem1-proof-ii-eq17}\\
&\sum\limits_{t = 1}^T {s_t^2}  \le {s_0}\sum\limits_{t = 1}^T {{s_t}} \le s_0^2\sum\limits_{t = 1}^T {\frac{1}{t}}  \le s_0^2\log ( T ) + s_0^2.
\label{Theorem1-proof-ii-eq18}
\end{flalign}
Similar to the way to get \eqref{Theorem1-proof-ii-eq14}, \eqref{Theorem1-proof-ii-eq15} and \eqref{Theorem1-proof-ii-eq16}, using \eqref{Theorem1-proof-ii-eq17}--\eqref{Theorem1-proof-ii-eq18}, we have
\begin{flalign}
\nonumber
&\;\;\;\;\;{( {\frac{1}{n}\sum\limits_{i = 1}^n {\sum\limits_{t = 1}^T {{\mathbf{E}}[ {\| {{{[ {{g_t}( {{x_{i,t}}} )} ]}_ + }} \|} ]} } } )^2} \\
\nonumber
& \le \frac{{8nR{{( \mathbb{X} )}^2}{T^{1 + {\theta _1} - {\theta _2}}}}}{{{\alpha _0}\gamma _0}} + \frac{{4n{p^2}F_1^2{\alpha _0}{T^{2 - {\theta _1} - {\theta _2} 
+ 2{\theta _3}}}}}{{( {1 - {\theta _1} - {\theta _2} + 2{\theta _3}} )r{{( \mathbb{X} )}^2}\gamma _0}} 
+ \frac{{8n{G_1}R( \mathbb{X} ){T^{2 - {\theta _2}}}}}{({1 - {\theta _2}})\gamma _0} + \frac{{4n{G_1}R( \mathbb{X} ){T^{2 - {\theta _2} - {\theta _3}}}}}{({1 - {\theta _2} - {\theta _3}})\gamma _0}\\
\nonumber
&\;\; + \frac{{8nLr( \mathbb{X} )R( \mathbb{X} ){T^{2 - {\theta _2} - {\theta _3}}}}}{({1 - {\theta _2} - {\theta _3}}\gamma _0} + \frac{{4n{F_2}{G_2}R( \mathbb{X} ){T^{2 - {\theta _3}}}}}{{1 - {\theta _3}}} 
 + \frac{{8n{F_2}{G_2}r( \mathbb{X} ){T^{2 - {\theta _3}}}}}{{1 - {\theta _3}}}\\
\nonumber
&\;\; + \frac{{8n\sqrt C R( \mathbb{X} ){s_0} T \log ( T )}}{{{\alpha _0}\gamma _0}} + \frac{{8n\sqrt C R( \mathbb{X} ){s_0}}T}{{{\alpha _0}\gamma _0}}
 + \frac{{8n{p^2}F_2^4G_2^2{\vartheta _4}\alpha _0^2\gamma _0^2{T^{2 - 2{\theta _1} + 2{\theta _2} 
+ 2{\theta _3}}}}}{{( {1 - 2{\theta _1} + 2{\theta _2} + 2{\theta _3}} )r{{( \mathbb{X} )}^2}}}\\
&\;\; + \frac{{8n{p^2}F_1^2G_2^2{\vartheta _4}\alpha _0^2{T^{2 - 2{\theta _1} 
+ 2{\theta _3}}}}}{{( {1 - 2{\theta _1} + 2{\theta _3}} )r{{( \mathbb{X} )}^2}}} + 4nCG_2^2{\vartheta _4}s_0^2 T\log ( T ) + 4nCG_2^2{\vartheta _4}s_0^2 T + 2G_2^2{\vartheta _3}T,
\label{Theorem1-proof-ii-eq19}
\end{flalign}
\begin{flalign}
\nonumber
&\;\;\;\;\;{( {\frac{1}{n}\sum\limits_{i = 1}^n {\sum\limits_{t = 1}^T {{\mathbf{E}}[ {\| {{{[ {{g_t}( {{x_{i,t}}} )} ]}_ + }} \|} ]} } } )^2} \\
\nonumber
& \le \frac{{8nR{{( \mathbb{X} )}^2}{T^{1 + {\theta _1} - {\theta _2}}}}}{{{\alpha _0}\gamma _0}} + \frac{{4n{p^2}F_1^2{\alpha _0}{T^{2 - {\theta _1} - {\theta _2} 
+ 2{\theta _3}}}}}{{( {1 - {\theta _1} - {\theta _2} + 2{\theta _3}} )r{{( \mathbb{X} )}^2}\gamma _0}} 
+ \frac{{8n{G_1}R( \mathbb{X} ){T^{2 - {\theta _2}}}}}{({1 - {\theta _2}})\gamma _0} + \frac{{4n{G_1}R( \mathbb{X} ){T^{2 - {\theta _2} - {\theta _3}}}}}{({1 - {\theta _2} - {\theta _3}})\gamma _0} \\
\nonumber
&\;\; + \frac{{8npLr( \mathbb{X} )R( \mathbb{X} ){T^{2 - {\theta _2} - {\theta _3}}}}}{({1 - {\theta _2} - {\theta _3}})\gamma _0} + \frac{{4n{F_2}{G_2}R( \mathbb{X} ){T^{2 - {\theta _3}}}}}{{1 - {\theta _3}}} 
 + \frac{{8n{F_2}{G_2}r( \mathbb{X} ){T^{2 - {\theta _3}}}}}{{1 - {\theta _3}}}  + 2G_2^2{\vartheta _3}T\\
\nonumber
&\;\; + \frac{{8n\sqrt C R( \mathbb{X} ){s_0}T\log ( T )}}{{{\alpha _0}\gamma _0}} + \frac{{8n\sqrt C R( \mathbb{X} ){s_0}}T}{{{\alpha _0}\gamma _0}} 
+ \frac{{8n{p^2}F_1^2G_2^2{\vartheta _4}\alpha _0^2 T\log ( T )}}{{r{{( \mathbb{X} )}^2}}} + \frac{{8n{p^2}F_1^2G_2^2{\vartheta _4}\alpha _0^2}T}{{r{{( \mathbb{X} )}^2}}} 
\\
&\;\; + \frac{{8n{p^2}F_2^4G_2^2{\vartheta _4}\alpha _0^2\gamma _0^2{T^{2 - 2{\theta _1} + 2{\theta _2} 
+ 2{\theta _3}}}}}{{( {1 - 2{\theta _1} + 2{\theta _2} + 2{\theta _3}} )r{{( \mathbb{X} )}^2}}} 
+ 4nCG_2^2{\vartheta _4}s_0^2 T\log ( T ) + 4nCG_2^2{\vartheta _4}s_0^2 T,
\label{Theorem1-proof-ii-eq20}
\end{flalign}
\begin{flalign}
\nonumber
&\;\;\;\;\;{( {\frac{1}{n}\sum\limits_{i = 1}^n {\sum\limits_{t = 1}^T {{\mathbf{E}}[ {\| {{{[ {{g_t}( {{x_{i,t}}} )} ]}_ + }} \|} ]} } } )^2} \\
\nonumber
& \le \frac{{8nR{{( \mathbb{X} )}^2}{T^{1 + {\theta _1} - {\theta _2}}}}}{{{\alpha _0}\gamma _0}} + \frac{{4n{p^2}F_1^2{\alpha _0}{T^{2 - {\theta _1} - {\theta _2} 
+ 2{\theta _3}}}}}{{( {1 - {\theta _1} - {\theta _2} + 2{\theta _3}} )r{{( \mathbb{X} )}^2}\gamma _0}} 
+ \frac{{8n{G_1}R( \mathbb{X} ){T^{2 - {\theta _2}}}}}{({1 - {\theta _2}})\gamma _0} + \frac{{4n{G_1}R( \mathbb{X} ){T^{2 - {\theta _2} - {\theta _3}}}}}{({1 - {\theta _2} - {\theta _3}})\gamma _0}\\
\nonumber
&\;\; + \frac{{8nLr( \mathbb{X} )R( \mathbb{X} ){T^{2 - {\theta _2} - {\theta _3}}}}}{({1 - {\theta _2} - {\theta _3}})\gamma _0} + \frac{{4n{F_2}{G_2}R( \mathbb{X} ){T^{2 - {\theta _3}}}}}{{1 - {\theta _3}}} 
 + \frac{{8n{F_2}{G_2}r( \mathbb{X} ){T^{2 - {\theta _3}}}}}{{1 - {\theta _3}}} + \frac{{8n{p^2}F_2^4G_2^2{\vartheta _4}\alpha _0^2}\gamma _0^2T}{{r{{( \mathbb{X} )}^2}}} \\
\nonumber
&\;\; + \frac{{8n\sqrt C R( \mathbb{X} ){s_0}T\log ( T )}}{{{\alpha _0}\gamma _0}} + \frac{{8n\sqrt C R( \mathbb{X} ){s_0}T}}{{{\alpha _0}\gamma _0}} 
+ \frac{{8n{p^2}F_1^2G_2^2{\vartheta _4}\alpha _0^2 T\log ( T )}}{{r{{( \mathbb{X} )}^2}}} + \frac{{8n{p^2}F_1^2G_2^2{\vartheta _4}\alpha _0^2}T}{{r{{( \mathbb{X} )}^2}}} 
 \\
&\;\; + \frac{{8n{p^2}F_2^4G_2^2{\vartheta _4}\alpha _0^2 \gamma _0^2T\log ( T )}}{{r{{( \mathbb{X} )}^2}}} 
+ 4nCG_2^2{\vartheta _4}s_0^2 T\log ( T ) + 4nCG_2^2{\vartheta _4}s_0^2 T  + 2G_2^2{\vartheta _3}T.
\label{Theorem1-proof-ii-eq21}
\end{flalign}

From \eqref{Theorem1-proof-ii-eq14}, \eqref{Theorem1-proof-ii-eq15}, \eqref{Theorem1-proof-ii-eq16} 
and \eqref{Theorem1-proof-ii-eq19}--\eqref{Theorem1-proof-ii-eq21}, we know that \eqref{theorem1-eq3} holds.

(iii)
We have
\begin{flalign}
\nonumber
2{\mathbf{E}}[{\Lambda _T}( {{x_s}} )] &= 2{\mathbf{E}}\Big[\sum\limits_{i = 1}^n {\sum\limits_{t = 1}^T {\frac{{v_{i,t + 1}^T{g_{i,t}}( {{x_s}} )}}{{{\gamma _t}}}} }\Big]
  = 2{\mathbf{E}}\Big[\sum\limits_{i = 1}^n {\sum\limits_{t = 1}^T {{{{\big( {{{[ {{g_{i,t}}( {{e_{i,t}}} )} ]}_ + }} \big)}^T}}{g_{i,t}}( {{x_s}} )} }\Big]  \\
\nonumber
& \le  - 2{\mathbf{E}}\Big[\sum\limits_{i = 1}^n {\sum\limits_{t = 1}^T {{\varsigma _s}{{{\big( {{{[ {{g_{i,t}}( {{e_{i,t}}} )} ]}_ + }} \big)}^T}}{\mathbf{1}_{{m_i}}}} }\Big]
 =  - 2{\varsigma _s}\sum\limits_{i = 1}^n {\sum\limits_{t = 1}^T {\mathbf{E}}[{{{\| {{{[ {{g_{i,t}}( {{e_{i,t}}} )} ]}_ + }} \|}_1}}] } \\
& \le  - 2{\varsigma _s}\sum\limits_{i = 1}^n {\sum\limits_{t = 1}^T {\mathbf{E}}[{\| {{{[ {{g_{i,t}}( {{e_{i,t}}} )} ]}_ + }} \|}] },
\label{Theorem1-proof-iii-eq1}
\end{flalign}
where the second equality holds due to \eqref{Algorithm1-eq4}, and the first inequality holds due to \eqref{ass8-eq1} and ${g_t}( x ) = {\rm{col}}( {{g_{1,t}}( x ), \cdot  \cdot  \cdot ,{g_{n,t}}( x )} )$.

Selecting $y = {x_s}$ in \eqref{Lemma8-eq2}, from \eqref{Theorem1-proof-iii-eq1}, we have
\begin{flalign}
2{\varsigma _s}\sum\limits_{i = 1}^n {\sum\limits_{t = 1}^T {\mathbf{E}}[{\| {{{[ {{g_{i,t}}( {{e_{i,t}}} )} ]}_ + }} \|}] }  \le {\Psi _T},
\label{Theorem1-proof-iii-eq2}
\end{flalign}
Combining \eqref{Theorem1-proof-iii-eq2} and \eqref{Theorem1-proof-ii-eq1}, and using \eqref{theorem1-eq1} with ${\theta _4} \in [ {1,1 + {\theta _1} - {\theta _2}} )$ 
and \eqref{Theorem1-proof-i-eq1}, we have
\begin{flalign}
\nonumber
&\;\;\;\;\;\sum\limits_{i = 1}^n {\sum\limits_{t = 1}^T {{\mathbf{E}}[ {\| {{{[ {{g_{i,t}}( {{e_{i,t}}} )} ]}_ + }} \|} ]} } \\
\nonumber
& \le \frac{{2nR{{( \mathbb{X} )}^2}{T^{{\theta _1} - {\theta _2}}}}}{{{\varsigma _s}{\alpha _0}\gamma _0}} 
+ \frac{{n{p^2}F_1^2{\alpha _0}{T^{1 - {\theta _1} - {\theta _2} + 2{\theta _3}}}}}{{( {1 - {\theta _1} - {\theta _2} + 2{\theta _3}} ){\varsigma _s}r{{( \mathbb{X} )}^2}\gamma _0}} 
+ \frac{{2n{G_1}R( \mathbb{X} ){T^{1 - {\theta _2}}}}}{{( {1 - {\theta _2}} ){\varsigma _s}\gamma _0}} 
+ \frac{{n{G_1}R( \mathbb{X} ){T^{1 - {\theta _2} - {\theta _3}}}}}{{( {1 - {\theta _2} - {\theta _3}} ){\varsigma _s}\gamma _0}} \\
\nonumber
&\;\; + \frac{{2nLr( \mathbb{X} )R( \mathbb{X} ){T^{1 - {\theta _2} - {\theta _3}}}}}{{( {1 - {\theta _2} - {\theta _3}} ){\varsigma _s}\gamma _0}} 
+ \frac{{2n{F_2}{G_2}r( \mathbb{X} ){T^{1 - {\theta _3}}}}}{{( {1 - {\theta _3}} ){\varsigma _s}}} 
+ \frac{{n{F_2}{G_2}R( \mathbb{X} ){T^{1 - {\theta _3}}}}}{{( {1 - {\theta _3}} ){\varsigma _s}}}\\
&\;\; 
+ \frac{{2n\sqrt C R( \mathbb{X} ){s_0}{T^{1 + {\theta _1} - {\theta _2} - {\theta _4}}}}}{{( {1 + {\theta _1} - {\theta _2} - {\theta _4}} ){\varsigma _s}{{\alpha _0}\gamma _0}}}.
\label{Theorem1-proof-iii-eq3}
\end{flalign}
Combining \eqref{Lemma9-eq3}, \eqref{Theorem1-proof-i-eq6}, \eqref{Theorem1-proof-iii-eq3}, and using \eqref{theorem1-eq1} with ${\theta _4} \in [ {1,1 + {\theta _1} - {\theta _2}} )$ and \eqref{Theorem1-proof-i-eq1}, we have
\begin{flalign}
\nonumber
&\;\;\;\;\;\mathbf{E}[{{\rm{Net} \mbox{-} \rm{CCV}}( T )}] \\
\nonumber
& \le \frac{{2nR{{( \mathbb{X} )}^2}{\vartheta _5}{T^{{\theta _1} - {\theta _2}}}}}{{{\varsigma _s}{\alpha _0}\gamma _0}}
 + \frac{{n{p^2}F_1^2{\alpha _0}{\vartheta _5}{T^{1 - {\theta _1} - {\theta _2} + 2{\theta _3}}}}}{{( {1 - {\theta _1} - {\theta _2} + 2{\theta _3}} ){\varsigma _s}r{{( \mathbb{X} )}^2}\gamma _0}} 
 + \frac{{2n{G_1}R( \mathbb{X} ){\vartheta _5}{T^{1 - {\theta _2}}}}}{{( {1 - {\theta _2}} ){\varsigma _s}\gamma _0}} 
+ \frac{{n{G_1}R( \mathbb{X} ){\vartheta _5}{T^{1 - {\theta _2} - {\theta _3}}}}}{{( {1 - {\theta _2} - {\theta _3}} ){\varsigma _s}\gamma _0}} \\
\nonumber
&\;\; + \frac{{2nLr( \mathbb{X} )R( \mathbb{X} ){\vartheta _5}{T^{1 - {\theta _2} - {\theta _3}}}}}{{( {1 - {\theta _2} - {\theta _3}} ){\varsigma _s}\gamma _0}} 
+ \frac{{n{F_2}{G_2}R( \mathbb{X} ){\vartheta _5}{T^{1 - {\theta _3}}}}}{{( {1 - {\theta _3}} ){\varsigma _s}}} 
+ \frac{{2n{F_2}{G_2}r( \mathbb{X} ){\vartheta _5}{T^{1 - {\theta _3}}}}}{{( {1 - {\theta _3}} ){\varsigma _s}}} \\
\nonumber
&\;\; + \frac{{np{F_1}{G_2}{\vartheta _2}{\alpha _0}{T^{1 - {\theta _1} + {\theta _3}}}}}{{( {1 - {\theta _1} + {\theta _3}} )r( \mathbb{X} )}} 
+ \frac{{2n\sqrt C R( \mathbb{X} ){s_0}{\vartheta _5}{T^{1 + {\theta _1} - {\theta _2} - {\theta _4}}}}}{{( {1 + {\theta _1} - {\theta _2} - {\theta _4}} ){\varsigma _s}{\alpha _0}\gamma _0}}
 + n{G_1}{\vartheta _1}
 + \frac{{n{G_2}r( \mathbb{X} ){T^{1 - {\theta _3}}}}}{{1 - {\theta _3}}}\\
&\;\; + n\sqrt C {G_2}{\vartheta _2}{s_0}\log ( T ) + n\sqrt C {G_2}{\vartheta _2}{s_0}.
\label{Theorem1-proof-iii-eq5}
\end{flalign}
Combining \eqref{Theorem1-proof-iii-eq2} and \eqref{Theorem1-proof-ii-eq17}, and using \eqref{theorem1-eq1} with ${\theta _4} > 1 + {\theta _1} - {\theta _2}$ 
and \eqref{Theorem1-proof-i-eq1}, we have
\begin{flalign}
\nonumber
&\;\;\;\;\;\sum\limits_{i = 1}^n {\sum\limits_{t = 1}^T {{\mathbf{E}}[ {\| {{{[ {{g_{i,t}}( {{e_{i,t}}} )} ]}_ + }} \|} ]} } \\
\nonumber
& \le \frac{{2nR{{( \mathbb{X} )}^2}{T^{{\theta _1} - {\theta _2}}}}}{{{\varsigma _s}{\alpha _0}\gamma _0}} 
+ \frac{{n{p^2}F_1^2{\alpha _0}{T^{1 - {\theta _1} - {\theta _2} + 2{\theta _3}}}}}{{( {1 - {\theta _1} - {\theta _2} + 2{\theta _3}} ){\varsigma _s}r{{( \mathbb{X} )}^2}\gamma _0}} 
+ \frac{{2n{G_1}R( \mathbb{X} ){T^{1 - {\theta _2}}}}}{{( {1 - {\theta _2}} ){\varsigma _s}\gamma _0}} 
+ \frac{{n{G_1}R( \mathbb{X} ){T^{1 - {\theta _2} - {\theta _3}}}}}{{( {1 - {\theta _2} - {\theta _3}} ){\varsigma _s}\gamma _0}} \\
\nonumber
&\;\; + \frac{{2nLr( \mathbb{X} )R( \mathbb{X} ){T^{1 - {\theta _2} - {\theta _3}}}}}{{( {1 - {\theta _2} - {\theta _3}} ){\varsigma _s}\gamma _0}} 
+ \frac{{n{F_2}{G_2}R( \mathbb{X} ){T^{1 - {\theta _3}}}}}{{( {1 - {\theta _3}} ){\varsigma _s}}} + \frac{{2n{F_2}{G_2}r( \mathbb{X} ){T^{1 - {\theta _3}}}}}{{( {1 - {\theta _3}} ){\varsigma _s}}}\\
&\;\; + \frac{{2n\sqrt C R( \mathbb{X} ){s_0}\log ( T )}}{{{\varsigma _s}{\alpha _0}\gamma _0}} + \frac{{2n\sqrt C R( \mathbb{X} ){s_0}}}{{{\varsigma _s}{\alpha _0}\gamma _0}}.
\label{Theorem1-proof-iii-eq6}
\end{flalign}
Combining \eqref{Lemma9-eq3}, \eqref{Theorem1-proof-i-eq6}, \eqref{Theorem1-proof-iii-eq6}, 
and using \eqref{theorem1-eq1} with ${\theta _4} > 1 + {\theta _1} - {\theta _2}$ and \eqref{Theorem1-proof-i-eq1}, we have
\begin{flalign}
\nonumber
&\;\;\;\;\;\mathbf{E}[{{\rm{Net} \mbox{-} \rm{CCV}}( T )}] \\
\nonumber
& \le \frac{{2nR{{( \mathbb{X} )}^2}{\vartheta _5}{T^{{\theta _1} - {\theta _2}}}}}{{{\varsigma _s}{\alpha _0}\gamma _0}}
 + \frac{{n{p^2}F_1^2{\alpha _0}{\vartheta _5}{T^{1 - {\theta _1} - {\theta _2} + 2{\theta _3}}}}}{{( {1 - {\theta _1} - {\theta _2} + 2{\theta _3}} ){\varsigma _s}r{{( \mathbb{X} )}^2}\gamma _0}} 
 + \frac{{2n{G_1}R( \mathbb{X} ){\vartheta _5}{T^{1 - {\theta _2}}}}}{{( {1 - {\theta _2}} ){\varsigma _s}\gamma _0}} 
+ \frac{{n{G_1}R( \mathbb{X} ){\vartheta _5}{T^{1 - {\theta _2} - {\theta _3}}}}}{{( {1 - {\theta _2} - {\theta _3}} ){\varsigma _s}\gamma _0}} \\
\nonumber
&\;\; + \frac{{2nLr( \mathbb{X} )R( \mathbb{X} ){\vartheta _5}{T^{1 - {\theta _2} - {\theta _3}}}}}{{( {1 - {\theta _2} - {\theta _3}} ){\varsigma _s}\gamma _0}} 
+ \frac{{n{F_2}{G_2}R( \mathbb{X} ){\vartheta _5}{T^{1 - {\theta _3}}}}}{{( {1 - {\theta _3}} ){\varsigma _s}}}
+ \frac{{2n{F_2}{G_2}r( \mathbb{X} ){\vartheta _5}{T^{1 - {\theta _3}}}}}{{( {1 - {\theta _3}} ){\varsigma _s}}} \\
\nonumber
&\;\; 
+ \frac{{2n\sqrt C R( \mathbb{X} ){\vartheta _5}{s_0}\log ( T )}}{{{\varsigma _s}{\alpha _0}\gamma _0}} 
+ \frac{{2n\sqrt C R( \mathbb{X} ){\vartheta _5}{s_0}}}{{{\varsigma _s}{\alpha _0}\gamma _0}}
 + n{G_1}{\vartheta _1} + \frac{{n{G_2}r( \mathbb{X} ){T^{1 - {\theta _3}}}}}{{1 - {\theta _3}}} 
 \\
&\;\; + \frac{{np{F_1}{G_2}{\vartheta _2}{\alpha _0}{T^{1 - {\theta _1} + {\theta _3}}}}}{{( {1 - {\theta _1} + {\theta _3}} )r( \mathbb{X} )}} 
+ n\sqrt C {G_2}{\vartheta _2}{s_0}\log ( T ) + n\sqrt C {G_2}{\vartheta _2}{s_0}.
\label{Theorem1-proof-iii-eq7}
\end{flalign}
From \eqref{Theorem1-proof-iii-eq5} and \eqref{Theorem1-proof-iii-eq7}, we know that \eqref{theorem1-eq4} holds.

\hspace{-3mm}\emph{C. Proof of Theorem~2}

We first analyze local regret at one iteration in the following lemma.
\begin{lemma}
Suppose Assumptions~1--2 and 4--5 hold. For all $i \in [ n ]$, let $\{ {{x_{i,t}}} \}$ be the sequences generated by Algorithm~2, $y$ be an arbitrary point in $\mathbb{X}$, 
and $\hat y = ( {1 - {\xi _t}} )y$, then
\begin{flalign}
\nonumber
&\;\;\;\;\;\frac{1}{n}\sum\limits_{i = 1}^n {{\mathbf{E}}[ {v_{i,t + 1}^T{g_{i,t}}( {{x_{i,t}}} )} ]} 
+ \frac{1}{n}\sum\limits_{i = 1}^n {{\mathbf{E}}[ {\langle {\nabla {f_{i,t}}( {{x_{i,t}}} ),{x_{i,t}} - y} \rangle } ]} \\
\nonumber
& \le \frac{1}{n}\sum\limits_{i = 1}^n {{\mathbf{E}}[ {v_{i,t + 1}^T{g_{i,t}}( y )} ]}  
+ \frac{1}{n}\sum\limits_{i = 1}^n {{{\mathbf{E}}[ {{\Delta _{i,t}}( {\hat y} )} ]}}  
+ \frac{1}{n}{{\mathbf{E}}[ {{{\tilde \Delta }_t}} ]} \\
&\;\; + \frac{1}{n}\sum\limits_{i = 1}^n {{G_2}\big( {R( \mathbb{X} ){\xi _t} + {\delta _t}} \big){\mathbf{E}}[ {\| {{v_{i,t + 1}}} \|} ]} 
+ {G_1}R( \mathbb{X} ){\xi _t} + 2LR( \mathbb{X} ){\delta _t} + 2\sqrt C R( \mathbb{X} )\frac{{{s_t}}}{{{\alpha _t}}}, \label{Lemma10-eq1}
\end{flalign}
where 
\begin{flalign}
\nonumber
{{ \tilde{\Delta} }_t} = \sum\limits_{i = 1}^n {\big( {p{G_1} + p{G_2}\| {{v_{i,t + 1}}} \|} \big)\| {x_{i,t}} - z_{i,t + 1} \|}  
- \sum\limits_{i = 1}^n {\frac{{{{\| {x_{i,t}} - z_{i,t + 1} \|}^2}}}{{2{\alpha _t}}}}.
\end{flalign}
\end{lemma}
\begin{proof}
Similar to the way to get \eqref{Lemma7-eq1}, replacing \eqref{lemma5-eq2} and \eqref{lemma5-eq6} with \eqref{lemma5-eq3} and \eqref{lemma5-eq7}, we know that \eqref{Lemma10-eq1} holds. 
\end{proof}

We then analyze local regret and squared cumulative constraint violation over $T$ iterations in the
following lemma.
\begin{lemma}
Suppose Assumptions~1--2 and 4--5 hold. For all $i \in [ n ]$, let $\{ {{x_{i,t}}} \}$ be the sequences generated by Algorithm~2 
with ${\gamma _t} = {\gamma _0}/{\alpha _t}$ and the constant ${\gamma _0} \in ( {0,1/ {4(p^2+1)G_2^2} } ]$, $y$ be an arbitrary point in $\mathbb{X}$, and $\hat y = ( {1 - {\xi _t}} )y$, 
Then, for any $T \in {\mathbb{N}_ + }$,
\begin{subequations}
\begin{flalign}
\nonumber
&\;\;\;\;\;\frac{1}{n}\sum\limits_{i = 1}^n {\sum\limits_{t = 1}^T {{\mathbf{E}}[ {\langle {\nabla {f_{i,t}}( {{x_{i,t}}} ),{x_{i,t}} - y} \rangle } ]} } \\
& \le  - \frac{1}{n}\sum\limits_{i = 1}^n {\sum\limits_{t = 1}^T {\frac{{\mathbf{E}[ {{{\| {{x_{i,t}} - {z_{i,t + 1}}} \|}^2}} ]}}{{4{\alpha _t}}}} } 
+ \frac{1}{n}\sum\limits_{i = 1}^n {\sum\limits_{t = 1}^T {{{\mathbf{E}}[ {{\Delta _{i,t}}( {\hat y} )} ]} } }  
+ {\bar{\Phi} _T}, \forall y \in {\mathcal{X}_T}, \label{Lemma11-eq1} \\
&\;\;\;\;\;\sum\limits_{i = 1}^n {\sum\limits_{t = 1}^T {\frac{1}{2}} } {\mathbf{E}}\Big[ {{{\| {{{[ {{g_{i,t}}( {{x_{i,t}}} )} ]}_ + }} \|}^2} 
+ \frac{{{{{\| {{x_{i,t}} - {z_{i,t + 1}}} \|}^2}} }}{{2{\gamma _0}}}} \Big] 
\le {\mathbf{E}}[ {{\Lambda _T}( y )} ] + {{\bar \Psi }_T}. \label{Lemma11-eq2}
\end{flalign}
\end{subequations}
where
\begin{flalign}
\nonumber
{\bar{\Phi} _T} &= 2{p^2}G_1^2\sum\limits_{t = 1}^T {{\alpha _t}}  + \frac{{R{{( \mathbb{X} )}^2}}}{4}\sum\limits_{t = 1}^T {\frac{{\xi _t^2}}{{{\alpha _t}}}}  
+ \frac{1}{4}\sum\limits_{t = 1}^T {\frac{{\delta _t^2}}{{{\alpha _t}}}}  + {G_1}R( \mathbb{X} )\sum\limits_{t = 1}^T {{\xi _t}} + 2LR( \mathbb{X} )\sum\limits_{t = 1}^T {{\delta _t}} \\
\nonumber
&\;\; + 2\sqrt C R( \mathbb{X} )\sum\limits_{t = 1}^T {\frac{{{s_t}}}{{{\alpha _t}}}}, \\
\nonumber
{{\bar \Psi }_T} &= \frac{{2nR{{( \mathbb{X} )}^2}}}{{{\gamma _0}}}  
+ 2n{G_1}R( \mathbb{X} )\sum\limits_{t = 1}^T {\frac{1}{{{\gamma _t}}}}  + 2np^2G_1^2{\gamma _0}\sum\limits_{t = 1}^T {\frac{1}{{\gamma _t^2}}} 
 + \frac{{nR{{( \mathbb{X} )}^2}}}{{4{\gamma _0}}}\sum\limits_{t = 1}^T {\xi _t^2}  + \frac{n}{{4{\gamma _0}}}\sum\limits_{t = 1}^T {\delta _t^2} \\
\nonumber
&\;\; + n{G_1}R( \mathbb{X} )\sum\limits_{t = 1}^T {\frac{{{\xi _t}}}{{{\gamma _t}}}}  
+ 2nLR( \mathbb{X} )\sum\limits_{t = 1}^T {\frac{{{\delta _t}}}{{{\gamma _t}}}}  
+ \frac{{2n\sqrt C R( \mathbb{X} )}}{{{\gamma _0}}}\sum\limits_{t = 1}^T {{s_t}}.
\end{flalign}
\end{lemma}
\begin{proof}
(i)
Since ${g_{i,t}}( y ) \le {\mathbf{0}_{{m_i}}}$, $\forall i \in [ n ]$, $\forall t \in {\mathbb{N}_ + }$ when $\forall y \in {\mathcal{X}_T}$, summing \eqref{Lemma10-eq1} over $t \in [ T ]$ gives
\begin{flalign}
\nonumber
&\;\;\;\;\;\frac{1}{n}\sum\limits_{i = 1}^n {\sum\limits_{t = 1}^T {\mathbf{E}}[{\langle {\nabla {f_{i,t}}( {{x_{i,t}}} ),{x_{i,t}} - y} \rangle }] } \\
\nonumber
& \le \frac{1}{n}\sum\limits_{i = 1}^n {\sum\limits_{t = 1}^T {{\mathbf{E}}\Big[ { - v_{i,t + 1}^T{g_{i,t}}( {{x_{i,t}}} ) 
+ \frac{1}{n}{{{\tilde \Delta }_t}} + {G_2}\big( {R( \mathbb{X} ){\xi _t} + {\delta _t}} \big)\| {{v_{i,t + 1}}} \|} \Big]} } \\
& \;+ \frac{1}{n}\sum\limits_{i = 1}^n {\sum\limits_{t = 1}^T {{{\mathbf{E}}[ {{\Delta _{i,t}}( {\hat y} )} ]}} }  
+ {G_1}R( \mathbb{X} )\sum\limits_{t = 1}^T {{\xi _t}} 
 + 2LR( \mathbb{X} )\sum\limits_{t = 1}^T {{\delta _t}}  + 2\sqrt C R( \mathbb{X} )\sum\limits_{t = 1}^T {\frac{{{s_t}}}{{{\alpha _t}}}}. \label{Lemma11-proof-eq1}
\end{flalign}
We have
\begin{flalign}
\nonumber
&\;\;\;\;\; \sum\limits_{i = 1}^n {\sum\limits_{t = 1}^T {( {p{G_1} + p{G_2}\| {{v_{i,t + 1}}} \|} ){\mathbf{E}}[\| {x_{i,t}} - z_{i,t + 1} \|]} } \\
& \le \sum\limits_{i = 1}^n {\sum\limits_{t = 1}^T {({2{{p^2{G_1^2} }}{\alpha _t}} 
+ {2p^2G_2^2{\alpha _t}{{\| {{v_{i,t + 1}}} \|}^2}} + \frac{{ {\mathbf{E}}[{{\| {x_{i,t}} - z_{i,t + 1} \|}^2}]}}{{4{\alpha _t}}})} }. \label{Lemma11-proof-eq2}
\end{flalign}
We have
\begin{flalign}
\sum\limits_{i = 1}^n {\sum\limits_{t = 1}^T {{G_2}\big( {R( \mathbb{X} ){\xi _t} + {\delta _t}} \big)\| {{v_{i,t + 1}}} \|} }  
\le \sum\limits_{i = 1}^n {\sum\limits_{t = 1}^T {\big( {{2G_2^2{\alpha _t}{{\| {{v_{i,t + 1}}} \|}^2}} 
+ \frac{{ R{{( \mathbb{X} )}^2}\xi _t^2}}{{4{\alpha _t}}} + \frac{{ \delta _t^2}}{{4{\alpha _t}}}} \big)} }. \label{Lemma11-proof-eq3}
\end{flalign}
From \eqref{Algorithm2-eq3}, for all $\forall t \in {\mathbb{N}_ + }$, we have
\begin{flalign}
\| {{v_{i,t + 1}}} \| = {\gamma _t}\| {{{[ {{g_{i,t}}( {{x_{i,t}}} )} ]}_ + }} \|. \label{Lemma11-proof-eq4}
\end{flalign}
From \eqref{Lemma11-proof-eq4}, \eqref{Algorithm2-eq3} and the fact that ${\varphi ^T}{[ \varphi  ]_ + } = {\| {{{[ \varphi  ]}_ + }} \|^2}$ for any vector $\varphi $, we have
\begin{flalign}
\nonumber
2( {{p^2} + 1} )G_2^2{\alpha _t}{\| {{v_{i,t + 1}}} \|^2} - v_{i,t + 1}^T{g_{i,t}}( {{x_{i,t}}} ) 
&= 2( {{p^2} + 1} )G_2^2{\alpha _t}\gamma _2^2{\| {{{[ {{g_{i,t}}( {{x_{i,t}}} )} ]}_ + }} \|^2} - {\gamma _t}{\| {{{[ {{g_{i,t}}( {{x_{i,t}}} )} ]}_ + }} \|^2}\\
& = \big( {2( {{p^2} + 1} )G_2^2{\gamma _0} - 1} \big){\gamma _t}{\| {{{[ {{g_{i,t}}( {{x_{i,t}}} )} ]}_ + }} \|^2}  <  0, \label{Lemma11-proof-eq5}
\end{flalign}
where the last equality holds due to ${\gamma _t} = {\gamma _0}/{\alpha _t}$, and the inequality holds due to ${\gamma _0} < 1 /{2(p^2+1)G_2^2}$.

From \eqref{Lemma11-proof-eq1}--\eqref{Lemma11-proof-eq3} and \eqref{Lemma11-proof-eq5}, we know that \eqref{Lemma11-eq1} holds.

(ii)
Dividing \eqref{Lemma10-eq1} by ${\gamma _t}$, using \eqref{Lemma8-proof-ii-eq1}, and summing over $t \in [ T ]$ yields
\begin{flalign}
\nonumber
&\;\;\;\;\;\sum\limits_{i = 1}^n {\sum\limits_{t = 1}^T {\frac{{{\mathbf{E}}[ {v_{i,t + 1}^T{g_{i,t}}( {{x_{i,t}}} )} ]}}{{{\gamma _t}}}} } \\
\nonumber
& \le {\mathbf{E}}[ {{\Lambda _T}( y )} ] + {2n{G_1}R( \mathbb{X} )}\sum\limits_{t = 1}^T {\frac{1}{{{\gamma _t}}}}  
 + \sum\limits_{i = 1}^n {\sum\limits_{t = 1}^T {\frac{{{{\mathbf{E}}[ {{\Delta _{i,t}}( {\hat y} )} ]} }}{{{\gamma _t}}}} }  
 + \sum\limits_{t = 1}^T {\frac{{{{\mathbf{E}}[ {{{\tilde \Delta }_t}} ]}}}{{{\gamma _t}}}} + \frac{{2n\sqrt C R( \mathbb{X} )}}{{{\gamma _0}}}\sum\limits_{t = 1}^T {{s_t}} \\
&\;\; + \sum\limits_{i = 1}^n {\sum\limits_{t = 1}^T {\frac{{{G_2}\big( {R( \mathbb{X} ){\xi _t} + {\delta _t}} \big){\mathbf{E}}[ {\| {{v_{i,t + 1}}} \|} ]}}{{{\gamma _t}}}} } 
      + n{G_1}R( \mathbb{X} )\sum\limits_{t = 1}^T {\frac{{{\xi _t}}}{{{\gamma _t}}}}  + 2nLR( \mathbb{X} )\sum\limits_{t = 1}^T {\frac{{{\delta _t}}}{{{\gamma _t}}}}. \label{Lemma11-proof-ii-eq1}
\end{flalign}
We have
\begin{flalign}
\nonumber
&\;\;\;\;\; \sum\limits_{i = 1}^n {\sum\limits_{t = 1}^T {\frac{{( {p{G_1} + p{G_2}\| {{v_{i,t + 1}}} \|} ){\mathbf{E}}[\| {x_{i,t}} - z_{i,t + 1} \|]}}{{{\gamma _t}}}} } \\
&\;\; \le \sum\limits_{i = 1}^n {\sum\limits_{t = 1}^T {\big( {\frac{{2{ p^2{G_1^2} }{\gamma _0}}}{{\gamma _t^2}} + \frac{{2p^2G_2^2{\gamma _0}{{\| {{v_{i,t + 1}}} \|}^2}}}{{\gamma _t^2}} 
+ \frac{{ {\mathbf{E}}[{{\| {x_{i,t}} - z_{i,t + 1} \|}^2}]}}{{4{\gamma _0}}}} \big)} }. \label{Lemma11-ii-proof-eq2}
\end{flalign}
We have
\begin{flalign}
\sum\limits_{i = 1}^n {\sum\limits_{t = 1}^T {\frac{{{G_2}\big( {R( \mathbb{X} ){\xi _t} + {\delta _t}} \big)\| {{v_{i,t + 1}}} \|}}{{{\gamma _t}}}} } 
\le \sum\limits_{i = 1}^n {\sum\limits_{t = 1}^T {\big( {\frac{{2G_2^2{\gamma _0}{{\| {{v_{i,t + 1}}} \|}^2}}}{{\gamma _t^2}} 
+ \frac{{R{{( \mathbb{X} )}^2}\xi _t^2}}{{4{\gamma _0}}} + \frac{{\delta _t^2}}{{4{\gamma _0}}}} \big)} }. \label{Lemma11-ii-proof-eq3}
\end{flalign}
From ${\gamma _t} = {\gamma _0}/{\alpha _t}$ and \eqref{ass1-eq1}, we have
\begin{flalign}
\sum\limits_{t = 1}^T {\frac{{{\Delta _{i,t}}( {\hat y} )}}{{{\gamma _t}}} = } 
\sum\limits_{t = 1}^T {\frac{1}{{2{\gamma _0}}}} \big( {{{\| {\hat y - {x_{i,t}}} \|}^2} - {{\| {\hat y - {x_{i,t + 1}}} \|}^2}} \big) 
\le \frac{{{{\| {\hat y - {x_{i,1}}} \|}^2}}}{{2{\gamma _0}}} \le \frac{{2R{{( \mathbb{X} )}^2}}}{{{\gamma _0}}}. \label{Lemma11-ii-proof-eq4}
\end{flalign}
From \eqref{Lemma11-proof-eq4}, \eqref{Algorithm2-eq3} and the fact that ${\varphi ^T}{[ \varphi  ]_ + } = {\| {{{[ \varphi  ]}_ + }} \|^2}$ for any vector $\varphi $, we have
\begin{flalign}
\nonumber
\frac{{2( {{p^2} + 1} )G_2^2{\gamma _0}{{\| {{v_{i,t + 1}}} \|}^2}}}{{\gamma _t^2}} - \frac{{v_{i,t + 1}^T{g_{i,t}}( {{x_{i,t}}} )}}{{2{\gamma _t}}} 
&= 2( {{p^2} + 1} )G_2^2{\gamma _0}{\| {{{[ {{g_{i,t}}( {{x_{i,t}}} )} ]}_ + }} \|^2} - \frac{{{{\| {{{[ {{g_{i,t}}( {{x_{i,t}}} )} ]}_ + }} \|}^2}}}{2} \\
&\le \Big( {2( {{p^2} + 1} )G_2^2{\gamma _0} - \frac{1}{2}} \Big)\| {{{[ {{g_{i,t}}( {{x_{i,t}}} )} ]}_ + }} \|. \label{Lemma11-ii-proof-eq5}
\end{flalign}

From \eqref{Lemma11-proof-ii-eq1}--\eqref{Lemma11-ii-proof-eq5} and ${\gamma _0} \le 1/4( {{p^2} + 1} )G_2^2$,  
and using \eqref{Algorithm2-eq3} and the fact that ${\varphi ^T}{[ \varphi  ]_ + } = {\| {{{[ \varphi  ]}_ + }} \|^2}$ for any vector $\varphi $, we know that \eqref{Lemma11-eq2} holds.
\end{proof}

We next analyze network regret and cumulative constraint violation over $T$ iterations in the following lemma.
\begin{lemma}
Under the same condition as stated in Lemma~9, and supposing that Assumption~3 holds, for any $T \in {\mathbb{N}_ + }$, it holds that
\begin{subequations}
\begin{flalign}
&\frac{1}{n}\sum\limits_{i = 1}^n {\sum\limits_{t = 1}^T {{\mathbf{E}}[ {\langle {\nabla {f_t}( {{x_{i,t}}} ),{x_{i,t}} - y} \rangle } ]} }
\le \frac{1}{n}\sum\limits_{i = 1}^n {\sum\limits_{t = 1}^T {{{\mathbf{E}}[ {{\Delta _{i,t}}( {\hat y} )} ]}} }  + {{\hat \Phi }_T}, \forall y \in {\mathcal{X}_T},\label{Lemma12-eq1} \\
&\frac{1}{n}\sum\limits_{i = 1}^n {\sum\limits_{t = 1}^T {{\mathbf{E}}[ {\| {{{[ {{g_t}( {{x_{i,t}}} )} ]}_ + }} \|} ]} }  
\le \sqrt {T{{\hat \Psi }_T}}, \label{Lemma12-eq2} \\
&\frac{1}{n}\sum\limits_{i = 1}^n {\sum\limits_{t = 1}^T {{\mathbf{E}}[ {\| {{{[ {{g_t}( {{x_{i,t}}} )} ]}_ + }} \|} ]} }  
\le {\vartheta _7}\sum\limits_{i = 1}^n {\sum\limits_{t = 1}^T {{\mathbf{E}}\big[ {\| {{{[ {{g_{i,t}}( {{x_{i,t}}} )} ]}_ + }} \|} \big]} } + {\hat{\Xi} _T}, \label{Lemma12-eq3}
\end{flalign}
\end{subequations}
where
\begin{flalign}
\nonumber
&{\hat \Phi _T} = {\bar \Phi _T} + \big( {{G_1} + 2LR( \mathbb{X} )} \big){\vartheta _1} + \big( {{G_1} + 2LR( \mathbb{X} )} \big){\vartheta _2}\sum\limits_{t = 1}^T {{s_t}} 
 + {\big( {{G_1} + 2LR( \mathbb{X} )} \big)^2}\vartheta _2^2\sum\limits_{t = 1}^T {{\alpha _t}}, \\
\nonumber
&{\vartheta _6} = \frac{{4\max \{ {1,2G_2^2{\vartheta _4}} \}}}{{\min \{ {1,\frac{1}{{2{\gamma _0}}}} \}}}, 
{{\hat \Psi }_T} = {\vartheta _6}{{\bar \Psi }_T} + 2G_2^2{\vartheta _3} + 4nCG_2^2{\vartheta _4}\sum\limits_{t = 1}^T {s_t^2},\\
\nonumber
&{\vartheta _7} = 1 + pG_2^2{\vartheta _2}{\gamma _0}, 
{{\hat \Xi }_T} = n{G_1}{\vartheta _1} + np{G_1}{G_2}{\vartheta _2}\sum\limits_{t = 1}^T {{\alpha _t} + n\sqrt C } {G_2}{\vartheta _2}\sum\limits_{t = 1}^T {{s_t}}.
\end{flalign}
\end{lemma}
\begin{proof}
(i)
Similar to the way to get \eqref{Lemma9-eq1}, replacing \eqref{Lemma8-eq1} with \eqref{Lemma11-eq1}, we know that \eqref{Lemma12-eq1} holds. 

(ii)
We have
\begin{flalign}
\nonumber
&\;\;\;\;\;\sum\limits_{i = 1}^n {\sum\limits_{t = 1}^T {{{\mathbf{E}}[ {{{\| {\varepsilon _{i,t}^z} \|}^2}} ]} } } \\
& \le 2\sum\limits_{i = 1}^n {\sum\limits_{t = 1}^T {{{\mathbf{E}}[ {{{\| {{x_{i,t}} - {z_{i,t + 1}}} \|}^2}} ]}} } 
 + 2\sum\limits_{i = 1}^n {\sum\limits_{t = 1}^T {{{\mathbf{E}}[ {{{\| {{z_{i,t + 1}} - {{\hat z}_{i,t + 1}}} \|}^2}} ]}} }.
\label{Lemma12-proof-ii-eq3}
\end{flalign}

From \eqref{Lemma9-proof-ii-eq1}--\eqref{Lemma9-proof-ii-eq5}, \eqref{Lemma12-proof-ii-eq3}, \eqref{Lemma9-proof-ii-eq9}, we have
\begin{flalign}
\nonumber
\frac{1}{n}\sum\limits_{i = 1}^n {\sum\limits_{t = 1}^T {{\mathbf{E}}\big[ {{{\| {{{[ {{g_t}( {{x_{i,t}}} )} ]}_ + }} \|}^2}} \big]} } 
&\le 2\sum\limits_{i = 1}^n {\sum\limits_{t = 1}^T {{\mathbf{E}}\big[ {{{\| {{{[ {{g_{i,t}}( {{x_{i,t}}} )} ]}_ + }} \|}^2}} \big]} } 
 + 4G_2^2{\vartheta _4}\sum\limits_{i = 1}^n {\sum\limits_{t = 1}^T {{{\mathbf{E}}[ {{{\| {{x_{i,t}} - {z_{i,t + 1}}} \|}^2}} ]} } } \\
&\;\; + 2G_2^2{\vartheta _3} + 4nCG_2^2{\vartheta _4}\sum\limits_{t = 1}^T {s_t^2}.
\label{Lemma12-proof-ii-eq4}
\end{flalign}
From \eqref{Lemma12-proof-ii-eq4}, \eqref{Lemma9-proof-ii-eq10} and \eqref{Lemma11-eq2}, we have
\begin{flalign}
\frac{1}{n}\sum\limits_{i = 1}^n {\sum\limits_{t = 1}^T {{\mathbf{E}}\big[ {{{\| {{{[ {{g_t}( {{x_{i,t}}} )} ]}_ + }} \|}^2}} \big]} } 
\le {\vartheta _6}{{\bar \Psi }_T} + 2G_2^2{\vartheta _3} + 4nCG_2^2{\vartheta _4}\sum\limits_{t = 1}^T {s_t^2}.
\label{Lemma12-proof-ii-eq5}
\end{flalign}

Combining \eqref{Lemma12-proof-ii-eq5} and \eqref{Lemma9-proof-ii-eq11} yields \eqref{Lemma12-eq2}.

(iii)
From \eqref{Lemma9-proof-iii-eq1}--\eqref{Lemma9-proof-iii-eq2}, we have
\begin{flalign}
\nonumber
\frac{1}{n}\sum\limits_{i = 1}^n {\sum\limits_{t = 1}^T {{\mathbf{E}}\big[ {\| {{{[ {{g_t}( {{x_{i,t}}} )} ]}_ + }} \|} \big]} } 
&\le \sum\limits_{i = 1}^n { {\sum\limits_{t = 1}^T {{\mathbf{E}}\big[ {\| {{{[ {{g_{i,t}}( {{x_{i,t}}} )} ]}_ + }} \|} \big]} } }  + n{G_2}{\vartheta _1}\\
&\;\;  
+ {G_2}{\vartheta _2}\sum\limits_{i = 1}^n {\sum\limits_{t = 1}^T {{\mathbf{E}}[ {\| {{z_{i,t + 1}} - {x_{i,t}}} \|} ]} }  + n\sqrt C {G_2}{\vartheta _2}\sum\limits_{t = 1}^T {{s_t}}.
\label{Lemma12-proof-iii-eq1}
\end{flalign}
By applying \eqref{Lemma3-eq2} to the update \eqref{Algorithm2-eq5}, we have
\begin{flalign}
\nonumber
\| {{z_{i,t + 1}} - {x_{i,t}}} \| &\le {\alpha _t}\| {{b_{i,t + 1}}} \| \\
\nonumber
& = {\alpha _t}\big\| {{{\hat \nabla }_2}{f_{i,t}}( {{x_{i,t}}} ) + { {{{\hat \nabla }_2}{g_{i,t}}( {{x_{i,t}}} )} }{v_{i,t + 1}}} \big\|\\
\nonumber
& \le {\alpha _t}\| {{{\hat \nabla }_2}{f_{i,t}}( {{x_{i,t}}} )} \| + {\alpha _t}\| {{{\hat \nabla }_2}{g_{i,t}}( {{x_{i,t}}} )} \|\| {{v_{i,t + 1}}} \|\\
& \le p{G_1}{\alpha _t} + p{G_2}{\gamma _0}\| {{{[ {{g_{i,t}}( {{x_{i,t}}} )} ]}_ + }} \|,
\label{Lemma12-proof-iii-eq2}
\end{flalign}
where the equality holds due to \eqref{Algorithm2-eq4}, the last inequality holds due to \eqref{lemma5-eq3}, \eqref{lemma5-eq7} and ${\gamma _t} = {\gamma _0}/{\alpha _t}$.

Combining \eqref{Lemma12-proof-iii-eq1} and \eqref{Lemma12-proof-iii-eq2} yields \eqref{Lemma12-eq3}.
\end{proof}
In the following, we prove that \eqref{theorem2-eq2}--\eqref{theorem2-eq4} in Theorem~2 hold in three separate steps.

(i)
From \eqref{theorem2-eq1} and \eqref{Theorem1-proof-i-eq1}, for any $T \in \mathbb{N}_ +$, we have
\begin{flalign}
&\sum\limits_{t = 1}^T {{s_t}} \le {s_0}\sum\limits_{t = 1}^T {\frac{1}{t}}  \le {s_0}\Big( {\int_1^T {\frac{1}{t}dt}  + 1} \Big) \le  {s_0}{\log ( T ) + {s_0}}, \label{Theorem2-proof-i-eq4}\\
&\sum\limits_{t = 1}^T {\frac{{{s_t}}}{{{\alpha _t}}}}  \le \frac{{{s_0}}}{{{\alpha _0}}}\sum\limits_{t = 1}^T {\frac{1}{{{t^{1 - {\theta _1}}}}}}  
\le \frac{{{{s_0}T^{ {\theta _1}}}}}{{{\theta _1} }{\alpha _0}}.
\label{Theorem2-proof-i-eq5}
\end{flalign}
Combining \eqref{Lemma12-eq1}, \eqref{Theorem2-proof-i-eq4}--\eqref{Theorem2-proof-i-eq5}, and \eqref{Theorem1-proof-i-eq4}, 
from \eqref{theorem2-eq1} with ${\theta _2} \in ( {{\theta _1},1} )$, \eqref{Theorem1-proof-i-eq1} and the arbitrariness of ${y} \in {\mathcal{X}_T}$, we have
\begin{flalign}
\nonumber
&\;\;\;\;\;\mathbf{E}[{{\rm{Net}\mbox{-}\rm{Reg}}( T )}] \\
\nonumber
& \le \frac{{\big( {8{p^2}G_1^2 + r{{( \mathbb{X} )}^2} + R{{( \mathbb{X} )}^2} 
+ 4{G_1}R( \mathbb{X} ) + 8Lr( \mathbb{X} )R( \mathbb{X} ) + 4{\big( {{G_1} + 2LR( \mathbb{X} )} \big)^2}\vartheta _2^2} \big){\alpha _0}{T^{1 - {\theta _1}}}}}{{4( {1 - {\theta _1}} )}}\\
\nonumber
&\;\; + \big( {{G_1} + 2LR( \mathbb{X} )} \big){\vartheta _1} 
+ \big( {{G_1} + 2LR( \mathbb{X} )} \big){\vartheta _2}{s_0} + \big( {{G_1} + 2LR( \mathbb{X} )} \big){\vartheta _2}{s_0}\log ( T ) \\
&\;\; + \frac{{2R( \mathbb{X} )\big( {\sqrt C {s_0} + {\theta _1}R( \mathbb{X} )} \big){T^{{\theta _1}}}}}{{{\theta _1}{\alpha _0}}}.
\label{Theorem2-proof-i-eq6}
\end{flalign}

From \eqref{Theorem2-proof-i-eq6}, we know that \eqref{theorem2-eq2} holds.

(ii)
From \eqref{theorem2-eq1}, for any $T \in \mathbb{N}_ +$, we have
\begin{flalign}
&\sum\limits_{t = 1}^T {\frac{1}{{\gamma _t^2}}}  = \frac{1}{{\gamma _0^2}}\sum\limits_{t = 1}^T {\alpha _t^2}  \le \frac{{{\alpha _0}}}{{\gamma _0^2}}\sum\limits_{t = 1}^T {{\alpha _t}}, \label{Theorem2-proof-ii-eq1} \\
&\sum\limits_{t = 1}^T {\xi _t^2}  = \sum\limits_{t = 1}^T {\alpha _t^2}  \le {\alpha _0}\sum\limits_{t = 1}^T {{\alpha _t}}, \label{Theorem2-proof-ii-eq2}\\
&\sum\limits_{t = 1}^T {\delta _t^2}  = r{( \mathbb{X} )^2}\sum\limits_{t = 1}^T {\alpha _t^2}  \le r{( \mathbb{X} )^2}{\alpha _0}\sum\limits_{t = 1}^T {{\alpha _t}}, \label{Theorem2-proof-ii-eq3}\\
&\sum\limits_{t = 1}^T {\frac{{{\xi _t}}}{{{\gamma _t}}}}  = \frac{1}{{{\gamma _0}}}\sum\limits_{t = 1}^T {\alpha _t^2}  \le \frac{{{\alpha _0}}}{{{\gamma _0}}}\sum\limits_{t = 1}^T {{\alpha _t}},
\label{Theorem2-proof-ii-eq4} \\
&\sum\limits_{t = 1}^T {\frac{{{\delta _t}}}{{{\gamma _t}}}}  
= \frac{{r( \mathbb{X} )}}{{{\gamma _0}}}\sum\limits_{t = 1}^T {\alpha _t^2}  \le \frac{{r( \mathbb{X} ){\alpha _0}}}{{{\gamma _0}}}\sum\limits_{t = 1}^T {{\alpha _t}},
\label{Theorem2-proof-ii-eq5} \\
&\sum\limits_{t = 1}^T {s_t^2}  \le {s_0}\sum\limits_{t = 1}^T {{s_t}}. \label{Theorem2-proof-ii-eq6}
\end{flalign}
Combining \eqref{Lemma12-eq2} and \eqref{Theorem2-proof-ii-eq1}--\eqref{Theorem2-proof-ii-eq6}, 
from \eqref{theorem2-eq1}, \eqref{Theorem2-proof-i-eq4} and \eqref{Theorem1-proof-i-eq1}, we have
\begin{flalign}
\nonumber
&\;\;\;\;\;{( {\frac{1}{n}\sum\limits_{i = 1}^n {\sum\limits_{t = 1}^T {{\mathbf{E}}[ {\| {{{[ {{g_t}( {{x_{i,t}}} )} ]}_ + }} \|} ]} } } )^2} \\
\nonumber
& \le \frac{2\big( {nR{{( \mathbb{X} )}^2}{\vartheta _6} + G_2^2{\vartheta _3} + n\sqrt C R( \mathbb{X} ){\vartheta _6}{s_0} + 2nCG_2^2{\vartheta _4}s_0^2{\gamma _0}} \big)T}{{{\gamma _0}}}
\\
\nonumber
&\;\; + \frac{{n{\vartheta _6}{\alpha _0}\big( {8{G_1}R( \mathbb{X} ) + 8{p^2}G_1^2{\alpha _0} + r{{( \mathbb{X} )}^2}{\alpha _0} 
+ R{{( \mathbb{X} )}^2}{\alpha _0} + 4{G_1}R( \mathbb{X} ){\alpha _0}} \big){T^{2 - {\theta _1}}}}}{{4( {1 - {\theta _1}} ){\gamma _0}}}\\
&\;\;  + \frac{{2nLr( \mathbb{X} )R( \mathbb{X} ){\vartheta _6}\alpha _0^2{T^{2 - {\theta _1}}}}}{{( {1 - {\theta _1}} ){\gamma _0}}}
+ \frac{{2n\sqrt C {s_0}\big( {R( \mathbb{X} ){\vartheta _6} 
 + 2\sqrt C G_2^2{\vartheta _4}{s_0}{\gamma _0}} \big)T\log ( T )}}{{{\gamma _0}}}.
\label{Theorem2-proof-ii-eq8}
\end{flalign}
From \eqref{Theorem2-proof-ii-eq8}, we know that \eqref{theorem2-eq3} holds.

(iii)
Similar to the way to get \eqref{Theorem1-proof-iii-eq1}, we have
\begin{flalign}
{\mathbf{E}}[ {{\Lambda _T}( {{x_s}} )} ] 
\le  - {\varsigma _s}\sum\limits_{i = 1}^n {\sum\limits_{t = 1}^T {{\mathbf{E}}\big[ {\| {{{[ {{g_{i,t}}( {{x_{i,t}}} )} ]}_ + }} \|} \big]} }.
\label{Theorem2-proof-iii-eq1}
\end{flalign}
Selecting $y = {x_s}$ in \eqref{Lemma11-eq2}, from \eqref{Theorem2-proof-iii-eq1}, we have
\begin{flalign}
{\varsigma _s}\sum\limits_{i = 1}^n {\sum\limits_{t = 1}^T {\mathbf{E}}[{\| {{{[ {{g_{i,t}}( {{x_{i,t}}} )} ]}_ + }} \|}] }  \le {{\bar \Psi }_T}.
\label{Theorem2-proof-iii-eq2}
\end{flalign}
Combining \eqref{Lemma12-eq3}, \eqref{Theorem2-proof-ii-eq1}--\eqref{Theorem2-proof-ii-eq5}, \eqref{Theorem2-proof-iii-eq2}, and \eqref{Theorem2-proof-i-eq4}, 
from \eqref{theorem2-eq1} with ${\theta _2} \ge 1$ and \eqref{Theorem1-proof-i-eq1}, we have
\begin{flalign}
\nonumber
&\;\;\;\;\;\mathbf{E}[{{\rm{Net} \mbox{-} \rm{CCV}}( T )}] \\
\nonumber
& \le n{G_1}{\vartheta _1} +  
+ \frac{{n\sqrt C {s_0}\big( {2R( \mathbb{X} ){\vartheta _7} + {\varsigma _s}{G_2}{\vartheta _2}{\gamma _0}} \big)\log ( T )}}{{{\varsigma _s}{\gamma _0}}}
+ \frac{{n{G_1}{\alpha _0}\big( {2R( \mathbb{X} ){\vartheta _7} 
+ p{\varsigma _s}{G_2}{\vartheta _2}{\gamma _0}} \big){T^{1 - {\theta _1}}}}}{{( {1 - {\theta _1}} ){\varsigma _s}{\gamma _0}}}\\
\nonumber
&\;\; + \frac{{2nR{{( \mathbb{X} )}^2}{\vartheta _7}}}{{{\varsigma _s}{\gamma _0}}} + \frac{{n{\vartheta _7}\alpha _0^2\big( {8{p^2}G_1^2 + r{{( \mathbb{X} )}^2} + R{{( \mathbb{X} )}^2} 
+ 4{G_1}R( \mathbb{X} ) + 8Lr( \mathbb{X} )R( \mathbb{X} )} \big){T^{1 - {\theta _1}}}}}{{4( {1 - {\theta _1}} ){\varsigma _s}{\gamma _0}}} \\
&\;\; + \frac{{n\sqrt C {s_0}\big( {2R( \mathbb{X} ){\vartheta _7} + {\varsigma _s}{G_2}{\vartheta _2}{\gamma _0}} \big)}}{{{\varsigma _s}{\gamma _0}}}.
\label{Theorem2-proof-iii-eq4}
\end{flalign}
From \eqref{Theorem2-proof-iii-eq4}, we know that \eqref{theorem2-eq4} holds.

\bibliographystyle{IEEEtran}
\bibliography{reference_online}

\end{document}